%% file: master.tex
\documentclass[a4paper,fleqn,12pt,twoside,draft]{article}

\input{layout_commands}
\input{math_commands}
\input{my_commands}
\title{\titlefamily\Huge Foliations of asymptotically flat manifolds
  by surfaces of Willmore type}
\author{
  \authname{Tobias Lamm
  \thanks{Partially supported by a PIMS Postdoctoral Fellowship.}}
  \authaddress
  {
    Department of Mathematics,
    University of British Columbia,
    1984 Mathematics Road,
    Vancouver, BC V6T 1Z2,
    Canada
  }
  \and
  \authname{Jan Metzger
  \thanks{Partially supported by a Feodor-Lynen fellowship of the Alexander von
    Humboldt Foundation.}}
  \authaddress
  {
    Albert-Einstein-Institut,
    Am M\"uhlenberg 1,
    14476 Potsdam,
    Germany
  }
\and
  \authname{Felix Schulze
  \thanks{Partially supported by a Feodor-Lynen fellowship of the Alexander von
    Humboldt Foundation.}}
  \authaddress
  {
    Freie Universit\"at Berlin,
    Institut f\"ur Mathematik und Informatik,\\
    Arnimallee 3,
    14195 Berlin,
    Germany.
  }
}
\date{}

\begin{document}
\hyphenation{}
\pagestyle{footnumber}
\maketitle
\thispagestyle{footnumber}
\begin{abst}%
  The goal of this paper is to establish the existence of a foliation
  of the asymptotic region of an asymptotically flat manifold with
  nonzero mass by surfaces which are critical points of the Willmore
  functional subject to an area constraint. Equivalently these
  surfaces are critical points of the Geroch-Hawking mass. Thus our
  result has applications in the theory of General Relativity.
\end{abst}
\input{intro}
%
\input{prelim}
%
\input{variation}
%
\input{integral}
%
\input{improved}
%
\input{position}
%
\input{linearization}
%
\input{foliation}
\appendix
\input{appendix}
%
\bibliographystyle{abbrv}
\bibliography{../../extern/references}
\end{document}

%% file: math_commands.tex
\usepackage{amsmath}
\usepackage{amssymb}

\newcounter{mynumcounter}
	       {%
		 \begin{list}{(\roman{mynumcounter})\hspace*{\fill}}%
		   {
		     \setlength{\topsep}{0cm}
		     \setlength{\partopsep}{0cm}
		     \setlength{\itemsep}{0ex}
		     \setlength{\parsep}{0cm}
		     \setlength{\leftmargin}{0cm}
		     \setlength{\itemindent}{8mm}		   
		     \setlength{\labelsep}{3mm}
		     \setlength{\labelwidth}{5mm}
		     \usecounter{mynumcounter}
		   }%
	       }
	       {\end{list}}

\newcommand{\eps}{\varepsilon}
\newcommand{\del}{\partial}

\newcommand{\dd}[2]{\frac{\del #1}{\del #2}}
\newcommand{\ddeval}[3]{\left.\dd{#1}{#2}\right|_{#3}}

\newcommand{\DD}[2]{\frac{d #1}{d #2}}
\newcommand{\DDeval}[3]{\left.\DD{#1}{#2}\right|_{#3}}

\newcommand{\Ric}{\operatorname{Ric}}

\newcommand{\tr}{\operatorname{tr}}
\renewcommand{\div}{\operatorname{div}}

\newcommand{\IR}{\mathbf{R}}

\newcommand{\CL}{\mathcal{L}}

\newcommand{\CU}{\mathcal{U}}
\newcommand{\CV}{\mathcal{V}}
\newcommand{\CW}{\mathcal{W}}
\newcommand{\CX}{\mathcal{X}}

\newcommand{\ga}{\gamma}

\newcommand{\Id}{\operatorname{Id}}
\newcommand{\id}{\operatorname{id}}

\newcommand{\rmd}{\mathrm{d}}
\newcommand{\rmin}{r_\mathrm{min}}

\newcommand{\supp}{\mathrm{supp}}

\newcommand{\dmu}{\,\rmd\mu}

\newcommand{\ra}{\rangle}
\newcommand{\la}{\langle}

\newcommand{\Connection}[1]{\smash{\sideset{^{#1}}{}{\mathop\nabla\nolimits}}}

\newcommand{\nabSig}{\!\Connection{\Sigma}}
\newcommand{\nabM}{\!\Connection{M}}

\newcommand{\Divergence}[1]{\smash{\sideset{^{#1}}{}{\mathop\mathrm{div}\nolimits}}}

\newcommand{\divSig}{\!\Divergence{\Sigma}}

\newcommand{\Riemann}[1]{\smash{\sideset{^{#1}}{}{\mathop\mathrm{Rm}\nolimits}}}

\newcommand{\RiemM}{\!\Riemann{M}}
\newcommand{\RiemSig}{\!\Riemann{\Sigma}}

\newcommand{\Ricci}[1]{\smash{\sideset{^{#1}}{}{\mathop\mathrm{Rc}\nolimits}}}

\newcommand{\RicM}{\!\Ricci{M}}

\newcommand{\Scalarcurv}[1]{\smash{\sideset{^{#1}}{}{\mathop\mathrm{Sc}\nolimits}}}

\newcommand{\ScalM}{\!\Scalarcurv{M}}
\newcommand{\ScalSig}{\!\Scalarcurv{\Sigma}}

\newcommand{\Acirc}{\hspace*{4pt}\raisebox{8.5pt}{\makebox[-4pt][l]{$\scriptstyle\circ$}}A}

\newcommand{\Acircbar}%
{\hspace*{4pt}\raisebox{8.5pt}{\makebox[-4pt][l]{$\scriptstyle\circ$}}\bar A}
\newcommand{\Kcircbar}%
{\hspace*{4pt}\raisebox{8.5pt}{\makebox[-4pt][l]{$\scriptstyle\circ$}}\bar K}

\newcommand{\half}{\tfrac{1}{2}}


%% file: my_commands.tex
\usepackage{color}
\newcommand{\annotation}[1]{}
\newcommand{\noteFelix}[1]{}
\newcommand{\noteTobias}[1]{}
\newcommand{\noteJan}[1]{}

\newcommand{\CD}{\mathcal{D}}
\renewcommand{\IR}{\mathbb{R}}
\renewcommand{\Ric}{\mathrm{Rc}}

%% file: intro.tex
\section*{Introduction}
In this paper we study foliations of asymptotically flat manifolds by
surfaces of Willmore type. This means that we are interested in
constructing embedded spheres $\Sigma$ in a three dimensional
Riemannian manifold $(M,g)$ which satisfy the equation
\begin{equation}
  \label{eq:1}
  -\Delta H -  H |\Acirc|^2 -  \RicM(\nu,\nu)H = \lambda H.
\end{equation}
Here $H$ is the mean curvature of $\Sigma$, $\Acirc$ is the traceless
part of the second fundamental form $A$ of $\Sigma$ in $M$, that is
$\Acirc= A - \half H \gamma$, and $\ga$ is the induced metric on
$\Sigma$. Moreover $\RicM$ is the Ricci curvature of $M$ and $\Delta$
the Laplace-Beltrami operator on $\Sigma$.

Equation~\eqref{eq:1} is the Euler-Lagrange equation of the
functional
\begin{equation}
  \label{eq:71}
  \CW(\Sigma) = \frac{1}{2} \int_\Sigma H^2 \dmu
\end{equation}
subject to the constraint that $|\Sigma|$ be fixed. Then $\lambda$
becomes the Lagrange parameter.

In mathematics this functional is known as the Willmore functional,
at least in flat space, whereas for curved ambient manifolds the
literature \cite{Weiner:1978} also considers the functional
\begin{equation*}
  \CU(\Sigma) = \int_\Sigma |\Acirc|^2\dmu.
\end{equation*}
In flat space these two functionals only differ by a topological
constant. However, the second functional is conformally invariant and
thus translation invariant in all conformally flat manifolds. Since
our model space, the spatial Schwarzschild metric $g^S_m = \phi_m^4
g^e$, with $\phi=1 + \frac{m}{2r}$, $g^e$ the Euclidean metric and
$m>0$ a mass parameter, is conformally flat, we could not hope to find
unique surfaces minimizing the corresponding constrained problem.

Furthermore, the functional~\eqref{eq:71} appears naturally in general
relativity in form of the Hawking mass $m_H(\Sigma)$ of a surface $\Sigma$,
defined as
\begin{equation*}
  m_H(\Sigma)
  =
  \frac{|\Sigma|^{1/2}}{(16\pi)^{3/2} } \left( 16\pi - 2 \CW(\Sigma)\right).
\end{equation*}
This quantity is used to measure the mass of a region enclosed by
$\Sigma$. Due to the area constraint, equation~\eqref{eq:1} also
appears as the Euler-Lagrange equation when maximizing $m_H(\Sigma)$
subject to fixed area $|\Sigma|$.

Foliations of asymptotically flat manifolds using constant mean
curvature surfaces have been considered in \cite{Huisken-Yau:1996},
\cite{ye:1996} and \cite{huang:2008foliation}. The uniqueness of such
foliations was considered in \cite{qing-tian}. In
\cite{Huisken-Yau:1996} these foliations have been used to define a
center of mass for initial data sets for isolated gravitating systems
in general relativity. Such data sets are three dimensional
asymptotically flat manifolds. We argue here that, due to its relation
to the Hawking mass, equation~\eqref{eq:1} is the most natural
equation to consider when defining a geometric center of the Hawking
mass. In fact, surfaces maximizing the Hawking mass are the optimal
surfaces to calculate the Hawking mass. This intuition is backed by
our observation that along the foliation we construct, the Hawking
mass is non-negative and non-decreasing in the outward direction,
provided the scalar curvature $\ScalM \geq 0$ is non-negative, cf.\
theorem~\ref{thm:hawking-mono}. We remark that on stable surfaces of
constant mean curvature the Hawking mass is also non-negative as was
shown by Christodoulou and Yau \cite{Christodoulou-Yau:1986}.

Moreover, we wish to mention here that in \cite{huisken:2006owr}
Huisken argues in the other direction and introduces a definition of
quasi-local mass with the constant mean curvature equation as
Euler-Lagrange equation for the optimal surfaces at a given enclosed
volume. This then fits together with the center of mass definition by
CMC spheres.

CMC foliations have also been studied in other contexts, in particular
with asymptotically hyperbolic background in
\cite{neves-tian:2006,neves-tian:2007,mazzeo-pacard:2007}. This
setting is also relevant in general relativity when studying data sets
which are asymptotically light-like. We expect that our results extend
to the asymptotically hyperbolic case.

In $\IR^3$, minima of functional~\eqref{eq:71} are round
spheres, and since the functional is scale and translation invariant,
we get an (at least) four dimensional transformation group. In
particular, we can not expect solutions of~\eqref{eq:1} to be
unique. The existence of surfaces $\Sigma \subset \IR^n$ of
higher genus which minimize the Willmore functional and in particular
satisfy~\eqref{eq:1} with $\lambda = 0$ has been shown by Simon
\cite{MR1243525} and Bauer \& Kuwert \cite{MR1941840}.

This changes when we take the background $M$ not to be $\IR^3$ but the
exterior region of an asymptotically flat manifold. That is $M = \IR^3
\setminus B_\sigma(0)$ and the metric on $M$ is asymptotic to the
spatial Schwarzschild metric $g^S_m$. This metric is the spatial part
of the Schwarzschild metric which describes a single, static black
hole of mass $m$. Thus $m$ has the interpretation of a mass parameter.

In the $g^S$-metric we no longer have translation and scaling
invariance. In fact we will show that solutions of~\eqref{eq:1} which
are close enough to large centered round spheres are in fact equal to
centered round spheres. The radius of the sphere is then uniquely
determined by $\lambda$, provided $\lambda \in (0,\lambda_0)$ is small
enough. If the metric on $M$ is asymptotic to $g^S$ with appropriate
decay conditions, we can show that solutions to~\eqref{eq:1} behave
accordingly and form a smooth foliation of the asymptotic region of
$(M,g)$.

To be precise, we consider metrics $g$ on $\IR^3\setminus
B_\sigma(0)$ with the following asymptotics
\begin{equation*}
  \sup_{\IR^3\setminus B_\sigma(0)}
  \big( r^2 |g-g_m^S|
  + r^{3} |\nabla -\nabla_m^S|
  + r^4 |\Ric-\Ric_m^S|
  + r^5 |\nabla\Ric -\nabla^S \Ric_m^S | \big)
  \leq \eta,
\end{equation*}
where $g_m^S$ is the spatial Schwarzschild metric of mass $m>0$,
$\nabla^S_m$ its Levi-Civita connection and $\Ric^S_m$ its
Ricci-curvature. Correspondingly, $\nabla$ and $\Ric$, are the
connection and curvature of $g$. Furthermore, $r$ is the Euclidean
radius function on $\IR^3\setminus B_\sigma(0)$. Such metrics shall be
called \emph{$(m,\eta,\sigma)$-asymptotically Schwarzschild}.

In this setting, we will prove the following theorem.
\begin{theorem}
  \label{thm:main_existence}
  For all $m>0$ and $\sigma$ there exists $\eta_0 >0$,
  $\lambda_0>0$ and $C<\infty0$ depending only on $m$ and $\sigma$
  such that the following holds.
  
  Let $(M,g)$ be an $(m,\eta,\sigma)$-asymptotically flat manifold
  with $\eta < \eta_0$ and
  \begin{equation*}
    | \ScalM | \leq \eta r^{-5}
  \end{equation*}
  then for each $\lambda\in (0,\lambda_0)$ there exists a surface
  $\Sigma_\lambda$ satisfying equation~\eqref{eq:1}.

  In Euclidean coordinates this surface is $W^{2,2}$-close to a
  Euclidean sphere $S_{R_\lambda}(a_\lambda)$ with radius $R_\lambda$
  and center $a_\lambda$ such that
  \begin{equation*}
    |a_\lambda|+
    \left| R_\lambda - (\lambda/2m)^{-1/3} \right|
    \leq C\eta.
  \end{equation*}

  Moreover, there exists a compact set $K\subset M$ such that
  $M\setminus K$ is foliated by the surfaces
  $\{\Sigma_\lambda\}_{\lambda\in(0,\lambda_0)}$.
\end{theorem}
For an arbitrary surface $\Sigma \subset \IR^3$ we can define a best
matching sphere by introducing the geometric area radius and the
center of gravity, both with respect to the Euclidean background:
\begin{equation*}
  R^e(\Sigma)
  =
  \sqrt{\frac{|\Sigma|_e}{4\pi}}
  \qquad\text{and}\qquad
  a_e (\Sigma)
  =
  |\Sigma|_e^{-1} \int_\Sigma x \dmu^e  
\end{equation*}
where in the second integral, the integrand is the position
vector. Then we define the scale-invariant translation parameter
\begin{equation*}
  \tau(\Sigma) = a_e(\Sigma) / R_e(\Sigma)
\end{equation*}
and we can state the uniqueness theorem
\begin{theorem}
  \label{thm:main_uniqueness}
  Let $m>0$ and $\sigma$ be given. Then there exists $\eta_0>0$,
  $\tau_0>0$, $\eps>0$ and $r_0<\infty$ depending only on $m$ and
  $\sigma$ such that the following holds.

  If $(M,g)$ is an $(m,\eta,\sigma)$-asymptotically flat manifold with
  $\eta<\eta_0$ and
  \begin{equation*}
    | \ScalM | \leq \eta r^{-5}
  \end{equation*}
  then every spherical surface $\Sigma \subset M$ with $\rmin :=
  \min_{\Sigma} r > r_0$, $\tau (\Sigma) < \tau_0$, $R_e \leq \eps
  \rmin^{2}$ and $H>0$ satisfying equation~\eqref{eq:1} for some
  $\lambda >0$ equals one of the surfaces $\Sigma_\lambda$ constructed
  in theorem~\ref{thm:main_existence}. In particular
  $\lambda\in(0,\lambda_0)$.
\end{theorem}
The outline of the paper and the proof of the above theorems is as
follows. After setting the stage by presenting some preliminary
material in section~\ref{sec:preliminaries}, we calculate the first
and second variation of~\eqref{eq:71}, to arrive at~\eqref{eq:1} and
its linearization. This is done in
section~\ref{sec:first-second-vari}.

In section~\ref{sec:integr-curv-estim} we prove a priori estimates for
solutions to~\eqref{eq:1} under the assumption that $H>0$ and
$\lambda>0$. These estimates in particular show that with increasing
area also the Hawking mass of the $\Sigma_\lambda$ increases.

Section~\ref{sec:impr-curv-estim} is devoted to a technical
improvement of the curvature estimates in
section~\ref{sec:integr-curv-estim}, under the additional assumption
that the surface in question is not too far off center in the sense
that the translation parameter above is not too large.

This allows us to break the translation invariance in
section~\ref{sec:position-estimates}, where we prove position
estimates. These estimates are at the heart of the uniqueness and are
quite delicate. In this section we also state the final version of our
a priori estimates. These estimates allow to control both the position
and the shape of solutions to~\eqref{eq:1} in a very precise way.

In section \ref{sec:est-lin-op} we analyze the linearization of
equation~\eqref{eq:1} and use the previous a priori estimates to show
that this operator is invertible. The reason why we are able to do
this, is that the estimates in section~\ref{sec:position-estimates}
allow to compare the linearization of~\eqref{eq:1} to the
corresponding operator on a centered sphere in Schwarzschild. The
latter operator is invertible and thus invertability of the former one
follows.

This is used in section~\ref{sec:exist-uniq-foli} to prove the
existence and uniqueness of theorem~\ref{thm:main_existence} and
theorem~\ref{thm:main_uniqueness} using an argument based on the
implicit function theorem.


%% file: prelim.tex
\section{Preliminaries}
\label{sec:preliminaries}
\subsection{Geometric equations}
We will consider three dimensional Riemannian manifolds $(M,g)$, where
$g$ is the metric tensor, which we write as $g_{ij}$ in
coordinates. Its inverse is denoted by $g^{ij}$, its Levi-Civita
connection by $\nabla$. For the Riemanninan curvature tensor we use
the convention
\begin{equation*}
  (\nabla_i \nabla_j - \nabla_j\nabla_i) \del_k  = \RiemM_{ijkl} g^{lm}\del_m.
\end{equation*}
Here we use the Einstein summation convention and sum over repeated
indices. Then the Ricci-curvature is given by
\begin{equation*}
  \RicM_{il} = g^{jk} \RiemM_{ijkl}
\end{equation*}
and the scalar curvature by $\ScalM = g^{ij}\RicM_{ij}$.

Our sign convention implies that commuting derivatives on a 2-tensor
$T_{ab}$ gives
\begin{equation*}
  \nabla_a\nabla_b T_{cd}
  =
  \nabla_b \nabla_a T_{cd}
  - \RiemM_{abce}g^{ef} T_{fd}
  - \RiemM_{abde}g^{ef} T_{cf}.
\end{equation*}
For a three dimensional manifold the Riemannian
curvature tensor can be expressed in terms of the Ricci curvature as
follows
\begin{equation}
  \label{eq:3}
  \RiemM_{ijkl}
  =
  \RicM_{il} g_{jk}
  - \RicM_{ik} g_{jl}
  - \RicM_{jl} g_{ik}
  + \RicM_{jk} g_{il}
  - \half \ScalM ( g_{il} g_{jk} - g_{ik} g_{jl} ).
\end{equation}
If $\Sigma\subset M$ ia a surface we denote by $\gamma$ the induced
metric and by $\nu$ its normal. The second fundamental form of
$\Sigma$ is denoted by $A$ and its mean curvature by $H$. The
Riemannian curvature tensor $\RiemSig$ of $\Sigma$ is given by the
Gauss equation
\begin{equation}
  \label{eq:4}
  \RiemSig_{ijkl} = \RiemM_{ijkl} + A_{il}A_{jk} - A_{ik} A_{jl}.
\end{equation}
Taking the trace twice implies
\begin{equation}
  \label{eq:14}
  \ScalSig = \ScalM - 2\RicM(\nu,\nu) + H^2 - |A|^2.
\end{equation}
Furthermore, we have the Codazzi equation
\begin{equation}
  \label{eq:6}
  \nabla_k A_{ij} = \nabla_i A_{kj} + \RiemM_{kiaj}\nu^a.
\end{equation}
Denote by $\omega:= \Ric(\nu,\cdot)^T$ the tangential projection of
the 1-form $\Ric(\nu,\cdot)$ to $\Sigma$. Then using the Gauss
equation \eqref{eq:4}, the Codazzi equation \eqref{eq:6} and equation
\eqref{eq:3}, the Simons identity \cite{Simons:1968} becomes
\begin{equation}
  \label{eq:5}
  \begin{split}
    \Delta A_{ij}
    &=
    \nabla_i\nabla_j H
    + H A_i^k A_{kj}
    - |A|^2 A_{ij}
    \\
    & \phantom{=}
    + A^k_j \gamma^{lm} \RiemM_{likm}
    + A^{kl}R_{ikjl}
    + 2\nabla_i\omega_j
    - \div\omega \gamma_{ij}.
  \end{split}
\end{equation}
For any two-tensor $T$, we denote the traceless part by
$T^0$, that is $T^0_{ij} = T_{ij} - \half (\tr T) \gamma_{ij}$. In particular we have
\begin{equation*}
  \Acirc_{ij}  = A_{ij} - \half H \gamma_{ij}.
\end{equation*}
This implies that
\begin{equation*}
  |\Acirc|^2 +\half H^2 = |A|^2.
\end{equation*}
With the help of these facts we get from Simons` identity that
\begin{equation}
  \label{eq:7}
  \begin{split}
    \Delta \Acirc_{ij}
    &=
    (\nabla^2 H)^0_{ij}
    + H\Acirc_i^k \Acirc_{kj}
    + \half H^2 \Acirc_{ij}
    -|\Acirc|^2 \Acirc_{ij}
    - \half H |\Acirc|^2 \gamma_{ij}
    \\
    & \phantom{=}
    + \Acirc_j^k \gamma^{lm} \RiemM_{likm}
    + \Acirc^{kl} \RiemM_{ikjl}
    + 2 \nabla_i \omega_j - \div \omega\gamma_{ij},
  \end{split}
\end{equation}
and therefore
\begin{equation}
  \label{eq:8}
  \begin{split}
    \Acirc^{ij} \Delta \Acirc_{ij}
    &=
    \la \Acirc, \nabla^2 H \ra
    + \half H^2 |\Acirc|^2
    - |\Acirc|^4
    \\
    & \phantom{=}
    -|\Acirc|^2 \RicM(\nu,\nu)
    + 2\Acirc^{ij}\Acirc_j^l \RicM_{il}
    + 2 \la \Acirc, \nabla\omega\ra.
  \end{split}
\end{equation}
\subsection{Asymptotically Schwarzschild manifolds}
Let $g^S_m$ be the spatial, conformally flat Schwarzschild metric on
$\IR^3 \setminus \{ 0\}$ of mass $m$.  That is $g^S_m = \phi_m^4 g^e$,
where $\phi_m = 1 +\frac{m}{2r}$, $g^e$ is the Euclidean metric on
$\IR^3$ and $r$ the distance to the origin in $\IR^3$. We will
suppress the depencence of $g^S_m$ and $\phi_m$ on $m$ and denote the metric
simply by $g^S$ and $\phi_m$ by $\phi$. The following lemma
summarizes the relationship of the geometry of $g^S$ and~$g^e$.
\begin{lemma}
  \label{thm:geometry-in-schwarzschild}
  \begin{enumerate}
  \item The Ricci curvature of $g^S$ is given by
    \begin{equation}
      \label{eq:12}
      \Ric^S_{ij}
      =
      \frac{m}{r^3} \phi^{-2} \big( g^e_{ij} - 3 \rho_i\rho_j \big),
    \end{equation}
    where $\rho_a$ is the 1-form dual to the vector $\dd{}{r}$ on
    $\IR^3$.    
    In particular, the scalar curvature of $g^S$ vanishes.
  \item If $\Sigma\subset \IR^3\setminus{0}$ is a surface, we denote
    by $\nu^e$ the normal of $\Sigma$ with respect to $g^e$ and by
    $\nu^S$ the normal of $\Sigma$ with respect to $g^S$. Analogously
    $\dmu^e$, $\dmu^S$ denote the respective volume forms, $\Acirc^e$,
    $\Acirc^S$ the respective traceless second fundamental forms and
    $H^e$ and $H^S$ the mean curvatures. We find the following
    relations:
    \begin{align}
      \label{eq:13}
      &
      \nu^S
      =
      \phi^{-2}\nu^e,
      \\
      &
      \dmu^S
      =
      \phi^4 \dmu^e,
      \\
      &
      \Acirc^S
      =
      \phi^{-2}\Acirc^e, \ \text{and}
      \\
      &
      H^S
      =
      \phi^{-2} H^e + 4 \phi^{-3} \del_{\nu^e}\phi.
    \end{align}
  \end{enumerate}
\end{lemma}
\begin{definition}
  \label{def:asymptotically-flat}
  We say that $(M,g)$ is $(m,\eta,\sigma)$-asymptotically Schwarzschild if
  there exists a compact set $B\subset M$, and a diffeomorphism $x:
  M\setminus B \to \IR^3\setminus B_\sigma(0)$, such that in these
  coordinates
  \begin{equation*}
    \sup_{\IR^3\setminus B_\sigma(0)}\!\!\!
    \big( r^2 |g-g^S|
    + r^{3} |\nabla^g -\nabla^S|
    + r^4 |\Ric^g-\Ric^S|
    + r^5 |\nabla\Ric^g -\nabla^S \Ric^S | \big)
    \leq \eta,
  \end{equation*}
  where $g^S$ is the metric for mass $m$.  
\end{definition}
For brevity we will subsequently refer to $\Ric^g$ simply by $\Ric$ or
by $\RicM$.

In the next lemma we relate geometric quantities with respect to $g$
to quantities with respect to $g^S$.
\begin{lemma}
  \label{thm:general-to-schwarzschild}
  If $(M,g)$ is $(m, \eta,\sigma)$ asymptotically Schwarzschild and if
  $\Sigma\subset \IR^3\setminus B_{\sigma}(0)$ is a surface, we have the
  following relation between the normals $\nu$ with respect to $g$
  and $\nu^S$ with respect to $g^S$
  \begin{equation*}
    r^2|\nu - \nu^S| \leq C\eta.
  \end{equation*}
  Furthermore, the area elements $\dmu$ and $\dmu^S$ satisfy $\dmu -
  \dmu^S = h\dmu$ with
  \begin{equation*}
    r^2 |h| \leq C\eta,
  \end{equation*}
  The second fundamental forms $A$ and $A^S$  satisfy
  \begin{align*}
    & |A - A^S|
    \leq
    C\eta(r^{-3} + r^{-2}|A|)
    \\
    & |\nabla A - \nabla A^S|
    \leq
    C\eta(r^{-4} + r^{-3}|A| + r^{-2}|\nabla A|).    
  \end{align*}
\end{lemma}
To estimate integrals of decaying quantities we use the variant of
\cite[Lemma 5.2]{Huisken-Yau:1996} as stated in \cite[Lemma
2.3]{Metzger:2007ce}.
\begin{lemma}
  \label{thm:integral-powers}
  Let $(M,g)$ be $(m,\eta,\sigma)$-asymptotically Schwarzschild, and let
  $p_0>2$ be fixed. Then there exists $c(p_0)$ and $r_0 =
  r_0(m,\eta,\sigma)$, such that for every surface $\Sigma\subset
  \IR^3\setminus B_{r_0}(0)$, and every $p>p_0$, the following
  estimate holds
  \begin{equation*}
    \int_\Sigma r^{-p} \dmu
    \leq c(p_0) \rmin^{2-p}\int_\Sigma H^2\dmu.
  \end{equation*}
  Here $\rmin := \min_\Sigma r$, where $r$ is the Euclidean radius.
\end{lemma}
In the sequel we will also need decay properties of volume integrals.
\begin{lemma}
  \label{thm:volume-integral-decay}
  Let $\Omega$ be an exterior domain with compact interior boundary
  $\Sigma$. Then for all $p>3$ there exists a constant $C(p)$ and
  $r_0$ such that if $\rmin>r_0$ we have 
  \begin{equation*}
    \int_\Omega r^{-p} dV
    \leq
    C(p) \rmin^{3-p} \int_\Sigma H^2\dmu.
  \end{equation*}
\end{lemma}
\begin{proof}
  Let $\rho$ be the Euclidean radial direction, and let $X =
  r^{-p+1}\rho$. With respect to $g$ we have
  \begin{equation*}
    \div X = (3-p) r^{-p} + O(r^{-p-1}).
  \end{equation*}
  Choose $r_0$ so large that the error term is dominated by the main
  term in this equation, that is
  \begin{equation*}
    (p-3 -\eps) r^{-p} \leq - \div X,    
  \end{equation*}
  where $\eps$ is such that $p-3 -\eps>0$. Integrating this relation
  over $\Omega$ and partially integrating on the right hand side
  yields the estimate
  \begin{equation*}
    \int_\Omega r^{-p} dV
    \leq
    \frac{1}{p-3-\eps}\int_\Sigma \la X, \nu \ra.
  \end{equation*}
  Note that the boundary integral at infinity vanishes as the surface
  integrand decays like $r^{-p+1}$. The claim then follows from
  lemma~\ref{thm:integral-powers}.
\end{proof}
Using the conformal invariance of $\|\Acirc\|_{L^2(\Sigma)}$, which
can be seen via lemma~\ref{thm:geometry-in-schwarzschild}, we derive:
\begin{lemma}
  \label{thm:a0-decay-preserved}
  Let $(M,g)$ be $(m,\eta,\sigma)$-asymptotically Schwarzschild. Then there
  exists $r_0 = r_0(\eta,\sigma)$ such that for every surface
  $\Sigma\subset \IR^3\setminus B_{r_0}(0)$ we have
  \begin{equation*}
    \begin{split}
      &\left| \|\Acirc^e\|^2_{L^2(\Sigma,g^e)} - \|\Acirc\|^2_{L^2(\Sigma,g)}
      \right|
      \\ 
      &\quad
      \leq
      C\eta\rmin^{-2}\left(\|\Acirc\|^2_{L^2(\Sigma,g)}
        + \|H\|_{L^2(\Sigma)}\|\Acirc\|_{L^2(\Sigma)} +  \eta\rmin^{-2}\|H\|^2_{L^2(\Sigma)}
      \right) .
    \end{split}
  \end{equation*}
\end{lemma}
\begin{corollary}
  \label{koro:acirc}
  Let $(M,g)$, $r_0$ and $\Sigma$ be as in the previous lemma. Assume in
  addition that $\|H\|_{L^2(\Sigma)} \leq C'$, then
  \[ \| \Acirc^e \|_{L^2(\Sigma)} \leq C(r_0)\|\Acirc\|_{L^2(\Sigma,g)} +
  C(r_0,C')\eta \rmin^{-2}.
  \]
\end{corollary}
We need the following variant of the Michael-Simon Sobolev inequality
\cite{Michael-Simon:1973} as stated in \cite[Proposition
5.4]{Huisken-Yau:1996}.
\begin{proposition}
  \label{thm:sobolev}
  Let $(M,g)$ be $(m,\eta,\sigma)$-asymptotically Schwarzschild. Then
  there is $r_0 = r_0 (m,\eta,\sigma)$ and an absolute constant $C_s$
  such that for each surface $\Sigma \subset M\setminus B_{r_0}(0)$
  and each Lipschitz function $f$ on $\Sigma$ we have the estimate
  \begin{equation}
    \label{eq:9}
    \left(\int_\Sigma |f|^2 \dmu \right)^{1/2}
    \leq
    C_s \int_\Sigma |\nabla f| + |Hf| \dmu.
  \end{equation}
  Via H\"older's inequality, this implies that for all $q\geq 2$
  \begin{equation}
    \label{eq:10}
    \left(\int_\Sigma |f|^{q} \dmu\right)^{\frac{2}{2+q}}
    \leq
    C_s \int_\Sigma
    |\nabla f|^{\frac{2q}{2+q}}
    + |Hf|^{\frac{2q}{2+q}}
    \dmu,    
  \end{equation}
  and for all $p\geq 1$,
  \begin{equation}
    \label{eq:11}
    \left( \int_\Sigma |f|^{2p} \dmu \right)^{1/p}
    \leq
    C_s p^2 |\supp f|^{1/p} \int_\Sigma |\nabla f|^2 +H^2 f^2 \dmu.
  \end{equation}
\end{proposition}
\subsection{Almost umbilical surfaces in Euclidean space}
To conclude that the surfaces we consider are close to spheres, we use
the following theorem for surfaces in Euclidean space. This is proved
in \cite[Theorem 1]{DeLellis-Muller:2005} and \cite[Theorem
2]{DeLellis-Muller:2006}.
\begin{theorem}
  \label{thm:umbilical}
  There exists a universal constant $c$ such that for each compact
  connected surface without boundary $\Sigma\subset\IR^3$ with area
  $|\Sigma|=4\pi$, the following estimate holds
  \begin{equation*}
    \| A^e - \gamma^e \|_{L^2(\Sigma,\gamma^e)}
    \leq
    c \|\Acirc^e\|_{L^2(\Sigma,\gamma^e)}.
  \end{equation*}
  If in addition $\|\Acirc^e\|_{L^2(\Sigma,\gamma^e)}\leq 8\pi$, then
  $\Sigma$ is a sphere, and there exists a conformal map $\psi: S^2
  \to \Sigma \subset \IR^3$ such that
  \begin{equation*}
    \| \psi - (a + \id_{S^2}) \|_{W^{2,2}(S^2)}
    \leq
    c \|\Acirc^e\|_{L^2(\Sigma,\gamma^e)},    
  \end{equation*}
  where $\id_{S^2}$ is the standard embedding of $S^2$ onto 
  the sphere $S_1(0)$ in $\IR^3$, and
  \begin{equation*}
    a
    =
    |\Sigma|_e^{-1} \int_\Sigma \id_\Sigma \dmu^e  
  \end{equation*}
  is the center of gravity of $\Sigma$.  
  The conformal factor $h$ of the embedding $\psi$, that is 
  $\psi^*\gamma^e = h^2 \gamma_{S^2}$, satisfies
  \begin{equation*}
    \| h - 1 \|_{W^{1,2}(S^2)} + \sup_{S^2} | h -1 |
    \leq
    c \|\Acirc^e\|_{L^2(\Sigma,\gamma^e)}.
  \end{equation*}
  The normal $\nu^e$ of $\Sigma$  satisfies
  \begin{equation*}
    \| N - \nu\circ \psi \|_{W^{1,2}(S^2)}
    \leq
    c \|\Acirc^e\|_{L^2(\Sigma,\gamma^e)},    
  \end{equation*}
  where $N$ is the normal of $S_1(a)$.
\end{theorem}
To get the scale-invariant form of these estimates, we proceed as
follows. For a surface $\Sigma$ with arbitrary area $|\Sigma|_e$ let
$R_e = \sqrt{|\Sigma|_e/4\pi}$. Then the first part of theorem
\ref{thm:umbilical} implies that
\begin{equation*}
  \| A - R_e^{-1}\gamma^e \|_{L^2(\Sigma,\gamma^e)}
  \leq
  c \|\Acirc^e\|_{L^2(\Sigma,\gamma^e)}.
\end{equation*}
Again let $a_e$ denote the center of gravity of $\Sigma$,
\begin{equation*}
  a_e :=  \frac{1}{4\pi R_e^2} \int_\Sigma \id_\Sigma \dmu^e \in \IR^3.
\end{equation*}
Then if $\|\Acirc^e\|_{L^2(\Sigma,\gamma^e)}\leq 8\pi$, the second
part of theorem \ref{thm:umbilical} gives that there exists a
conformal parametrization $\psi : S_{R_e}(a_e) \to \Sigma$. The
estimates from theorem \ref{thm:umbilical} imply together with the
Sobolev-embedding theorems on $S^2$, that the following estimates hold
\begin{align}
  \label{eq:16}
  \sup_{S_{R_e}(a_e)} \big| \psi - \id_{S_{R_e}(a_e)} \big|
  &\leq
  C R_e  \|\Acirc^e\|_{L^2(\Sigma,\gamma^e)},
  \\
  \label{eq:17}
  \| N \circ \id_{S_{R_e}(a_e)} - \nu \circ \psi \|_{L^2(S)}
  &\leq
  C R_e \|\Acirc^e\|_{L^2(\Sigma,\gamma^e)}.
  \intertext{and}
  \label{eq:18}  
  \sup_{S_{R_e}(a_e)}| h^2 - 1 |
  &\leq
  C \|\Acirc^e\|_{L^2(\Sigma,\gamma^e)}.
\end{align}
Here, as before, $h$ denotes the conformal factor of the map $\psi$
and $N$ is the normal of $S_{R_e}(a_e)$.


%% file: variation.tex
\section{First and Second Variation}
\label{sec:first-second-vari}
In this section we calculate the first and second variation of the Willmore functional subject to an area constraint.

To compute the first variation of $\CW$ let $\Sigma\subset M$ be a
surface and let $F:\Sigma\times (-\eps,\eps)\to M$ be a variation of
$\Sigma$ with $F(\Sigma, s) = \Sigma_s$ and lapse $\ddeval{F}{s}{s=0} =
\alpha \nu$. Recall the following well known evolution equations for deformations of hypersurfaces (see for example \cite{huisken-polden:1999}). Here and in the following we will understand that all
$s$-derivatives are evaluated at $s=0$, and will not further denote
this explicitely:
\begin{align*}
  &\dd{}{s} \gamma_{ij} = 2\alpha A_{ij},
  \\
  &\dd{}{s} \dmu =\alpha H,
  \\
  &\dd{}{s} \gamma^{ij} = -2\alpha A^{ij},
  \\
  &\dd{}{s}\nu = -\nabla\alpha,
  \\
  &\dd{}{s}A_{ij} = -\nabla_i\nabla_j\alpha + \alpha \big(A_{ik}A^k_j - T_{ij}\big),
  \\
  &\dd{}{s}H = L\alpha,
\end{align*}
where 
\begin{equation}
  \label{eq:31}
  Lf = -\Delta f - f\big( |A|^2 + \RicM(\nu,\nu) \big)
\end{equation}
is the well known Jacobi operator for minimal surfaces, 
\[
T_{ij} = \RiemM(\del_i,\nu,\nu,\del_j) = \RicM^T_{ij} + G(\nu,\nu)\gamma_{ij}
\] 
and $G=\RicM-\frac{1}{2} \ScalM \cdot g$ is the Einstein tensor.

The first variation of $\CW$ can then be computed as
\begin{equation}
  \label{eq:30}
  0 = \DDeval{}{s}{s=0} \CW[\Sigma_s]
  =
  \int _\Sigma H L\alpha + \half H^3\alpha \dmu
  =
  \int_\Sigma \big( LH + \half H^3 \big) \alpha \dmu.
\end{equation}
A critical
point for $\CW$ therefore satisfies the Euler-Lagrange equation
\begin{equation}
  \label{eq:32}
  LH + \half H^3 = 0.
\end{equation}
To compute the second variation of $\CW$, note that by \eqref{eq:30}
\begin{equation}
  \label{eq:33}
  \begin{split}
    \DDeval{^2}{s^2}{s=0} \CW [\Sigma_s]
    &=
    \int_\Sigma \dd{}{s}\big(-\Delta H - H|A|^2 - H \RicM(\nu,\nu) +
    \half H^3 \big)\alpha \dmu\Big|_{s=0}
    \\
    &\phantom{=}
    + \int_\Sigma
    \big(LH + \half H^3\big)\big( \dd{\alpha}{s} +
    H\alpha^2 \big)\Big|_{s=0} \dmu.
  \end{split}
\end{equation}
Thus we have to compute the linearization of the Willmore operator
defined as follows
\begin{equation}
  \label{eq:35}
  \begin{split}
W\alpha :=&\ \DDeval{}{s}{s=0} \big(-\Delta H - H|A|^2 - H
  \RicM(\nu,\nu) + \half H^3\big)
  \\
    = &\ 
    - [\dd{}{s}, \Delta] H
    - H \dd{}{s}|A|^2
    - H\dd{}{s}\RicM(\nu,\nu)
    + LL\alpha + \tfrac{3}{2} H^2 L\alpha.
  \end{split}
\end{equation}
Using the above formula for the variations of the metric and the second fundamental form we compute
\begin{equation*}
  \dd{}{s} A^{ij} 
  =
  -3\alpha A^{ik}A^j_k -\nabla^i\nabla^j\alpha - \alpha T^{ij}
\end{equation*}
and therefore
\begin{equation}
  \label{eq:41}
  \begin{split}
    \dd{}{s} |A|^2
    &=
    \dd{}{s} (A^{ij}A_{ij})
    =
    -2\alpha \tr A^3 - 2 A_{ij}\nabla^i\nabla^j\alpha - 2\alpha
    A^{ij}T_{ij}. 
  \end{split}
\end{equation}
The next term we compute is $\dd{}{s}\RicM(\nu,\nu)$, yielding
\begin{equation}
\label{eq:42}
  \dd{}{s}\RicM(\nu,\nu) = \alpha\nabla_\nu\RicM(\nu,\nu) - 2\,\RicM(\nabla\alpha,\nu).
\end{equation}
We turn to computing the commutator $[\dd{}{s},\Delta]$. We write
$\Delta= \div \nabla$ and we compute the commutator of
$[\dd{}{s},\div]$ and $[\dd{}{s},\nabla]$ individually. First note
that since $\nabla^k \phi = \gamma^{kl}\dd{\phi}{x^l}$ we have
\begin{equation*}
  \dd{}{s}\big(\nabla^k \phi\big) = -2\alpha A^{kl}\dd{\phi}{x^l} +
  \gamma^{kl}\dd{}{x^l}\dd{}{s}\phi,
\end{equation*}
and hence
\begin{equation}
  \label{eq:38}
  [\dd{}{s},\nabla] \phi = -2\alpha A^k_l \nabla^l\phi = -2\alpha S(\nabla\phi).
\end{equation}
Here $S$ is the shape operator, that is the tensor defined by
\begin{equation*}
  \gamma(S(X),Y) = A(X,Y)
\end{equation*}
for all $X,Y\in\CX(\Sigma)$. Now we turn to the computation of
$[\dd{}{s},\div]$, operating on vector fields. Let $X,Y\in\CX(\Sigma)$
be vector fields. We compute
\begin{equation}
  \label{eq:36}
  \gamma(\nabla_X Y, X) = \gamma(\nabla_Y X, X) + \gamma([X,Y],X) = \frac{1}{2} Y (\gamma(X,X)) + \gamma(X, [X,Y]). 
\end{equation}
We choose a local orthonormal frame $\{e_i\}$ and propagate it using the ODE
\begin{equation*}
  \dd{}{s}e_i = -\alpha S(e_i).
\end{equation*}
Then the $\{e_i\}$ remain orthonormal under the evolution.
Plugging $X = e_i$ into equation \eqref{eq:36} yields
\begin{equation*}
  \gamma(\nabla_{e_i} Y, e_i) = \gamma(e_i, [e_i,Y]). 
\end{equation*}
Differentiating this equation and using the above formulas we get by a fairly standard computation
\begin{equation*}
  \begin{split}
    \dd{}{s} \gamma(\nabla_{e_i} Y, e_i)
    &= 2\alpha A(e_i,[e_i,Y])-\gamma(\alpha S(e_i),[e_i,Y])-\gamma(e_i,[\alpha S(e_i),Y])\\
    &= \alpha A(e_i,\nabla_{e_i} Y)-\alpha A (e_i,\nabla_Y e_i)-\alpha \gamma(e_i,\nabla_{S(e_i)} Y)
    \\ & \phantom{=}
    + \alpha Y(\gamma(e_i,S(e_i)))-\alpha \gamma(\nabla_Y e_i,S(e_i))+Y(\alpha) A(e_i,e_i)\\
    &= \alpha A(e_i,\nabla_{e_i} Y)-\alpha \gamma(e_i,\nabla_{S(e_i)} Y)
    \\ & \phantom{=}
    + \alpha \nabla_Y A(e_i,e_i) 
    + Y(\alpha) A(e_i,e_i).
  \end{split}
\end{equation*}
If we now choose $\{e_i\}$ to be an orthogonal system of eigenvectors for
$S$, that is $S(e_i) = \lambda_i e_i$, then we see that the first two terms
cancel, and after summation over $i$ we infer
\begin{equation}
  \label{eq:37}
  [\dd{}{s},\div] Y =  \sum_i \alpha \nabla_Y A(e_i,e_i) 
    + Y(\alpha) A(e_i,e_i) = \nabla_Y(\alpha H).
\end{equation}
We combine equations \eqref{eq:37} and (\ref{eq:38}) and get, using
$\Delta = \div \nabla$,
\begin{equation}
  \label{eq:39}
  \begin{split}
    [\dd{}{s},\Delta]\phi
    &=
    \la \nabla\phi, \nabla(\alpha H) \ra -2 A(\nabla\alpha,\nabla\phi)
    - 2\alpha \div\big(S(\nabla\phi)\big).
  \end{split}
\end{equation}
Using an ON frame $\{e_i\}$, we compute further that
\begin{equation*}
  \begin{split}
    \div \big(S(\nabla\phi)\big)
    &=
    \sum_i \nabla_{e_i}A(\nabla\phi,e_i) + A(\nabla_{e_i}\nabla\phi,e_i)
  \end{split}
\end{equation*}
and in view of the Codazzi equation this yields
\begin{equation*}
  \begin{split}
    \div \big(S(\nabla\phi)\big)
    &=
    \la \nabla\phi,\nabla H\ra
    + \sum_i \RiemM(e_i,\nabla\phi,\nu,e_i)
    + A(\nabla_{e_i}\nabla\phi,e_i)
    \\
    &=
    \la \nabla\phi,\nabla H\ra
    + \RicM(\nabla\phi,\nu)
    + \la A, \nabla^2\phi\ra.
  \end{split}
\end{equation*}
Plugging this formula into \eqref{eq:39} gives
\begin{equation}
  \label{eq:40}
  \begin{split}
    [\dd{}{s},\Delta]\phi
    &=
    H\la\nabla\alpha,\nabla\phi\ra
    -\alpha\la\nabla\phi, \nabla H\ra
    - 2 A(\nabla\alpha, \nabla\phi)
    \\
    &\phantom{=}
    - 2\alpha\RicM(\nabla\phi,\nu)
    - 2\alpha\la A,\nabla^2\phi\ra.
  \end{split}
\end{equation}
Finally we substitute the results \eqref{eq:41}, \eqref{eq:42} and
\eqref{eq:40} into \eqref{eq:35} to obtain
\begin{equation}
  \label{eq:43}
  \begin{split}
    W\alpha
    &=
    LL\alpha
    + \tfrac{3}{2} H^2 L\alpha
    - H\la \nabla\alpha, \nabla H\ra
    + \alpha |\nabla H|^2
    \\
    &\quad
    + 2 A(\nabla\alpha,\nabla H)
    + 2 \alpha\RicM(\nabla H, \nu)
    + 2 \alpha \la A, \nabla^2 H\ra
    \\
    &\quad
    + 2 \alpha H\tr A^3
    + 2 H\la A, \nabla^2\alpha\ra
    + 2 \alpha H \la A, T\ra
    \\
    &\quad
    - \alpha H \nabla_\nu\RicM(\nu,\nu)
    + 2 H\RicM(\nabla\alpha,\nu).
  \end{split}  
\end{equation}
We rewrite equation \eqref{eq:43} in dimension two, as it somewhat
simplifies. We split $A = \Acirc + \frac{1}{2} H \gamma$ in the following terms
\begin{align*}
  \la A, \nabla^2 \alpha \ra
  &=
  \la \Acirc, \nabla^2\alpha\ra + \half H \Delta\alpha,
  \\
  A(\nabla\alpha,\nabla H)
  &=
  \Acirc(\nabla\alpha, \nabla H) + \half H \la \nabla\alpha,\nabla
  H\ra,
  \\
  \la \nabla^2 H,A \ra
  &=
  \half H \Delta H + \la \Acirc, \nabla^2 H\ra,
  \\
  \tr A^3
  &=
  \tr \Acirc^3 + H|\Acirc|^2 + \half H |A|^2
  =
  H|\Acirc|^2 + \half H |A|^2,
  \\
  \la A, T \ra
  &=
  \half H\, \RicM(\nu,\nu) + \la \Acirc,T\ra.
\end{align*}
Plugging these into \eqref{eq:43}, and setting $\omega =
\Ric(\nu,\cdot)^T$ yields
\begin{equation}
  \label{eq:44}
  \begin{split}
    W\alpha
    &=
    LL\alpha
    + \half H^2 L\alpha
    + 2H\la \Acirc, \nabla^2 \alpha\ra
    + 2H\omega(\nabla\alpha)
    + 2\Acirc(\nabla\alpha,\nabla H)
    \\
    &\quad
    +\alpha\big(
    |\nabla H|^2
    + 2 \omega(\nabla H)
    + H \Delta H + 2\la \nabla^2 H, \Acirc\ra 
    \\
    &\quad \phantom{+2\alpha\big(}
    + 2 H^2 |\Acirc|^2
    + 2 H\la \Acirc, T\ra
    - H\nabla_\nu\RicM(\nu,\nu)
    \big).
  \end{split}
\end{equation}
To demonstrate that $W$ is $L^2$-self adjoint we compute, with $D=
|A|^2 + \Ric(\nu,\nu)$,
\begin{equation*}
  \begin{split}
    \int_\Sigma\beta H^2 L\alpha \dmu
    &=
    \int_\Sigma\beta H^2 (-\Delta\alpha - \alpha D) \dmu
    \\
    &=
    \int_\Sigma H^2 \la \nabla\alpha,\nabla\beta\ra
    + 2 H \beta\la\nabla H,\nabla\alpha\ra
    - \alpha\beta H^2 D \dmu,
  \end{split}
\end{equation*}
and, using $\div\Acirc = \half \nabla H + \omega$,
\begin{equation*}
  \begin{split}
    &\int_\Sigma\beta H \la \Acirc, \nabla^2\alpha\ra \dmu
    \\
    &=
    -\int_\Sigma \beta \Acirc(\nabla\alpha,\nabla H) + H \Acirc(\nabla\alpha,
    \nabla\beta) + \half \beta H \la \nabla\alpha, \nabla H \ra
    + H\beta \omega(\nabla\alpha)\dmu.
  \end{split}
\end{equation*}
Thus
\begin{equation}
  \label{eq:45}
  \begin{split}
    &\int_\Sigma \beta W\alpha\dmu
    \\
    &=
    \int_\Sigma
    L\alpha L\beta
    + \half H^2 \la\nabla\alpha,\nabla\beta\ra
    - 2 H\Acirc(\nabla\alpha,\nabla\beta)
    \\
    &\quad
    + \alpha\beta\big(
     |\nabla H|^2
    + 2 \omega(\nabla H)
    + H \Delta H + 2\la \nabla^2 H, \Acirc\ra 
    + 2 H^2 |\Acirc|^2
    \\
    &\quad \phantom{+\alpha\big(}
    + 2 H\la \Acirc, T\ra
    - H\nabla_\nu\RicM(\nu,\nu)
    - \half H^2|A|^2
    - \half H^2\RicM(\nu,\nu)
    \big).
  \end{split}
\end{equation}
and from this representation it is obvious that the bilinear form
associated to $W$ is symmetric, and hence $W$ is $L^2$-self adjoint.

Recall that the goal is to find a critical point of the Willmore
energy in the class of surfaces with given area. From (\ref{eq:32}) we get that for a critical point of this problem we have
\begin{equation}
  \label{eq:52}
  0 = \int_\Sigma(LH +\half H^3)\alpha \dmu
\end{equation}
for all $\alpha$ which respect the constraint $\int_\Sigma
\alpha H \dmu = 0$. We thus find the Euler-Lagrange equation
\begin{equation}
  \label{eq:51}
  LH + \half H^3 = \lambda H,
\end{equation}
where $\lambda$ is a constant. Let us turn to the computation of the
second variation
\begin{equation}
  \label{eq:53}
  \ddeval{^2}{s^2}{s=0} \CW[\Sigma_s]
  =
  \int_\Sigma \alpha W\alpha + (LH + \half H^3)(\dd{\alpha}{s} + H\alpha^2)\dmu.
\end{equation}
At this point we only consider variations that leave the area constant
up to second order. This gives
\begin{equation}
  \label{eq:54}
  0
  =
  \ddeval{^2}{s^2}{s=0} |\Sigma_s|
  =
  \ddeval{}{s}{s=0}\int_{\Sigma_s}\alpha H \dmu
  =
  \int_\Sigma \dd{\alpha}{s} H + \alpha L\alpha + \alpha^2 H^2 \dmu.
\end{equation}
Thus we can compute
\begin{equation}
  \label{eq:55}
  \int_\Sigma (LH + \half H^3)(\dd{\alpha}{s} + H\alpha^2)\dmu
  =
  \int_\Sigma \lambda H (\dd{\alpha}{s} + H\alpha^2)\dmu
  =
  -\lambda \int_\Sigma \alpha L \alpha.  
\end{equation}
Plugging this into \eqref{eq:53} yields that the second variation of
$\CW$ on a stationary surface $\Sigma$ is given by
\begin{equation}
  \label{eq:56}
  \delta^2 \CW(\alpha,\alpha)
  =
  \int_\Sigma \alpha W \alpha - \lambda \alpha L \alpha \dmu,
\end{equation}
for all valid test functions $\alpha\in C^\infty(\Sigma)$ satisfying
$\int_\Sigma \alpha H \dmu = 0$.


%% file: integral.tex
\section{Integral curvature estimates}
\label{sec:integr-curv-estim}
In this section we derive a priori bounds on the curvature of surfaces
which are solutions of the equation~\eqref{eq:1}. We will later make the assumption that both $H>0$ and $\lambda >0$ on these surfaces. Without the assumption on $\lambda$ we
can derive the following lemma.
\begin{lemma}
  \label{thm:init-integral}
  If a spherical surface $\Sigma$ satisfies equation~\eqref{eq:1} with  $H>0$, then
  \begin{equation*}
    \lambda |\Sigma|
    +
    \int_\Sigma |\nabla \log H|^2 +
    \tfrac{1}{4} H^2  + \half |\Acirc|^2\dmu
    \leq
    4\pi - \int_\Sigma \half \ScalM \dmu.
  \end{equation*}
  If $\ScalM\geq 0$ we have that
  \begin{equation*}
    4 \lambda |\Sigma|
    +
    \int_\Sigma H^2 \dmu \leq 16\pi.
  \end{equation*}
\end{lemma}
\begin{proof}
  Multiply equation~\eqref{eq:1} by $H^{-1}$ and integrate the first
  term by parts. This yields
  \begin{equation}
    \label{eq:2}
    \lambda|\Sigma| + \int_\Sigma  |\nabla \log H|^2 + |\Acirc|^2 +
     \RicM(\nu,\nu)\dmu = 0.
  \end{equation}
  We can now use the Gauss equation~\eqref{eq:14} and the Gauss-Bonnet formula to get
  \begin{equation*}
    \lambda |\Sigma| + \int_\Sigma |\nabla \log H|^2 +
    \tfrac{1}{4} H^2  + \half |\Acirc|^2\dmu
    \leq
    4\pi - \int_\Sigma \half\ScalM \dmu.
  \end{equation*}
\end{proof}
The above lemma already implies that the Hawking mass is positive on such surfaces.
\begin{theorem}
  \label{thm:hawking-mono}
  If $(M,g)$ satisfies $\ScalM\geq 0$ and if $\Sigma$ is a compact
  spherical surface satisfying equation~\eqref{eq:1} with $H>0$,
  then $m_H(\Sigma)\geq 0$ if $\lambda \geq 0$.

  Furthermore if $F : \Sigma\times[0,\eps) \to M$ is a variation with
  initial velocity $\ddeval{F}{s}{s=0} = \alpha\nu$ and $\int_\Sigma
  \alpha H \dmu \geq 0$, then
  \begin{equation*}
    \DD{}{s} m_H\big(F(\Sigma,s)\big) \geq 0.
  \end{equation*}
  Note that the condition on $\alpha$ means that the area is
  increasing along the variation.
\end{theorem}
\begin{proof}
  Non-negativity of the Hawking-mass is obvious from
  lemma~\ref{thm:init-integral}. To observe monotonicity, we compute
  the variation of the Hawking-mass. We denote $F(\Sigma,s) =
  \Sigma_s$.
  \begin{equation*}
    \begin{split}
      &(16\pi)^{3/2} \DDeval{}{s}{s=0} m_H(\Sigma_s)
      \\
      &\quad=
      \frac{1}{2|\Sigma|^{1/2}}\left(\int_\Sigma \alpha H \dmu
      \right)\left(16\pi - \int_\Sigma H^2\dmu\right)
      - 2|\Sigma|^{1/2} \int_\Sigma \lambda \alpha H \dmu
    \end{split}
  \end{equation*}
  as equation~(\ref{eq:1}) implies that the variation of $\int_\Sigma
  H^2\dmu$ is given by $2\lambda H$. This yields
  \begin{equation*}
    (16\pi)^{3/2} \DDeval{}{s}{s=0} m_H
    =
    \frac{1}{2|\Sigma|^{1/2}}\left(\int_\Sigma \alpha H \dmu
    \right)\left(16\pi -4\lambda|\Sigma| -\int_\Sigma H^2 \dmu\right).
  \end{equation*}
  Lemma~\ref{thm:init-integral} implies non-negativity of the right
  hand side.  
\end{proof}
Subsequently we assume that the manifold $(M,g)$ is
$(m,\eta,\sigma)$-asymp\-to\-ti\-cal\-ly Schwarzschild for some
$\eta<\eta_0$, where $\eta_0$ is fixed. Furthermore $\Sigma\subset M$
is a surface with $\rmin \geq r_0$ large enough. The particular $r_0$
will only depend on $m$, $\eta_0$ and $\sigma$, and we will no longer
explicitly denote the dependence on these quantities. Similarly,
constants denoted with a capital $C$ are understood to depend on $m$,
$\eta_0$ and $\sigma$, in addition to quantities explicitly mentioned.
In contrast, constants denoted by $c$ will not have any implicit
dependency. We no longer require the condition $\ScalM\geq 0$.
\begin{lemma}
  \label{thm:initial-roundness}
  Let $(M,g)$ be $(m,\eta,\sigma)$-asymptotically Schwarzschild. Then
  there exists $r_0= r_0(m,\eta,\sigma)$ and a constant $C =
  C(m,\eta,\sigma)$ such that for all spherical surfaces
  $\Sigma\subset M\setminus B_{r_0}(0)$ satisfying
  equation~\eqref{eq:1} with $\lambda>0$ and $H>0$, we have the
  following estimates.
  \begin{align*}
    &
    \int_\Sigma|\Acirc|^2 + |\nabla \log H|^2 \dmu
    \leq
    C \rmin^{-1},
    \\
    &
    \left| \int_\Sigma H^2\dmu - 16\pi \right|
    \leq 
    C \rmin^{-1},
    \intertext{and}
    &
    \lambda |\Sigma|
    \leq
     C\rmin^{-1}.
  \end{align*}
\end{lemma}
\begin{proof}
  From lemma~\ref{thm:init-integral} we get the bound
  \begin{equation*}
    \int_\Sigma H^2 \dmu
    \leq
    16 \pi - 2\int_\Sigma \ScalM\dmu
  \end{equation*}
  As $|\ScalM|\leq C(\eta)r^{-4}$ we find that in view of
  lemma~\ref{thm:integral-powers}
  \begin{equation*}
    \int_\Sigma H^2 \dmu
    \leq
    16 \pi + C \rmin^{-2} \int_\Sigma H^2\dmu.
  \end{equation*}
  So if $\rmin$ is large enough, eventually
  \begin{equation*}
    \int_\Sigma H^2 \dmu \leq 16\pi + C \rmin^{-2}.
  \end{equation*}
  We can write the Gauss equation (\ref{eq:14}) in the following form
  \begin{equation*}
    \half \ScalSig
    \leq
    \half \ScalSig + \half |\Acirc|^2
    = \frac{1}{4} H^2 + \half \ScalM - \RicM(\nu,\nu).
  \end{equation*}
  Integrating and using lemma~\ref{thm:integral-powers} gives
  \begin{equation*}
    16\pi \leq \int_\Sigma H^2\dmu + C\rmin^{-1}.
  \end{equation*}
  The remaining claims now follow from lemma~\ref{thm:init-integral}.
\end{proof}
The initial bound on $\Acirc$ derived above is crucial for higher
curvature estimates on $\Sigma$. We vary on the strategy outlined in
\cite[Section 2]{Kuwert-Schatzle:2001}. The estimates there were
derived in flat ambient space and therefore we review them here for
the readers convenience.  More importantly, we can use the fact that
$H>0$, which improves the estimates, as the absolute error is slightly
better behaved.
\begin{lemma}\label{thm:estimate_lapH}
  Under the assumtions of lemma~\ref{thm:initial-roundness} we have
  \begin{equation*}
    \int_\Sigma \frac{|\Delta H|^2}{H^2} \dmu
    \leq
    2 \int_\Sigma |\Acirc|^4\dmu
    +
    2 \int_\Sigma \big(\RicM(\nu,\nu) + \lambda \big)^2\dmu.
  \end{equation*}
\end{lemma}
\begin{proof}
  We use equation \eqref{eq:1}, divided by $H$, which gives
  \begin{equation*}
    \begin{split}
      \int_\Sigma \frac{|\Delta H|^2}{H^2}\dmu
      &=
      \int_\Sigma \big(|\Acirc|^2 + \RicM(\nu,\nu) +
      \lambda\big)^2\dmu
      \\
      &\leq
      2 \int_\Sigma |\Acirc|^4 \dmu +
      2\int_\Sigma\big(\RicM(\nu,\nu) +  \lambda\big)^2 \dmu.
    \end{split}
  \end{equation*}
\end{proof}
\begin{lemma}
  \label{thm:estimate_d2H}
  Under the assumtions of lemma~\ref{thm:initial-roundness} we have
  \begin{equation*}
    \begin{split}
      &
      \int_\Sigma \frac{|\nabla^2H|^2}{H^2}\dmu
      + \tfrac{1}{2}|\nabla H|^2 \dmu
      \\
      &\quad
      \leq
      C\rmin^{-3}\int_\Sigma |\nabla \log H|^2
      +
      \int_\Sigma
      \big(\RicM(\nu,\nu) +\lambda\big)^2
      + |\Acirc|^4
      + |\nabla \log H|^4 \dmu.
    \end{split}
  \end{equation*}
\end{lemma}
\begin{proof}
  \begin{equation}
    \label{eq:15}
    \begin{split}
      \int_\Sigma \frac{|\nabla^2 H|^2}{H^2}\dmu
      &=
      \int_\Sigma
      - H^{-2} \nabla_i\nabla_j\nabla_i H \nabla_j H
      +2 H^{-3} \nabla^2 H(\nabla H, \nabla H)
      \dmu
      \\
      &=
      \int_\Sigma
      - H^{-2} \nabla_j \Delta H \nabla_j H
      - H^{-2} \RiemSig_{ijki} \nabla_j H \nabla_k H
      \dmu
      \\
      &\phantom{=}
      + \int_\Sigma
      2  H^{-3} \nabla^2 H(\nabla H, \nabla H)
      \dmu
      \\
      &=
      \int_\Sigma
      \frac{|\Delta H|^2}{H^2}
      - H^{-2} \RiemSig_{ijki} \nabla_j H \nabla_k H
      \dmu
      \\
      &\phantom{=}
      + \int_\Sigma
      2 H^{-3} \nabla^2 H(\nabla H, \nabla H)
      - 2 H^{-3} |\nabla H|^2 \Delta H 
      \dmu.
    \end{split}
  \end{equation}
  In view of the Gauss equation (\ref{eq:4}) the curvature term yields 
  \begin{equation*}
    \begin{split}
      \RiemSig_{ijki}\nabla_j H \nabla_k H
      &=
      \big(
      \RiemM_{ijki}
      + \tfrac{1}{4} H^2 \gamma_{jk}
      - \Acirc_{ik}\Acirc_{ij}
      \big)
      \nabla_j H \nabla_k H
      \\      
      &=
      \frac{1}{4} H^2|\nabla H|^2
      + \RiemM_{ijki}\nabla_j H \nabla_k H
      - \Acirc_{ik}\Acirc_{ij}\nabla_j H \nabla_k H.
    \end{split}
  \end{equation*}
  Furthermore, we estimate
  \begin{equation*}
    \begin{split}
      &\int_\Sigma
      2 H^{-3} \nabla^2 H(\nabla H, \nabla H)
      - 2 H^{-3} |\nabla H|^2 \Delta H
      \dmu
      \\
      &\quad
      \leq
      \int_\Sigma \frac{1}{2} \frac{|\nabla^2 H|^2}{H^2} + c |\nabla
      \log H|^4 \dmu.
    \end{split}
  \end{equation*}
  The first term can be absorbed to the right hand side of
  equation~(\ref{eq:15}). We infer
  \begin{equation*}
    \begin{split}
      &\int_\Sigma
      \frac{1}{2}\frac{|\nabla^2 H|^2}{H^2}
      + \frac{1}{4} |\nabla H|^2
      \dmu
      \\
      &\quad
      \leq
      \int_\Sigma \frac{|\Delta H|^2}{H^2} + c |\nabla \log H|^4 + c|\Acirc|^4
      + c |\RiemM| |\nabla \log H|^2 \dmu.    
    \end{split}
  \end{equation*}
  We use $|\RiemM|\leq C\rmin^{-3}$ and lemma~\ref{thm:estimate_lapH}
  to conclude the claimed inequality.
\end{proof}
\begin{lemma}
  \label{thm:estimate_dA0}
 Under the assumtions of lemma~\ref{thm:initial-roundness} we have
  \begin{equation*}
    \int_\Sigma
    |\nabla \Acirc|^2 \dmu
    + \half H^2 |\Acirc|^2
    \dmu
    \leq
    \int_\Sigma |\omega|^2
    +
    C\rmin^{-3}\int_\Sigma |\Acirc|^2\dmu
    +
    \int_\Sigma
    |\nabla H|^2
    +
    |\Acirc|^4
    \dmu.
  \end{equation*}
\end{lemma}
\begin{proof}
  Integrate equation~(\ref{eq:8}), and use integration by parts on the left hand
  side, and on the first and the last term on the right hand side to
  conclude
  \begin{equation*}
  \begin{split}
   \int_\Sigma |\nabla \Acirc|^2 \dmu+\frac{1}{2}\int_\Sigma H^2|\Acirc|^2 \dmu
    =&
    \int_\Sigma
    2 \la \div \Acirc, \half \nabla H + \omega\ra
    + |\Acirc|^4\\
    &+ |\Acirc|^2 \RicM(\nu,\nu)
    - 2 \Acirc^{ij}\Acirc^l_j\RicM_{il}
    \dmu.
   \end{split}     
  \end{equation*}
  From the Codazzi equation we conclude that $\div \Acirc = \half
  \nabla H + \omega$, and hence
  \begin{equation*}
    \int_\Sigma |\nabla \Acirc|^2 \dmu+\frac{1}{2}\int_\Sigma H^2|\Acirc|^2 \dmu
    \leq
    \int_\Sigma
    |\nabla H|^2
    + |\Acirc|^4
    + 4|\omega|^2
    + c|\Acirc|^2|\RiemM|
    \dmu.
  \end{equation*}
  In view of $|\RiemM|\leq C\rmin^{-3}$ the claimed estimate follows.
\end{proof}
Combining lemma~\ref{thm:estimate_d2H} and
lemma~\ref{thm:estimate_dA0}, we infer the following estimate.
\begin{lemma}
  \label{thm:estimate-d2H-dA0-A0-pre}
  Under the assumtions of lemma~\ref{thm:initial-roundness} we have
  \begin{equation*}
    \begin{split}
      &\int_\Sigma
      \frac{|\nabla^2 H|^2}{H^2}
      + |\nabla A|^2
      + |A|^2 |\Acirc|^2
      \dmu
      \\
      &\quad
      \leq
      c\int_\Sigma |\omega|^2 + \big(\RicM(\nu,\nu) +\lambda)^2 \dmu
      + c \int_\Sigma
      |\Acirc|^4 + |\nabla \log H|^4
      \dmu
      \\
      &\quad
      + C\rmin^{-3}\int_\Sigma |\nabla \log H|^2 + |\Acirc|^2 \dmu.
    \end{split}
  \end{equation*}
\end{lemma}
At this point we need a variation on the multiplicative Sobolev
inequality from \cite[Lemma 2.5]{Kuwert-Schatzle:2001}.
\begin{lemma}
  \label{thm:mult-sobolev}
Under the assumtions of lemma~\ref{thm:initial-roundness} we have
  \begin{equation*}
    \begin{split}
      &\int_\Sigma
      |\Acirc|^4
      + |\nabla \log H|^4
      \dmu
      \\
      &      
      \leq
      c \left( \int_\Sigma |\Acirc|^2 + |\nabla \log H|^2 \dmu \right)
      \\
      & \qquad \cdot
      \left(
        \int_\Sigma
        \frac{|\nabla^2 H|^2}{H^2}
        + |\nabla A|^2
        + |\nabla \log H |^4
        + H^2|\Acirc|^2
        \dmu
      \right).
    \end{split}
  \end{equation*}
\end{lemma}
\begin{proof}
 We use the Michael-Simon-Sobolev inequality from Proposition
  \ref{thm:sobolev} and H\"older's inequality to estimate
  \begin{equation*}
    \begin{split}
      &
      \left(\int_\Sigma \big(|\nabla \log H|^2\big)^2 \dmu
      \right)^{1/2}
      \\
      &\quad
      \leq
      c \int_\Sigma \frac{|\nabla^2 H|}{H} |\nabla \log H| +
      |\nabla \log H|^3 + H |\nabla \log H|^2 \dmu
      \\
      &\quad
      \leq
      c \left(\int_\Sigma |\nabla \log H|^2\dmu\right)^{1/2}
      \left(
        \int_\Sigma
        \frac{|\nabla^2 H|^2}{H^2}
        + |\nabla \log H |^4
        + |\nabla H|^2
        \dmu
      \right)^{1/2}.
    \end{split}
  \end{equation*}
  Furthermore
  \begin{equation*}    
    \begin{split}
      \left( \int_\Sigma |\Acirc|^4\dmu \right)^{1/2}
      &
      \leq
      c 
      \int_\Sigma |\Acirc| |\nabla \Acirc| + H |\Acirc|^2
      \dmu
      \\
      &\leq
      c 
      \left( \int_\Sigma |\Acirc|^2 \dmu \right)^{1/2}
      \left( \int_\Sigma |\nabla \Acirc|^2 + H^2 |\Acirc|^2\dmu \right)^{1/2}.
    \end{split}
  \end{equation*}
  Combining both inequalities yields the claim.
\end{proof}
The estimates above yield the initial curvature estimates.
\begin{theorem}
  \label{thm:curv-est-1}
  For every $m,\eta,\sigma$ there exist constants
  $r_0=r_0(m,\eta,\sigma)$ and $C=C(m,\eta,\sigma)$ with the following
  properties:

  If $(M,g)$ is $(m,\eta,\sigma)$-asymptotically
  Schwarzschild and $\Sigma \subset M \setminus B_{r_0}$ satisfies
  equation~(\ref{eq:1}) with $H>0$ and $\lambda>0$, then $\Sigma$
  satisfies the estimate
  \begin{equation*}
    \begin{split}
      &\int_\Sigma \frac{|\nabla^2 H|^2}{H^2} + |\nabla A|^2 + |\nabla
      \log H|^4 + |A|^2|\Acirc|^2 \dmu
      \\
      &\quad\leq c \int_\Sigma |\omega|^2 + \big(\RicM(\nu,\nu) +
      \lambda \big)^2 \dmu + C \rmin^{-3} \int_\Sigma|\nabla\log H|^2
      + |\Acirc|^2\dmu
    \end{split}
  \end{equation*}
\end{theorem}
\begin{proof}
  This is a consequence of lemma~\ref{thm:initial-roundness}, lemma \ref{thm:estimate-d2H-dA0-A0-pre} and lemma~\ref{thm:mult-sobolev}. 
\end{proof}
\begin{corollary}
  \label{thm:A0-decay}
  Under the assumptions of theorem~\ref{thm:curv-est-1} we have the
  estimate
  \begin{equation*}  
    \int_\Sigma
    \frac{|\nabla^2 H|^2}{H^2}
    + |\nabla A|^2
    + |\nabla \log H|^4
    + |A|^2|\Acirc|^2
    \dmu
    \leq
    C \rmin^{-4} + C\rmin^{-2} |\Sigma|^{-1}
  \end{equation*}
\end{corollary}
\begin{proof}
  The claim follows in view of $|\omega| + |\Ric| \leq C r^{-3}$,
  lemma~\ref{thm:integral-powers} and the estimates from
  lemma~\ref{thm:initial-roundness}.
\end{proof}
\begin{corollary}
\label{thm:A0-l2-1}
  Under the assumptions of theorem \ref{thm:curv-est-1}, we have the
  estimate
  \begin{equation*}
    \int_\Sigma |\Acirc|^2 + |\nabla \log H|^2 \dmu
    \leq
    C \rmin^{-4}|\Sigma| + C\rmin^{-2}.
  \end{equation*}
\end{corollary}
\begin{proof}
  This follows from the Michael-Simon-Sobolev inequality and Kato`s inequality. For example
  \begin{equation*}
    \begin{split}
      \left(\int_\Sigma |\Acirc|^2\dmu \right)^{1/2}
      &
      \leq    
      C_s \int_\Sigma \big|\nabla |\Acirc|\big| + H |\Acirc| \dmu
      \\
      &
      \leq
      c C_s |\Sigma|^{1/2}
      \left(
        \int_\Sigma
        |\nabla \Acirc|^2 + H^2 |\Acirc|^2
        \dmu
      \right)^{1/2}      
    \end{split}
  \end{equation*}
  Using corollary \ref{thm:A0-decay} the claimed inequality for
  $\int|\Acirc|^2\dmu$ follows. The proof for $\int|\nabla\log H|^2
  \dmu$ is similar.
\end{proof}


%% file: improved.tex
\section{Improved curvature estimates}
\label{sec:impr-curv-estim}
Before we can approach the position estimates, we discuss how the
decay rates in the curvature estimates in
section~\ref{sec:integr-curv-estim} can be improved. First we note
that the estimates in section~\ref{sec:integr-curv-estim} and
theorem~\ref{thm:umbilical} imply that solutions to
equation~\eqref{eq:1} are close to spheres.
\begin{proposition}
  \label{thm:sphere-approx-1}
  Let $R_e$ be the geometric area radius of $\Sigma$ with respect to
  the Euclidean metric, i.e. $\int_\Sigma\dmu^e = 4\pi R_e^2$, and let $a_e$ be
  the Euclidean center of gravity of $\Sigma$, that is
  \begin{equation*}
    a_e = \frac{\int_\Sigma \id_\Sigma \dmu^e}{\int_\Sigma \dmu^e}.
  \end{equation*}
  Let $S:=S_{R_e}(a_e)$ be the sphere of radius $R_e$ centered at $a_e$ and let
  $N$ be the Euclidean normal of $S$. Then there exists a conformal
  parameterization $\psi : S \to (\Sigma,\gamma^e)$ with conformal
  factor $h^2$ satisfying the following estimates.
  \begin{align}
    \label{eq:19}
    &\sup_S |\psi -\id_S| \leq C R_e \big( \|\Acirc\|_{L^2} + \eta\rmin^{-2} \big)
    \\
    \label{eq:20}
    &\| N \circ \id_S -\nu^e \circ \psi \|_{L^2(S)}
    \leq
    C R_e \big( \|\Acirc\|_{L^2} + \eta\rmin^{-2} \big)
    \\
    \label{eq:21}
    &\sup_S | h^2 - 1 |
    \leq
    C \big( \|\Acirc\|_{L^2} + \eta\rmin^{-2} \big)
  \end{align}
\end{proposition}
\begin{proof}
  This follows immediately from corollary~\ref{koro:acirc}, theorem~\ref{thm:umbilical} and corollary \ref{thm:A0-decay}.
\end{proof}
In the sequel, an essential quantity will be the ratio between the
center of mass and the radius of the approximating sphere. We denote it by
\begin{equation}
  \label{eq:22}
  \tau : = \frac{|a_e|}{R_e},
\end{equation}
where $a_e$ and $R_e$ are as in proposition~\ref{thm:sphere-approx-1}. Note that by corollary \ref{thm:A0-decay} and \eqref{eq:19} we have
\begin{align}
  \rmin
  &\geq
  R_e - |a_e| - CR_e(||\Acirc||_{L^2} + \eta \rmin^{-2}) \nonumber \\
  &\geq
  R_e(1-\tau) - C_1 R_e (R_e \rmin^{-2} + \rmin^{-1} + \eta\rmin^{-2}).\label{comparer}
\end{align}
Analogously we can estimate $\rmin$ from above. If we now assume that
\begin{equation}
  \label{eq:23}
  \tau \leq (1-\eps) \quad \text{and}\quad R_e\leq \frac{\eps}{4C_1}\rmin^{2}
\end{equation}
for some arbitrary $\eps>0$, we get for $\rmin$ large enough  
\begin{align*} 
  C_1(R_e \rmin^{-2} + \rmin^{-1} + \eta\rmin^{-2}) \leq \frac{\eps}{2}
\end{align*} 
and this shows that 
\begin{equation*}
  C^{-1} \rmin \leq R_e \leq C \rmin.
\end{equation*}
Hence $R_e$ and $\rmin$ are comparable to each other and
therefore we will not distinguish between them any more and we phrase
the estimates only in terms of $\rmin$. Constants $C$ in this section
will also depend on $\eps$.

We can use the fact that $\Sigma$ is well approximated by a round
sphere to compute a precise expression for $\lambda$.
\begin{proposition}
  \label{thm:compute_lambda}
  If $(M,g)$ and $\Sigma$ are as in theorem~\ref{thm:curv-est-1}, then
  if assumption~\eqref{eq:23} holds, we have
  \begin{equation*}
    \left| \lambda - \frac{2m}{R_S^3} \right|
    \leq
    C\rmin^{-2} \big(\|\Acirc\|_{L^2}^2 + \|\nabla \log H\|_{L^2}^2\big)
    +
    C\rmin^{-4}(\tau + \rmin\|\Acirc\|_{L^2} + \eta\rmin^{-1}) 
  \end{equation*}
  Here we set $R_S := \bar\phi^2 R_e$ where $\bar\phi = \phi(R_e) = 1
  + \frac{m}{2R_e}$.
\end{proposition}
\begin{proof}
  Recall that from \eqref{eq:2} we
  have
  \begin{equation}
    \label{eq:72}
    \left| \lambda|\Sigma|  + \int_\Sigma \Ric(\nu,\nu) \dmu \right|
    \leq \int_\Sigma |\Acirc|^2 + |\nabla \log H |^2 \dmu
  \end{equation}
  The goal is now to calculate the integral on the left. We
  start by estimating the error to the respective integral in
  Schwarzschild.
  \begin{equation*}
    \begin{split}
      &\left|
        \int_\Sigma \Ric(\nu,\nu) \dmu
        -
        \int_\Sigma \Ric^S(\nu^S,\nu^S)\dmu^S
      \right|
      \\
      &\quad \leq
      c \int_\Sigma |\Ric-\Ric^S| + |\Ric| |\nu-\nu^S| + |\Ric|
      |\dmu-\dmu^S| \dmu
      \leq
      C\eta\rmin^{-2}.
    \end{split}
  \end{equation*}
  We furthermore replace $\nu^S$ and $\dmu^S$ by the respective
  Euclidean quantities. This introduces some factors of $\phi$ which
  all cancel, and we therefore get no further error in the following
  step:
  \begin{equation*}
    \left|
      \int_\Sigma \Ric(\nu,\nu) \dmu
      -
      \int_\Sigma \Ric^S(\nu^e,\nu^e)\dmu^e
    \right|
    \leq
    C\eta\rmin^{-2}.
  \end{equation*}    
  The second integral on the left can be replaced by an integration
  over the sphere $S=S_{R_e}(a_e)$ from
  proposition~\ref{thm:sphere-approx-1}, introducing only acceptable
  error terms. This technique was used extensively
  in~\cite{Metzger:2007ce}. To see how this works, we use the
  parameterization $\psi:S \to \Sigma$ from
  proposition~\ref{thm:sphere-approx-1} to calculate
  \begin{equation*}
    \begin{split}
      &\left|
        \int_\Sigma \Ric^S(\nu^e,\nu^e)\dmu^e
        -
        \int_S \Ric^S(N,N)\dmu^e
      \right|
      \\
      &\quad=
      \left|
        \int_S
        (\Ric^S\circ\psi)\big(\nu^e\circ\psi,\nu^e\circ\psi\big) h^2
        -
        \Ric^S(N,N)\dmu^e
      \right|
      \\
      &\quad
      \leq c
      \int_S |\Ric^S\circ\psi - \Ric^S| + |\Ric^S| |h^2 -1|  +
      |\Ric^S | |\nu^e\circ\psi - N| \dmu^e
      \\
      &\quad
      \leq c
      \| \nabla^e \Ric^S \|_{L^\infty} \|\psi - \Id \|_{L^\infty}|\Sigma|
      + c
      \|\Ric^S\|_{L^1} \|h^2 -1 \|_{L^\infty}
      \\
      &\quad
      \phantom{\leq}
      + c
      \|\Ric^S\|_{L^2} \|\nu^e\circ\psi - N\|_{L^2}
      \\
      &\quad
      \leq
      C\rmin^{-2}(\rmin\|\Acirc\|_{L^2} + \eta \rmin^{-1})
    \end{split}
  \end{equation*}    
  Now use coordinates $\varphi,\vartheta$ on $S_R(a)$ such that $\cos
  \varphi = g^e(\frac{a_e}{|a_e|}, N)$. Then the representation
  $\Ric^S(N,N) = \phi^{-2}\frac{m}{r^3} ( 1 - 3 g^e(\rho,N)^2)$
  together with $\rho = r^{-1}(R_eN + a_e)$, implies that
  \begin{equation*}
    \begin{split}
      &\int_S \Ric^S(N,N) \dmu^e
      \\
      &\quad
      =
      m \int_S \phi^{-2}\left(
        \frac{1}{r^3}
        - 3R_e^2 \frac{1}{r^5}
        - 6 R_e |a_e| \frac{\cos\varphi}{r^5}
        - 3 |a_e|^2 \frac{\cos^2\varphi}{r^5}
      \right)
      \dmu^e
    \end{split}
  \end{equation*}
  Letting $\bar\phi := 1 + \frac{m}{2R_e}$ we can use the expression
  $r^2 = R_e^2 + 2R_e|a_e|\cos\varphi + |a_e|^2$ to estimate that
  \begin{equation*}
    \sup_S |\phi -\bar\phi| \leq C\tau\rmin^{-1} 
  \end{equation*}
  which renders  
  \begin{equation*}
    \begin{split}
      &\int_S \Ric^S(N,N) \dmu^e
      \\
      &\quad
      =
      \frac{m}{\bar\phi^{2}} \int_S \left(
        \frac{1}{r^3}
        - 3R_e^2 \frac{1}{r^5}
        - 6 R_e |a_e| \frac{\cos\varphi}{r^5}
        - 3 |a_e|^2 \frac{\cos^2\varphi}{r^5}
      \right)
      \dmu^e
      +
      O(\tau\rmin^{-2})
    \end{split}
  \end{equation*} 
  Integrals of this type can be computed explicitly as follows. First
  write
  \begin{equation*}
    \int_S \frac{\cos^l\varphi}{r^k} \dmu^e
    =
    2\pi R_e^2 \int_0^\pi \sin\varphi \frac{\cos^l\varphi}{r^k} d\varphi.
  \end{equation*}
  We have $x = R_eN + a_e$, and hence $r = \sqrt{R_e^2 + 2R_e|a_e|\cos\varphi +
    |a_e|^2}$. Thus $\frac{d\varphi}{dr} = -\frac{r}{R_e|a_e|\sin\varphi}$,
  and $\cos\varphi = \frac{r^2 - R_e^2 - |a_e|^2}{2R_e|a_e|}$. Substituting
  this into the integral yields
  \begin{equation*}
    2\pi R_e^2 \int_0^\pi \sin\varphi \frac{\cos^l\varphi}{r^k} d\varphi.
    =
    \frac{2\pi R_e}{|a_e|} (2R_e|a_e|)^{-l}
    \int_{|R_e-|a_e||}^{R_e+|a_e|}
    r^{1-k} (r^2 -R_e^2 - |a_e|^2)^l dr.
  \end{equation*}
  Thus we can compute (see appendix $A.1$), if $|a_e|<R_e$,
  \begin{equation*}
    \int_S \Ric^S(N,N) \dmu^e
    =
    - \bar\phi^{-2} \frac{8\pi m}{R_e}
    + O(\tau\rmin^{-2})
  \end{equation*}
  Collecting all error terms we introduced, this yields that
  \begin{equation*}
    \left| \int_\Sigma \Ric(\nu,\nu) \dmu + \frac{8\pi m}{R_S} \right|
    \leq 
    C \rmin^{-2}\big(\tau + \rmin\|\Acirc\|_{L^2} + \eta\rmin^{-1}\big) 
  \end{equation*}
  The next step is to calculate the area of $\Sigma$. Similar to the above
  argument we estimate
  \begin{equation*}
    \left| \int_\Sigma 1 \dmu - \int_\Sigma 1 \dmu^S \right|
    \leq
    C\eta.
  \end{equation*}
  From lemma~\ref{thm:geometry-in-schwarzschild} we get
  \begin{equation*}
    \int_\Sigma 1 \dmu^S
    =
    \int_\Sigma \phi^4 \dmu^e.
  \end{equation*}
  We now replace $\phi$ by $\bar\phi$ in this integral. This yields an
  error of the following form
  \begin{equation*}
    \left|
      \int_\Sigma \phi^4 \dmu^e
      -
      \int_\Sigma \bar\phi^4 \dmu^e
    \right|
    \leq
    C\rmin\big(\tau + \rmin\|\Acirc\|_{L^2} + \eta\rmin^{-1}\big)
  \end{equation*}
  In conclusion we find that
  \begin{equation*}
    \big| |\Sigma| - 4\pi R_S^2 \big|
    \leq
    C\rmin\big(\tau + \rmin\|\Acirc\|_{L^2} + \eta\rmin^{-1}\big)
  \end{equation*}
  Using lemma~\ref{thm:initial-roundness} we get
  \begin{equation*}
    \big| \lambda |\Sigma| - 4\pi R_S^2\lambda \big|
    \leq
    C\rmin^{-2}\big(\tau + \rmin\|\Acirc\|_{L^2} + \eta\rmin^{-1}\big)
  \end{equation*}
  Plugging this expression into equation~\eqref{eq:72}, we arrive at
  the estimate
  \begin{equation}
    \label{eq:73}
    \left| \lambda - \frac{2m}{R_S^3} \right|
    \leq
    C\rmin^{-2} \big(\|\Acirc\|_{L^2}^2 + \|\nabla \log H\|_{L^2}^2\big)
    +
    C\rmin^{-4}\big(\tau + \rmin\|\Acirc\|_{L^2} + \eta\rmin^{-1}\big) 
  \end{equation}
  This yields the claim.
\end{proof}
If $\tau$ behaves as above, we have more control over the curvature
terms which did not allow us to increase the decay rates in
section~\ref{sec:integr-curv-estim}. In particular,
\begin{proposition}
  \label{thm:decay-curvature}
  Under the assumptions of theorem~\ref{thm:curv-est-1}, if
  conditions~\eqref{eq:23} hold, then
  \begin{align*}
    \| \nu -\phi^{-2}\rho \|^2_{L^2(\Sigma)}    
    &\leq
    C\rmin^2(\tau^2 + \|\Acirc\|_{L^2}^2 + \eta\rmin^{-2})
    \\
    \|\Ric(\nu,\nu) - \phi^{-4} \Ric^S(\rho,\rho)\|^2_{L^2(\Sigma)}
    & \leq
    C\rmin^{-4}(\tau^2 + \|\Acirc\|_{L^2}^2 + \eta\rmin^{-2})
    \\
    \|\omega\|^2_{L^2(\Sigma)}
    & \leq
    C\rmin^{-4}(\tau^2 + \|\Acirc\|_{L^2}^2 + \eta\rmin^{-2})
    \\
    \|\Ric^T - P^S_{\phi^{-2} \rho}\Ric^S\|^2_{L^2(\Sigma)}
    & \leq
    C\rmin^{-4}(\tau^2 + \|\Acirc\|_{L^2}^2 + \eta\rmin^{-2})
  \end{align*}
  Here, $P^S_{\phi^{-2} \rho}\Ric^S$ denotes the $g^S$-orthogonal
  projection of $\Ric^S$ to the subspace perpendicular to $\phi^{-2}
  \rho$. 
\end{proposition}
\begin{proof}
  The proof is the similar to \cite[Proposition
  4.6]{Metzger:2007ce}. However the claimed estimate here is somewhat
  more precise, so we briefly sketch the argument. To show the first
  assertion we first replace the quantities in the integral by the
  respective quantities computed with respect to the Schwarzschild
  metric
  \begin{equation*}
    \left| \int_\Sigma g(\nu -\phi^{-2}\rho,\nu -\phi^{-2}\rho)\dmu
      - \int_\Sigma g^S(\nu^S -\phi^{-2}\rho,\nu^S
      -\phi^{-2}\rho)\dmu^S\right|
    \leq C\eta.
  \end{equation*}
  Then we note that
  \begin{equation*}
    \int_\Sigma g^S(\nu^S -\phi^{-2}\rho,\nu^S
    -\phi^{-2}\rho)\dmu^S
    =
    \int_\Sigma g^e(\nu^e -\rho,\nu^e
    -\rho)\dmu^e.
  \end{equation*}
  We now parameterize again by $\psi$ and calculate the difference to
  the respective quantity on $S$. We obtain
  \begin{equation*}
    \begin{split}
      &\left|
        \int_\Sigma g^e(\nu^e -\rho,\nu^e -\rho)\dmu^e
        -
        \int_S g^e(N-\rho, N-\rho)\dmu^e
      \right|
      \\
      &\quad\leq
      C(\tau \rmin^2\|\Acirc\|_{L^2} + \rmin^2 \|\Acirc\|_{L^2}^2 +
       \tau^2 + \tau\eta +\eta^2\rmin^{-2})
    \end{split}
  \end{equation*}
  Since
  \begin{align*}
    \int_S g^e(N-\rho, N-\rho)\dmu^e \le& C\int_S r^{-2}(|r-R_e|^2+|a_e|^2)\dmu^e\\
    \le&  C\rmin^2(\tau^2  +\|\Acirc\|_{L^2}^2 + \eta\rmin^{-2}), 
  \end{align*}
  where we used \eqref{comparer}, we obtain the first inequality.

  The other inequalities are then a consequence of the first, since
  they basically follow from expressing the quantities in
  terms of the respective quantities in Schwarzschild.
\end{proof}
This proposition can be used to improve the mean value estimate we obtained
in proposition~\ref{thm:compute_lambda} to the following $L^2$-estimate.
\begin{proposition}
\label{thm:ric-lambda-l2-est}
  Under the assumptions of theorem~\ref{thm:curv-est-1}, if
  conditions~\eqref{eq:23} hold, we have
  \begin{equation*}
    \| \lambda + \Ric(\nu,\nu) \|_{L^2(\Sigma)}
    \leq
    C\rmin^{-2}\big(\tau + \|\Acirc\|_{L^2} + \|\nabla\log H \|_{L^2} + \eta\rmin^{-1}\big)
  \end{equation*}
\end{proposition}
\begin{proof}
  We use the second estimate of proposition~\ref{thm:decay-curvature}
  to express $\Ric(\nu,\nu)$ in terms of
  $\phi^{-4}\Ric^S(\rho,\rho)$ plus error. Then we use that
  up to second order $\phi^{-4}\Ric^S(\rho,\rho) = -\frac{2m}{R_S^3}$ plus
  error. In combination with proposition~\ref{thm:compute_lambda} this
  yields the estimate.
\end{proof}
Propositions~\ref{thm:decay-curvature} and~\ref{thm:ric-lambda-l2-est}
give more precise estimates of the terms on the right hand side of
theorem~\ref{thm:curv-est-1}. In combination with the initial estimate
for $\|\Acirc\|_{L^2}$ we thus infer the following improved curvature
estimates.
\begin{theorem}
  \label{thm:improved-curvature-est}
  Under the assumptions of theorem~\ref{thm:curv-est-1}, if
  conditions~\eqref{eq:23} hold, then
  \begin{equation*}
    \int_\Sigma
    \frac{|\nabla^2 H|^2}{H^2}
    + |\nabla A|^2
    + |\nabla \log H|^4
    + |A|^2 |\Acirc|^2
    \dmu
    \leq
    C\rmin^{-4}\big(\tau^2 +  \eta\rmin^{-2}\big)
  \end{equation*}
  and furthermore
  \begin{equation*}
    \int_\Sigma |\Acirc|^2 + |\nabla \log H|^2 \dmu
    \leq
    C\rmin^{-2}\big(\tau^2 + \eta\rmin^{-2}\big)
  \end{equation*}
\end{theorem}
\begin{proof}
  First of all note that by the calculation in
  corollary~\ref{thm:A0-l2-1} we can estimate
  \begin{equation}
    \label{eq:74}
    \int_\Sigma |\Acirc|^2 + |\nabla\log H|^2 \dmu
    \leq
    C |\Sigma| \int_\Sigma \frac{|\nabla^2 H|^2}{H^2} + |\nabla A|^2 +
    |\nabla \log H|^4 + |A|^2|\Acirc|^2 \dmu.
  \end{equation}
  Since, under assumption~\eqref{eq:23} we have that
  $|\Sigma|\rmin^{-3} \to 0$, we can eventually absorb the second term
  on the right in theorem~\ref{thm:curv-est-1} to the left hand
  side. In combination with proposition~\ref{thm:decay-curvature} and proposition \ref{thm:ric-lambda-l2-est} this yields that
  \begin{equation}
\label{eq:74a}
    \begin{split}
      &\int_\Sigma
      \frac{|\nabla^2 H|^2}{H^2}
      + |\nabla A|^2
      + |\nabla \log H|^4
      + |A|^2 |\Acirc|^2
      \dmu
      \\
      &\quad
      \leq
      C\rmin^{-4}\big(\tau^2 + \eta\rmin^{-2} + \|\Acirc\|_{L^2}^2 + \|\nabla\log H\|^2_{L^2} \big)    
    \end{split}
  \end{equation}
  together with~\eqref{eq:74} we infer
  \begin{equation*}
    \int_\Sigma |\Acirc|^2 + |\nabla\log H|^2 \dmu
    \leq
    C\rmin^{-2}\big(\tau^2 + \eta\rmin^{-2} + \|\Acirc\|_{L^2}^2 + \|\nabla\log H\|^2_{L^2} \big)    
  \end{equation*}
  We absorb $\|\Acirc\|_{L^2}^2 + \|\nabla\log H\|^2_{L^2}$ to the left
  and obtain the second estimate. The first estimate follows
  from~\eqref{eq:74a} and this estimate.
\end{proof}
Using this estimate, we also get a better control on derivatives of
$\omega$. In particular, we have the following 
\begin{proposition}
  \label{thm:nabla-omega-small}
  Under the assumptions of theorem~\ref{thm:curv-est-1}, if
  conditions~\eqref{eq:23} hold, then
  \begin{equation*}
    \| \nabla \omega \|^2_{L^2(\Sigma)}
    \leq
    C\rmin^{-6}\big(\tau^2 + \eta\rmin^{-2} \big),
  \end{equation*}
  and 
  \begin{equation*}
    \|\nabla \Ric(\nu,\nu) \|^2_{L^2(\Sigma)}
    \leq
    C\rmin^{-6}\big(\tau^2 + \eta\rmin^{-2} \big),
  \end{equation*}
\end{proposition}
\begin{proof}
  To prove the first estimate calculate for $\{e_i\}$ a ON-frame on
  $\Sigma$ that
  \begin{equation}
    \label{eq:70}
    \begin{split}
      \nabla_{e_i} \omega(e_k)
      &=
      e_i(\RicM(\nu,e_k)) - \RicM(\nu,\nabSig_{e_i}e_k)
      \\
      &=
      \nabM_{e_i} \RicM(\nu,e_k)
      + \half H \RicM(e_i,e_k)
      - \half H \RicM(\nu,\nu)
      \\
      &\phantom{=}
      + \RicM(e_l,e_k) \Acirc_{il}
      - \RicM(\nu,\nu)\Acirc_{ik}.
    \end{split}
  \end{equation}
  The last two terms including $\Acirc$ have the claimed decay, so
  we focus on the first three terms. 
  
  In Schwarzschild we have that on the centered spheres
  $\nabla^S\omega^S$ vanishes as $\omega^S$ vanishes, so we find that
  on centered spheres for a ON-frame $\{e_i^S\}$ tangent to the
  centered spheres
  \begin{equation}
   \label{eq:69}
    0
    =
    \nabla_{e_i^S} \omega^S
    =
    \nabla^S_{e_i^S} \Ric^S(\phi^{-2}\rho, e^S_k)
    + 
    \half H^S \Ric^S(e^S_i,e^S_k)
    -
    \half H \Ric^S(\phi^{-2}\rho,\phi^{-2}\rho).
  \end{equation}
  Following proposition~\ref{thm:decay-curvature} we get that the
  first three terms of \eqref{eq:69} equal the right hand side of
  \eqref{eq:70} up to an error with $L^2$-norm bounded by
  $C\tau\rmin^{-3}$. This yields the first estimate. The second one is
  proved similarly.
\end{proof}
In the sequel we will use the improved integral estimates to derive
improved pointwise estimates of the second fundamental form and its
derivatives. Before doing this we need the following Lemma which is
due to Kuwert and Sch\"atzle \cite{Kuwert-Schatzle:2001} in the case
that $M=\IR^n$.
\begin{lemma}
  \label{thm:interpolation}
  Under the assumptions of theorem~\ref{thm:curv-est-1} we have for
  every smooth form $\varphi$ along $\Sigma$
  \begin{align}
    \|\varphi\|_{L^\infty(\Sigma)}^4
    \le
    C \|\varphi\|_{L^2(\Sigma)}^2\int_\Sigma (|\nabla^2 \varphi|^2 +|H|^4 |\varphi|^2)\dmu . 
    \label{inter1}
  \end{align}
\end{lemma}
\begin{proof}
  The proof of lemma 2.8 in \cite{Kuwert-Schatzle:2001} can be carried over to our situation since we saw in proposition \ref{thm:sobolev} that the Michael-Simon Sobolev inequality remains
  unchanged if $(M,g)$ is $(m,\eta,\sigma)$-asymptotically Schwarzschild.
\end{proof}
In the next lemma we derive an $L^2$-estimate for $\nabla^2H$.
\begin{lemma}
  \label{thm:l2H}
  Under the assumptions of theorem~\ref{thm:curv-est-1}, if
  conditions~\eqref{eq:23} hold, then
  \begin{align}
    \label{eq:l2H1}
    \int_\Sigma |\nabla^2 H|^2
    \dmu
    \leq
    C \rmin^{-4}\big(\|H\|_{L^\infty}^2 + \rmin^{-2} \big)
    \big (\tau^2 +\eta\rmin^{-2} \big)
    . 
  \end{align}
\end{lemma} 
\begin{proof}
  We multiply equation \eqref{eq:1} with $\Delta H$ and integrate to get
  \begin{equation}
    \begin{split}
      &\int_\Sigma |\Delta H|^2 \dmu
      =
      - \int_\Sigma
      H\Delta H (|\Acirc|^2 + \RicM(\nu,\nu) + \lambda)  \dmu 
      \\
      &\quad
      \leq
      \half \int_\Sigma |\Delta H|^2 \dmu
      + c \int_\Sigma
      H^2|\Acirc|^4
      + H^2 \big(\RicM(\nu,\nu) + \lambda \big) \dmu
    \end{split}
  \end{equation}
  Defining $f=|\Acirc|^2|H|$ and applying proposition \ref{thm:sobolev}
  we get
  \begin{align*}
    &\left(\int_\Sigma |\Acirc|^4 H^2\dmu\right)^{1/2}
    \leq
    C\int_\Sigma (|A||\Acirc||\nabla A| + |\Acirc|^2 H^2)\dmu
    \\    
    &\quad
    \leq
    C\left( \int_\Sigma |A|^2 |\Acirc|^2 \dmu\right)^{1/2}
    \left(\int_\Sigma |\nabla A|^2 + H^2 |\Acirc|^2\dmu\right)^{1/2}
  \end{align*}
  In combination, we infer 
  \begin{equation*}
    \begin{split}
      \int_\Sigma |\Delta H|^2 \dmu
      &
      \leq \int_\Sigma H^2 \big(\RicM(\nu,\nu) + \lambda \big)^2\dmu
      \\
      &\quad
      + 
      C\left( \int_\Sigma |A|^2 |\Acirc|^2 \dmu\right)
      \left(\int_\Sigma |\nabla A|^2 + H^2 |\Acirc|^2\dmu\right)
    \end{split}
  \end{equation*}
  This implies the claim, since the first term is estimated in view of
  proposition~\ref{thm:ric-lambda-l2-est} and the second one in view
  of theorem~\ref{thm:improved-curvature-est}.  Using the Bochner
  identity as in the proof of lemma \ref{thm:estimate_d2H} finishes
  the proof.
\end{proof}
Now we are in a position to prove a pointwise estimate for $H$. 
\begin{proposition}
  \label{thm:pointwiseH}
  Let $S=S_{R_e}(a_e)$ be the approximating sphere for $\Sigma$ from
  proposition~\ref{thm:sphere-approx-1}. As in
  proposition~\ref{thm:compute_lambda} we let $\bar\phi = 1 +
  \frac{m}{2R_e}$ and define
  \begin{equation*}
    \bar{H}^S
    =
    \bar{\phi}^{-2} \frac{2}{R_e}  - 2\bar{\phi}^{-3}\frac{m}{R_e^2}
  \end{equation*}
  Under the assumptions of theorem~\ref{thm:curv-est-1}, if
  conditions~\eqref{eq:23} hold, we have that
  \begin{align}
    \label{eq:pointwiseH1}
    \|H-\bar{H}^S\|_{L^\infty(\Sigma)}
    \leq
    C\rmin^{-2}\big(\tau + \sqrt{\eta}\rmin^{-1}\big).
  \end{align}
\end{proposition}
\begin{proof}
  Since
  \begin{equation*}
    \| H - H^S \|^2_{L^2(\Sigma)} \leq C\eta^2\rmin^{-4}
  \end{equation*}
  and $H^S = \phi^{-2} H^e -\frac{2m}{r^2}\phi^{-3}g^e(\rho,\nu^e)$ by
  lemma~\ref{thm:geometry-in-schwarzschild}, we can estimate using
  propositions~\ref{thm:sphere-approx-1}, \ref{thm:decay-curvature} and theorem \ref{thm:improved-curvature-est} that
  \begin{align*}
   \| H^S - \bar{H}^S \|^2_{L^2(\Sigma)}\le&\ C (\|\phi^{-2}(H^e-\frac{2}{R_e})\|^2_{L^2(\Sigma)}+ \|(\phi^{-2}-\bar{\phi}^{-2})\frac{2}{R_e}\|^2_{L^2(\Sigma)}\\
&+\|(\phi^{-3}-\bar{\phi}^{-3})\frac{2m}{R^2_e}\|^2_{L^2(\Sigma)}\\
&+\|\phi^{-3}(\frac{2m}{r^2}g^e(\rho,\nu^e)-\frac{2m}{R^2_e})\|^2_{L^2(\Sigma)})\\
\le&\ C\| \Acirc\|^2_{L^2(\Sigma)}+C\tau^2 \rmin^{-2} + C \eta \rmin^{-4}\\
\le&\ C\rmin^{-2}\big(\tau^2 + \eta\rmin^{-2}\big).
\end{align*}
Combining these two estimates we conclude
 \begin{equation*}
    \|H - \bar{H}^S \|^2_{L^2(\Sigma)}
    \leq
    C\rmin^{-2}\big(\tau^2 + \eta\rmin^{-2}\big).
  \end{equation*}
  We apply lemma~\ref{thm:interpolation} to $\varphi=H-\bar{H}^S$
  and get
\begin{align}
   \|H-\bar{H}^S\|^4_{L^\infty(\Sigma)}
      &\leq
      C\|H-\bar{H}^S\|^2_{L^2(\Sigma)}
      \left(\int_\Sigma (|\nabla^2 H|^2+ H^4 |H-\bar{H}^S|^2\dmu\right)\nonumber \\
      &= I+II. \label{short}
\end{align}
Now we estimate term by term. We use lemma \ref{thm:l2H} and the fact that $\|H\|_{L^\infty(\Sigma)} \leq \|\bar{H}^S\|_{L^\infty(\Sigma)} + \|H-
  \bar{H}^S\|_{L^\infty(\Sigma)}$ to get
\begin{align*}
I\le&\ C\rmin^{-4}(\|H\|_{L^\infty(\Sigma)}^2+\rmin^{-2})(\tau^2+\eta \rmin^{-2})\|H-\bar{H}^S\|^2_{L^2(\Sigma)}\\
\le&\ C \rmin^{-2}\|H-\bar{H}^S\|^4_{L^\infty(\Sigma)}+C\rmin^{-8}\big(\tau^2 + \eta\rmin^{-2}\big)^2
\end{align*}
where we also used the above estimate for $\|H-\bar{H}^S\|^2_{L^2(\Sigma)}$. Next we note that
\begin{align*}
\int_\Sigma H^4 |H-\bar{H}^S|^2\dmu\le&\ C\int_\Sigma H^2\Big((\bar{H}^S)^2|H-\bar{H}^S|^2+|H-\bar{H}^S|^4\Big)\dmu\\
\le&\ C(\bar{H}^S)^4\int_\Sigma |H-\bar{H}^S|^2\dmu+C\|H-\bar{H}^S\|^4_{L^\infty(\Sigma)}.
\end{align*}
Hence we get
\begin{align*}
II\le C \rmin^{-2}\|H-\bar{H}^S\|^4_{L^\infty(\Sigma)}+C\rmin^{-8}\big(\tau^2 + \eta\rmin^{-2}\big)^2.
\end{align*}
Inserting these two estimates into \eqref{short} we conclude
 \begin{align*}
   \|H-\bar{H}^S\|^4_{L^\infty(\Sigma)}\le C \rmin^{-2}\|H-\bar{H}^S\|^4_{L^\infty(\Sigma)}+C\rmin^{-8}\big(\tau^2 + \eta\rmin^{-2}\big)^2
\end{align*}
and therefore, by choosing $r_0$ large enough we can absorb the first term on the right hand side and this finishes the proof of the proposition.   
\end{proof}
In the next lemma we derive pointwise estimates for higher derivatives of the curvature.
\begin{lemma}
  \label{thm:linfinityestimates}
  Under the assumptions of theorem~\ref{thm:curv-est-1}, if
  conditions~\eqref{eq:23} hold, we have that
  \begin{align}
    \label{eq:linf1}
    \rmin\|\nabla H\|_{L^\infty(\Sigma)}
    +
    \|\Acirc\|_{L^\infty(\Sigma)}
    \leq
    C\rmin^{-2}\big(\tau + \sqrt{\eta}\rmin^{-1}\big)
  \end{align}
\end{lemma}
\begin{proof}
  Using \eqref{eq:7} we estimate
  \begin{align*}
   \!\! \|\Delta \Acirc\|_{L^2}
    \leq
    &\, 
    c(\|\nabla^2 H\|_{L^2}
    +
    \|H\|_{L^\infty}\|\Acirc\|_{L^4}^2
    +
    \|H\|_{L^\infty}^2\|\Acirc\|_{L^2}+\|\Acirc\|_{L^2}\|\Acirc\|_{L^\infty}^2
    \\
    &\, 
    +\| \RiemM\|_{L^\infty}\|\Acirc\|_{L^2}
    +\|\nabla \omega\|_{L^2})
    \\
    \leq
    &\, 
    C\rmin^{-3}\big(\tau + \sqrt{\eta}\rmin^{-1}\big)+C\|\Acirc\|_{L^2}\|\Acirc\|_{L^\infty}^2,
  \end{align*}
  where we used theorem \ref{thm:improved-curvature-est},
  definition \ref{def:asymptotically-flat}, corollary
  \ref{thm:pointwiseH} and propositions~\ref{thm:nabla-omega-small}
  and~\ref{thm:l2H}. Using an integration by parts argument as in the
  proof of lemma \ref{thm:estimate_d2H} we get
  \begin{align*}
    \|\nabla^2 \Acirc\|_{L^2(\Sigma)}
    \leq
    C\rmin^{-3}\big(\tau + \sqrt{\eta}\rmin^{-1}\big)+C\|\Acirc\|_{L^2(\Sigma)}\|\Acirc\|_{L^\infty(\Sigma)}^2,
  \end{align*}
  Hence we can apply lemma \ref{thm:interpolation} and get
  \begin{align*}
    \|\Acirc\|_{L^\infty(\Sigma)}^4
    &
    \leq
    c\|\Acirc\|_{L^2(\Sigma)}^2
    ( \|\nabla^2 \Acirc\|_{L^2(\Sigma)}^2
    + \|H\|_{L^\infty(\Sigma)}^4\|\Acirc\|_{L^2(\Sigma)}^2)
    \\
    &
    \leq
    C\rmin^{-8}\big(\tau^2 + \eta\rmin^{-2}\big)^2+C\rmin^{-4}\|\Acirc\|_{L^\infty(\Sigma)}^4,
  \end{align*}
  where we used the above estimate for $\nabla^2 \Acirc$ and
  theorem \ref{thm:improved-curvature-est}. Absorbing the last term on the right hand side into the term on the left hand side finishes the proof of the 
  $L^\infty$-estimate for $\Acirc$. For the estimate of $\nabla H$ we
  differentiate \eqref{eq:1} and get
  \begin{align*}
    \|\nabla \Delta H\|_{L^2(\Sigma)}
    \le
    &\,
    c(\lambda \|\nabla H\|_{L^2(\Sigma)}+
    \|\Acirc\|_{L^\infty(\Sigma)}^2\|\nabla H\|_{L^2(\Sigma)}
    \\
    &
    +\|H\|_{L^\infty(\Sigma)}\|\Acirc\|_{L^\infty(\Sigma)}\|\nabla
    \Acirc\|_{L^2(\Sigma)}
    \\
    &
    +\|\Ric(\nu,\nu)\|_{L^\infty(\Sigma)}\|\nabla H\|_{L^2(\Sigma)}
    \\
    &+\|\Ric^T(\cdot,\nu)\|_{L^2(\Sigma)}\|A\|^2_{L^\infty(\Sigma)}
    \\
    &
    +\|H\|_{L^\infty(\Sigma)}\|\nabla \Ric(\nu,\nu)\|_{L^2(\Sigma)})
    \\
    \le
    &\,
    C\rmin^{-4}\big(\tau + \sqrt{\eta}\rmin^{-1}\big).
  \end{align*}
  Hence by interchanging derivatives and integration by parts we get as before
  \begin{align*}
    \|\nabla^3 H\|_{L^2(\Sigma)}
    \leq
    C\rmin^{-4}\big(\tau + \sqrt{\eta}\rmin^{-1}\big).
  \end{align*}
  Applying theorem \ref{thm:improved-curvature-est} and lemma \ref{thm:interpolation} once more, we conclude
  \begin{align*}
    \|\nabla H\|_{L^\infty(\Sigma)}^4
    &
    \le
    C\|\nabla H\|_{L^2(\Sigma)}^2(\|\nabla^3
    H\|_{L^2(\Sigma)}^2+\rmin^{-4}\|\nabla H\|_{L^2(\Sigma)}^2)
    \\
    &
    \le C\rmin^{-12}\big(\tau^2 + \eta\rmin^{-2}\big)^2.
  \end{align*}
  This finishes the proof of the Lemma. 
\end{proof}


%% file: position.tex
\section{Position estimates}
\label{sec:position-estimates}
To get estimates on the position of the approximating sphere, we
exploit the translation sensitivity of surfaces satisfying
\begin{equation}
  \label{eq:65}
  LH + \half H^3 = \lambda H.
\end{equation}
As it turns out, this position estimate is a delicate matter. The goal
is to obtain an estimate for $\tau = |a_e|/R_e$ where $a_e$ and $R_e$
are the center and radius of the approximating sphere constructed in
proposition~\ref{thm:sphere-approx-1}. In fact, we subsequently prove
the following theorem
\begin{theorem}
  \label{thm:position-estimate}
  For all $m>0$, $\eta_0$ and $\sigma$ there exist $r_0<\infty$,
  $\tau_0>0$ and $\eps>0$ with the following properties.  Assume that
  $(M,g)$ is $(m,\eta,\sigma)$-asymptotically Schwarzschild with
  $\eta\leq\eta_0$ and
  \begin{equation*}
    |\ScalM| \leq \eta r^{-5}.
  \end{equation*}
  Then if $\Sigma$ is a surface satisfying equation~\eqref{eq:1} with
  $H>0$, $\lambda>0$, $\rmin>r_0$ and 
  \begin{equation*}
    \tau \leq \tau_0\qquad\text{and}\qquad R_e \leq \eps \rmin^{2},
  \end{equation*}
  then
  \begin{equation*}
    \tau \leq C\sqrt{\eta}\rmin^{-1}.
  \end{equation*}
\end{theorem}
Note that the assumptions of theorem~\ref{thm:position-estimate} imply the
assumptions~\eqref{eq:23}. We will therefore take $r_0$ large enough to be able
to apply the estimates derived in section~\ref{sec:impr-curv-estim}.

Theorem~\ref{thm:position-estimate} follows from
proposition~\ref{thm:position-estimate-pre}, which states that under
the assumptions of theorem~\ref{thm:position-estimate} we have in fact
\begin{equation*}
  \tau \leq C\big(\tau^2 + \sqrt{\eta}\rmin^{-1}\big),
\end{equation*}
for some constant $C$ depending only on $m,\eta_0$ and $\sigma$,
whenever $r_0$ is large enough. Assuming that $\tau_0^2 <1/2C$ yields
the claim.

The crucial ingredients for this estimate are the quadratic structure
of certain error terms, the translation invariance of the functional
$\CU$ with respect to the Schwarzschild background, the Pohozaev
identity, and the contribution of the Schwarzschild geometry to break
the translation invariance. We split the proof of the theorem into the
following subsections.
\subsection{Splitting}
\label{sec:splitting}
Integrating the Gauss equation on
$\Sigma$ yields
\begin{equation*}
  8\pi(1-q(\Sigma)) = \CW(\Sigma) - \CU(\Sigma) -\CV(\Sigma),
\end{equation*}
where $q(\Sigma)$ is the genus of $\Sigma$ and
\begin{align*}
  \CU(\Sigma)
  &:=
  \int_\Sigma |\Acirc|^2\dmu,
  \\
  \CV(\Sigma)
  &:=
  2\int_\Sigma G(\nu,\nu)\dmu,
\end{align*}
where $G = \RicM - \half \ScalM g$ is the Einstein tensor of $M$.
Denoting by $\delta_{f}$ the variation induced by a normal variation
of $\Sigma$ with normal velocity $f$, we infer from the above relation
that
\begin{equation*}
  \delta_f \CW(\Sigma) = \delta_f \CU(\Sigma) + \delta_f \CV(\Sigma).
\end{equation*}
By assumption we have
\begin{equation*}
  \delta_f \CW(\Sigma) = \lambda \int_\Sigma Hf \dmu,
\end{equation*}
hence
\begin{equation}
  \label{eq:48}
  \lambda \int_\Sigma Hf\dmu = \delta_f \CU(\Sigma) + \delta_f \CV(\Sigma).
\end{equation}
By a fairly straightforward computation (given all the expressions in
section~\ref{sec:first-second-vari}), we find
\begin{equation}
  \label{eq:26}
  \delta_f \CU(\Sigma)
  =
  - \int_\Sigma
    2 \Acirc^{ij}\nabla^2_{ij} f
  + 2 f \Acirc^{ij}\RicM^T_{ij}
  + f H |\Acirc|^2  
  \dmu.
\end{equation}

\subsection{The variations of $\CU$ in $g$ and $g^S$}
\label{sec:variation-U}
Here we compute the difference of the variation of $\CU$ with respect
to $g$ and to $g^S$, that is the error when changing the metric.

To do this, we restrict to the special case where
\begin{equation*}
  f = \frac{g(\nu,b)}{H},
\end{equation*}
and $b=\tfrac{a_e}{|a_e|}$, where $a_e$ is as in
proposition~\ref{thm:sphere-approx-1} and $\nu$ is the normal of
$\Sigma$ with respect to $g$. Thus, up to the factor of $H^{-1}$, the
function $f$ is the normal velocity induced by translating $\Sigma$ in the 
direction of $b$. We also define
\begin{equation*}
  f^S = \frac{g^S(\nu^S,b)}{\bar H^s},
\end{equation*}
where $\bar H^S$ is as in proposition~\ref{thm:pointwiseH}.  As
$|\nu-\nu^S| \leq C\eta r^{-2}$ and $||H-\bar H^S||_{L^\infty(\Sigma)} \le C
\rmin^{-2}(\tau+\sqrt{\eta}\rmin^{-1})$ we find that
\begin{equation*}
  | f - f^S | \leq C(\tau + \sqrt{\eta}\rmin^{-1}).
\end{equation*}
Before we proceed, we compute the first and second derivative of $f$. 
\begin{equation}
  \label{eq:60}
  \nabla_i f
  =
  H^{-1}\big(g(\nabla_i b, \nu) + g(b, A_i^j e_j)\big)
  -
  H^{-2} \nabla_i H g(b,\nu),
\end{equation}
and hence, as $|\nabla b| \leq Cr^{-2}$, we find that
\begin{equation*}
  \int_\Sigma |\nabla f|^2 \dmu
  \leq
  C\int_\Sigma (r^{-2} + \frac{|A|^2}{H^2}+ \frac{|\nabla H|^2}{H^4}) \dmu
  \leq
  C\rmin^2.
\end{equation*}
The second derivative of $f$ is given by
\begin{equation*}
  \begin{split}
    &\nabla_i\nabla_j f
    \\
    &=
   - A_i^k A_{jk} f
    +2 H^{-3}\nabla_i H \nabla_j H g(b,\nu)
    - H^{-2}\nabla^2_{i,j} H g(b,\nu)
    \\
    &\quad
    + H^{-1}\big(g(\nabla_i\nabla_j b , \nu)
    + g(\nabla_i b, e_k) A_j^k
    + g(\nabla_j b, e_k) A_i^k
    + \nabla_j A_i^k g(b, e_k)
    \big)
    \\
    &\quad
    - H^{-2}\big(\nabla_i H(g(\nabla_j b,\nu)+g(b,e_k)A^k_j)
    +\nabla_j H (g(\nabla_ib,\nu)+g(b,e_k)A^k_i)\big).
  \end{split}
\end{equation*}
In view of our estimates and the rapid decay of $\nabla b$, $\nabla^2
b$, $\nabla H$ and $\nabla^2 H$, the first term on the right hand side
of this equation is one magnitude larger than the other ones. However,
the main contribution is in the trace of $\nabla^2 f$. We will not
have to consider the trace part, as $\nabla^2 f$ is contracted with
the traceless $\Acirc$ in equation~\eqref{eq:26}. The traceless part
$(\nabla^2 f )^0$ can be estimated as follows
\begin{equation}
  \label{eq:28}
  \int_\Sigma |(\nabla^2 f )^0|^2 \dmu
  \leq
  C \int_\Sigma r^{-4}  \dmu
  \leq
  C \rmin^{-2}.  
\end{equation}
Note the jump in decay rates compared to the $L^2$-norm of $|\nabla f|$.
Finally we need to calculate the second derivative of $f^S$
\begin{equation*}
 \begin{split}
 \nabla^S_i\nabla^S_j f^S
  =&
  (\bar H^S)^{-1}\big(
  g^S(\nabla^S_i\nabla^S_j b , \nu^S)
  + g^S(\nabla^S_i b, e_k) (A^S)_j^k
  + g^S(\nabla^S_j b, e_k) (A^S)_i^k
  \\
  &\phantom{(\bar H^S)^{-1}\big(}
  + \nabla^S_j (A^S)_i^k g^S(b, e_k)\big) - (A^S)_i^k (A^S)_{jk} f^S. 
\end{split}
\end{equation*}
We are now in the position to examine
\begin{equation*}
  |\delta_f \CU(\Sigma) - \delta_{f^S} \CU^S(\Sigma)|.
\end{equation*}
We will do this in detail, as this requires some care. First,
consider the first term in equation~(\ref{eq:26}):
\begin{equation*}
  \begin{split}    
    E_1
    &=
    \left| \int_\Sigma g(\Acirc, (\nabla^2 f)^0 ) \dmu
      - \int_\Sigma g^S(\Acirc^S, ((\nabla^2)^S f)^0) \dmu^S \right|
    \\
    &
    \leq
    \left| \int_\Sigma (g-g^S) (\Acirc,(\nabla^2 f)^0)\dmu \right|
    +
    \left| \int_\Sigma g^S(\Acirc-\Acirc^S, (\nabla^2 f)^0) \dmu
    \right|
    \\
    &\phantom{\leq}
    +
    \left| \int_\Sigma g^S(\Acirc^S, (\nabla^2 f)^0) (\dmu -\dmu^S)
    \right|\\
    &\phantom{\leq}
    +
    \left| \int_\Sigma g^S(\Acirc^S, (\nabla^2 f-(\nabla^S)^2f^S)^0)
      \dmu^S \right| .
  \end{split}
\end{equation*}
The first three terms can be estimated using the asymptotics of $g$
and the curvature estimates from theorems~\ref{thm:curv-est-1}
and~\ref{thm:A0-l2-1}.
\begin{equation*}
  \begin{split}
    E_1^a
    &\leq
    C \eta \rmin^{-2}\int_\Sigma |\Acirc| |(\nabla^2 f)^0)|
    +
    (r^{-1} + |A| ) |(\nabla^2 f)^0)|
    +
     |\Acirc^S| |(\nabla^2 f)^0)|
    \dmu
    \\
    &\leq
    C \eta\rmin^{-2} \| (\nabla^2 f)^0\|_{L^2(\Sigma)}
    \big(\|A\|_{L^2} + \rmin^{-1} |\Sigma|^{1/2}\big)
    \\
    &\leq C \eta \rmin^{-3}.
  \end{split}
\end{equation*}
Using again the fact that we are contracting with the traceless
second fundamental form and the above equations for the second
derivatives of $f$ and $f^S$ we see that we can estimate the last term
for $E_1$, denoted by $E_1^b$, by
\begin{equation}
  \label{eq:77}
  \begin{split}
    E_1^b &
    \leq
    C\int_\Sigma |\Acirc^S|H^{-2}
    \big(|\nabla H||\nabla b|
    +|\nabla H||A|
    +|\nabla^2 H|
    +H^{-1}|\nabla H|^2
    \big)\dmu^S
    \\
    &\quad
    + C\int_\Sigma \frac{|\Acirc^S||H-\bar H^S|}{H\bar H^S}
    \big( |\nabla^2 b|+|\nabla b||A|+|\nabla A|+H |\Acirc|\big)\dmu^S
    \\
    &\quad
    + C\bar H_S^{-1} \int_\Sigma
    |\Acirc^S||g(\nabla_i\nabla_j b , \nu)-g^S(\nabla^S_i\nabla^S_j b , \nu^S)|
    \\
    &\quad\phantom{+ C\bar H_S^{-1} \int_\Sigma}
    + |\Acirc^S| |g(\nabla_i b, e_k) A_j^k- g^S(\nabla^S_i b, e_k) (A^S)_j^k|
    \\
    &\quad\phantom{+ C\bar H_S^{-1} \int_\Sigma}
    + |\Acirc^S| |g(\nabla_j b, e_k) A_i^k - g^S(\nabla^S_j b, e_k)
    (A^S)_i^k|
    \\
    &\quad\phantom{+ C\bar H_S^{-1} \int_\Sigma}
    + |\Acirc^S| |\nabla_j A_i^k g(b, e_k) -\nabla^S_j (A^S)_i^k
    g^S(b, e_k) | 
    \dmu^S
    \\
    &\quad
    + C\int_\Sigma
    |\Acirc^S| |(A_i^k A_{jk})^0 f - ((A^S)_i^k
    A^S_{jk})^0 f^S|
    \dmu^S .
  \end{split}
\end{equation}
By the curvature estimates from section~\ref{sec:impr-curv-estim} the
terms on the first two lines in equation~\eqref{eq:77} are estimated
by
\begin{equation*}  
  \begin{split}   
    &C\|\Acirc^S\|_{L^2(\Sigma)}
    \big(
    \rmin \|\nabla A\|_{L^2(\Sigma)}
    +\rmin^2 \|\nabla^2 H\|_{L^2(\Sigma)}
    + \rmin \|\nabla (\operatorname{log} H)\|_{L^4(\Sigma)}^2
    \\
    & \phantom{C\|\Acirc^S\|_{L^2(\Sigma)}\big(}
    + \|\nabla^2 b\|_{L^2(\Sigma)}
    + \rmin^{-2} \|A\|_{L^2}
    + \rmin^{-1}\|\Acirc\|_{L^2(\Sigma)}
    \\
    &\quad
    \leq
    C\rmin^{-2}\big(\tau^2 + \tau\rmin^{-1} +
    \sqrt{\eta}\rmin^{-1}\big) .
  \end{split}
\end{equation*}
We estimate the terms on the last five lines of equation~\eqref{eq:77}
seperately. The third line yields
\begin{equation*}
  \begin{split}
    &
    \int_\Sigma |\Acirc^S| |g(\nabla_i \nabla_j b,\nu)-g^S (\nabla^S_i
    \nabla^S_j b,\nu^S)|\dmu^S
    \\
    &\quad
    \leq
    \int_\Sigma |\Acirc^S|\big(
    |(g-g^S)(\nabla_i \nabla_j b,\nu)|
    +|g^S((\nabla_i \nabla_j -\nabla^S_i \nabla^S_j)b,\nu^S)|
    \\
    &\quad\phantom{\leq\int_\Sigma |\Acirc^S|\big(}
    +|g^S(\nabla_i \nabla_j b,\nu-\nu^S)|)\dmu^S
    \\
    &\quad
    \leq
    C\eta\rmin^{-5}.
  \end{split}
\end{equation*}
The fourth and fifth line of~\eqref{eq:77} are estimated as follows
\begin{equation*}
  \begin{split}
    &
    \int_\Sigma |\Acirc^S| |g(\nabla_i b,e_k)A^k_j
     - g^S (\nabla^S_i b,e_k)(A^S)^k_j|\dmu^S
    \\
    &\quad
    \leq
    \int_\Sigma |\Acirc^S|\big(
    |(g-g^S)(\nabla_i b,e_k)A^k_j|
    +|g^S((\nabla_i  -\nabla^S_i)b,e_k)(A^S)^k_j|\\
    &\quad\phantom{\leq\int_\Sigma |\Acirc^S|\big(}
    +|g^S(\nabla_i  b,e_k)(A^k_j-(A^S)^k_j)|\big)\dmu^S \\
    &\quad
    \leq
    C\eta\rmin^{-4}.
  \end{split}
\end{equation*}
For the sixth line of~\eqref{eq:77} we get
\begin{equation*}
  \begin{split}
    &\int_\Sigma |\Acirc^S| |g( b,e_k)\nabla_i A^k_j
    - g^S (b,e_k)\nabla^S_i (A^S)^k_j|\dmu^S 
    \\
    &\quad
    \leq
    \int_\Sigma |\Acirc^S|\big(
    |(g-g^S)(  b,e_k)\nabla_i A^k_j|
    +|g^S(b,e_k)\nabla^S_i (A^k_j-(A^S)^k_j)|
    \\
    &\quad\phantom{\leq\int_\Sigma |\Acirc^S|\big(}
    +|g^S(b,e_k)(\nabla_i  -\nabla^S_i)A^k_j|\big)\dmu^S
    \\
    &\quad
    \leq
    C\eta\rmin^{-4}.
  \end{split}
\end{equation*}
It remains to estimate the last line of~\eqref{eq:77}
\begin{equation*}
  \begin{split}
    &
    \int_\Sigma |\Acirc^S||\Acirc^k_i A_{jk}f-(\Acirc^S)^k_i
    A^S_{jk}f^S|\dmu^S
    \\
    &\quad
    \leq
    C\int_\Sigma |\Acirc^S|\big(
    |\Acirc^k_i A_{jk}||f-f^S|
    +|A-A^S||A||f^S|
    \big)\dmu^S
    \\
    &\quad
    \leq
    C\rmin^{-3}\big(\tau+\sqrt{\eta}\rmin^{-1}\big).
\end{split}
\end{equation*}
Combining all these estimates we arrive at the estimate for the first
error term
\begin{equation*}
  \begin{split}
    E_1
    \leq
    C\rmin^{-2}\big(\tau^2 + \tau\rmin^{-1} +
    \sqrt{\eta}\rmin^{-1}\big).
  \end{split}
\end{equation*}
Similarly, the second term in equation~\eqref{eq:26} gives the error
\begin{equation*}
  \begin{split}
    E_2
    &:=
    \left| \int_\Sigma f \la \Acirc, \Ric^T \ra \dmu
      - \int_\Sigma  f^S \la \Acirc^S, (\Ric^S)^T \ra \dmu^S
    \right|
    \\
    &\leq
    \int_\Sigma |\Acirc - \Acirc^S| |\Ric^T| |f| \dmu 
    +
    \int_\Sigma |\Acirc^S| |\Ric^T - (\Ric^S)^T| |f| \dmu
    \\
    &\phantom{\leq}
    +
    \int_\Sigma |\Acirc^S| |(\Ric^S)^T| |f| |\dmu -\dmu^S|+\int_\Sigma |\Acirc^S| |(\Ric^S)^T||f-f^S| \dmu^S
    \\
    &\leq
    C\rmin^{-3}\big(\tau + \sqrt{\eta}\big).
  \end{split}
\end{equation*}
And the third term in equation~\eqref{eq:26} contributes 
\begin{equation*}  
  \begin{split}
    E_3
    &:=
    \left| \int_\Sigma f H|\Acirc|^2 \dmu
      - \int_\Sigma f^S H^S |\Acirc^S|^2 \dmu^S \right|
    \\
    &
    \leq C \int_\Sigma |\Acirc-\Acirc^S| |\Acirc| \dmu
    + C \int_\Sigma |\Acirc^S|^2 |\dmu-\dmu^S| \dmu
    \\
    &\phantom{\leq}
    + C\int_\Sigma|f^SH^S-g^S(b,\nu^S)| |\Acirc^S|^2\dmu^S\\ 
    \\
    &\leq
    C\rmin^{-3}\big(\tau + \sqrt{\eta}\big).
  \end{split}
\end{equation*}
In summary, we find that
\begin{equation}
  \label{eq:29}
  |\delta_f \CU(\Sigma) - \delta_{f^S} \CU^S(\Sigma) |
  \leq
  C\rmin^{-2}\big(\tau^2 + \tau\rmin^{-1} + \sqrt{\eta}\rmin^{-1}\big)
  .
\end{equation}
As the functional $\CU^S$ is translation invariant, due to conformal
invariance and conformal flatness of $g^S$, we find that
\begin{equation*}
  \delta_{f^S} \CU^S(\Sigma) = 0
\end{equation*}
and hence
\begin{equation}
  \label{eq:46}
  |\delta_f \CU(\Sigma) |
  \leq
  C \rmin ^{-2}\big(\tau^2 + \tau\rmin^{-1} + \sqrt{\eta}\rmin^{-1}\big).
\end{equation}
\subsection{The left hand side of \eqref{eq:48}}
\label{sec:right-hand-side}
Here we estimate the left hand side of equation~\eqref{eq:48}. By
our choice of test function this becomes (omitting $\lambda$ for now).
\begin{equation*}
  \int_\Sigma g(b,\nu) \dmu.
\end{equation*}
First, we estimate the error when we take all quantities with respect
to the metric $g^S$.
\begin{equation*} 
  \begin{split}
    &\left| \int_\Sigma g(b,\nu) \dmu - \int_\Sigma g^S(b,\nu^S) \dmu^S \right|
    \\
    &
    \quad\leq
    \int_\Sigma |g-g^S|\dmu
    +
    \int_\Sigma |\nu-\nu^S|\dmu
    +
    \int_\Sigma |\dmu-\dmu^S| \dmu
    \leq
    C\eta.
  \end{split}
\end{equation*}
Then we insert the relations from
lemma~\ref{thm:geometry-in-schwarzschild} to compute
\begin{equation*}
  \begin{split}
    \int_\Sigma g^S(b,\nu^S) \dmu^S
    &=
    \int_\Sigma \phi^{6} g^e(b,\nu^e) \dmu^e
    \\
    &=
    \int_\Sigma \big(1 + \tfrac{3m}{r} + \text{lower order} \big) g^e(b,\nu^e) \dmu^e.
  \end{split}
\end{equation*}
We deal with the highest order term first. Note that by translation
invariance of the volume enclosed by $\Sigma$ in Euclidean space, we
find
\begin{equation}
  \label{eq:75}
  \int_\Sigma g^e(b,\nu^e) \dmu^e = 0,
\end{equation}
and hence
\begin{equation*}
  \int_\Sigma g^S(b,\nu^S) \dmu^S
  =
  \int_\Sigma \big(\tfrac{3m}{r} + \text{lower order}\big) g^e(b,\nu^e) \dmu^e.
\end{equation*}
The lower order terms are of the form $c_k r^{-k}$ where $c_k$ depends
only on $m$ and $k=2,\ldots,6$. We can replace $r$ by $R_e$ in these
integrals, and in view of proposition~\ref{thm:sphere-approx-1} and
theorem~\ref{thm:improved-curvature-est} we find that
\begin{equation*}
  \big|r^{-k} - R_e^{-k}\big|
  \leq
  C\rmin^{-k}\big(\tau +\sqrt{\eta}\rmin^{-1}\big).
\end{equation*}
Since $k\geq 2$, we can estimate all resulting error terms by
\begin{equation*}
  \sum_{k=2}^6
  \int_\Sigma \left| \frac{c_k}{r^k} - \frac{c_k}{R_e^k}\right| \dmu
    \leq
    C\big(\tau + \sqrt{\eta}\rmin^{-1}\big).
\end{equation*}
The remaining integrals satisfy
\begin{equation*}
  \int_\Sigma \frac{c_k}{R_e^k} g^e(b,\nu^e) \dmu^e = 0
\end{equation*}
due to relation~\eqref{eq:75}. Combining the above calculations, we
find that
\begin{equation*}
  \left|
    \int_\Sigma g(b,\nu) \dmu
    -
    \int_\Sigma \tfrac{3m}{r}g^e(b,\nu^e)\dmu^e
  \right|
  \leq
  C\big(\tau + \sqrt{\eta}\rmin^{-1} + \eta \big).
\end{equation*}
The estimate on $\|\Acirc\|_{L^2(\Sigma)}$ allows us to change the
domain of integration to the round sphere $S:=S_{R_e}(a_e)$, and change
$\nu^e$ to $N$, the normal of $S$ while introducing only an error
estimated by $C(\tau+\sqrt{\eta}\rmin^{-1})$. The corresponding
integral on the sphere can be computed using the methods introduced
in the proof of proposition~\ref{thm:compute_lambda}. The result is (see appendix $A.2$)
\begin{equation*}
  \int_S \tfrac{3m}{r} g^e(b,N) \dmu^e
  =
  -4\pi m |a_e|.
\end{equation*}
Hence, collecting the error terms acquired on the way, we find
\begin{equation}
  \label{eq:49}
  \left| \int_\Sigma H f\dmu  + 4\pi m |a_e| \right|
  \leq
  C\big(\tau + \sqrt{\eta}\rmin^{-1} + \eta \big).
\end{equation}
recall that $|\lambda - \tfrac{2m}{R_S^3}| \leq C\big(\rmin^{-4}(\tau+\sqrt{\eta}\rmin^{-1})\big)$, whence
\begin{equation}
  \label{eq:57}
  \left| \lambda \int_\Sigma Hf \dmu  + \frac{8\pi m^2
      \tau}{\bar\phi^2 R_s^2} \right|
  \leq
  C\rmin^{-3}(\tau + \sqrt{\eta}\rmin^{-1} + \eta \big),
\end{equation}
where $\bar\phi = 1 + \frac{m}{2R_e}$, $R_S= \bar\phi^2 R_e$ as in
proposition~\ref{thm:compute_lambda} and we used the definition $\tau = |a_e|/R_e$.
\subsection{The Pohozaev identity}
\label{sec:pohozaev-identity}
Before we study the variation of $\CV$, we recall the (geometric)
Pohozaev identity. To this end we denote the conformal Killing
operator by
\begin{equation*}
  \CD X : = \CL_Xg - \frac{1}{3} \tr (\CL_Xg)g 
\end{equation*}
where $X$ is a vector field on $M$ and $\CL_X g$ denotes the Lie
derivative of $g$ with respect to $X$. Let $\Omega\subset M$ be a
smooth domain with boundary $\Sigma$ and let $dV$ be the volume form
of $M$. Then the Pohozaev identity\footnote{In the literature (see for example \cite{MR929283}) the
  Pohozaev identity is usually stated for the trace-free Ricci tensor,
  not for the Einstein tensor. For our purposes however, it is more
  convenient to write it in terms of $G$.} can be stated as
\begin{equation}
  \label{eq:24}
  \frac{1}{2} \int_\Omega \la G, \CD X \ra dV - \frac{1}{6}
  \int_\Omega \ScalM \div X dV
  =
  \int_\Sigma G(X,\nu)\dmu.
\end{equation}
This identity can be seen as follows: In local coordinates we have
\begin{align*}
(\CD X)_{kl} = \nabla_k X_l+\nabla_l X_k-\frac23 \div X g_{kl}
\end{align*}
and therefore 
\begin{align*}
 \frac{1}{2} \int_\Omega \la G, \CD X \ra dV=&\, \frac{1}{2} \int_\Omega \Big(g^{ik}g^{jl} G_{ij} (\nabla_k X_l+\nabla_l X_k) -\frac23 G_{ii} \div X \Big)dV\\
=& -\int_\Omega \la \div G, X \ra dV+\frac16  \int_\Omega \ScalM \div X dV \\
&+\int_\Sigma G(X,\nu)\dmu,
\end{align*}
which proves \eqref{eq:24} since $G$ is divergence free.
\begin{lemma}
  \label{thm:pohozaev-decay}
  Let $\Sigma$ be a surface as in theorem \ref{thm:position-estimate} which bounds an exterior domain $\Omega$,
  and let $b\in\IR^3$ be a constant vector. Then
  \begin{equation*}
    \left| \int_\Sigma G(b,\nu)\dmu \right|
    \leq
    C \eta \rmin^{-3}.
  \end{equation*}
\end{lemma}
\begin{proof}
  Consider the vector field $b$, where $b\in\IR^3$ is constant. Then $b$
  is a Killing vector field in flat $\IR^3$ and hence a conformal
  Killing vector field with respect to $g^S$. Denoting by $\CD^S$ the
  conformal Killing operator with respect to $g^S$, we thus find
  \begin{equation*}
    \CD^S b = 0.
  \end{equation*}
  With respect to the general metric $g$, this implies the decay rate
  \begin{equation*}
    | \CD b | \leq C \eta r^{-3},
  \end{equation*}
  since $|\nabla - \nabla^S| \leq C\eta r^{-3}$. The other terms in
  equation~\eqref{eq:24} have decay $|G| \leq Cr^{-3}$, $|\ScalM|\leq
  C\eta r^{-4}$, and $|\div b|\leq C r^{-2}$.

  Let $S_\sigma$ be a coordinate sphere of radius $\sigma$ outside of $\Sigma$ and let
  $\Omega_\sigma$ be the domain bounded by $\Sigma$ and $S_\sigma$. The
  contribution of $S_\sigma$ to the boundary integral in
  equation~\eqref{eq:24} decays like $\sigma^{-1}$ and thus we infer
  that
  \begin{equation}
    \label{eq:59}
    \int_\Sigma G(b,\nu)\dmu    
    = \lim_{\sigma \rightarrow \infty}
    \left(- \frac{1}{2} \int_{\Omega_\sigma} \la G, \CD b \ra dV
      +
      \frac{1}{6} \int_{\Omega_\sigma} \ScalM \div b\, dV \right).
  \end{equation}
  The sign of the right hand side is different to~\eqref{eq:24}, as
  our conventions are that $\nu$ is the outward pointing normal to $\Sigma$
  which points into $\Omega$. 

  The integrand in the volume integral decays like $C\eta r^{-6}$, which
  implies via lemma~\ref{thm:volume-integral-decay} that the integral
  can be estimated by $C\eta\rmin^{-3}$ as claimed.
\end{proof}
\subsection{The variation of $\CV(\Sigma)$}
\label{sec:variation-V}
The variation of $\CV$ can be computed to be
\begin{equation}
  \label{eq:62}
  \half \delta_f\CV(\Sigma)
  =
  \int_\Sigma
  f\big( \nabla_\nu G(\nu,\nu)
  +H G(\nu,\nu) \big)
  - 2 G(\nu,\nabla f)
  \big)
  \dmu.
\end{equation}
Since $G$ is divergence-free we calculate
\begin{align}
  \nabla_\nu G(\nu,\nu) =&\, \div G(\nu) - \nabla_{e_i} G(\nu,e_i)=- \nabla_{e_i} G(\nu,e_i)\nonumber \\
=& -\nabla_{e_i}\RicM(\nu,e_i)\nonumber \\
=& -\divSig \omega + \RicM(h_{ik}e_k ,e_i) - H \RicM(\nu,\nu)\nonumber \\
=&-\divSig \omega- H \RicM(\nu,\nu)+\Acirc_{ik} \RicM_{ik} \nonumber \\
&+ \frac12 H (\ScalM-\RicM(\nu,\nu))\nonumber \\
=& -\divSig \omega  + \la \Acirc, G^T \ra - \tfrac{1}{4} H \ScalM -
  \tfrac{3}{2} H G(\nu,\nu),\label{eq:62a}
\end{align}
where, as usual, $\omega = \RicM(\nu,\cdot)^T = G(\nu,\cdot)^T$.
Inserting this into~\eqref{eq:62}, we find that
\begin{equation*}
  \begin{split}
    \half \delta_f\CV(\Sigma)
    &=
    \int_\Sigma
    f \la \Acirc,G^T\ra
    - f \divSig\omega
    - \half f H G(\nu,\nu)
    - \tfrac{1}{4}fH\ScalM
    -2 \omega(\nabla f)
    \dmu
    \\
    &=
    \int_\Sigma
    - \half f H G(\nu,\nu)
    - \tfrac{1}{4}fH\ScalM
    + f \la \Acirc,G^T\ra
    - \omega(\nabla f)
    \dmu.
  \end{split}
\end{equation*}
We specialize again to the test function
\begin{equation*}
  f = \frac{g(b,\nu)}{H}
\end{equation*}
for a fixed vector $b\in\IR^3$. In the expression~\eqref{eq:60} for
$\nabla f$ we can split $A = \Acirc + \half H \gamma$ and obtain
\begin{equation}
  \label{eq:61}
  \nabla_i f
  = H^{-1}
  \big(
  g (\nabla_i b, \nu)
  + g(b,e_j) \Acirc_i^j
  - \nabla_i\log H g(b,\nu)
  \big)
  + \half g(b,e_i).
\end{equation}
Inserting this into equation~\eqref{eq:62}, we find that
\begin{equation}
  \label{eq:63}
  \begin{split}
     \half \delta_f\CV(\Sigma)
    &=
    \int_\Sigma
    - \half f H G(\nu,\nu)
    - \tfrac{1}{4}fH\ScalM
    + f \la \Acirc,G^T\ra
    - \half G(\nu,b^T)
    \\
    &\phantom{= \int_\Sigma}    
    -  H^{-1} \omega(e_i)
    \big(
    g(\nabla_i b,\nu)
    + g(b,e_j) \Acirc_i^j
    - \nabla_i\log H  g(b,\nu)
    \big) \dmu
    \\
    &=
    \int_\Sigma
    - \half G(b,\nu)
    - \tfrac{1}{4} g(b,\nu) \ScalM
    + H^{-1} g(b,\nu) \la \Acirc,G^T\ra
    \\
    &\phantom{= \int_\Sigma}    
    - H^{-1} \omega(e_i)
    \big(
    g(\nabla_i b,\nu)
    + g(b,e_j) \Acirc_i^j
    - \nabla_i\log H  g(b,\nu)
    \big) \dmu .
  \end{split}
\end{equation}
It is this expression for $\delta_f \CV$ which will give rise to the
position estimates. We will thus spend some time on understanding the
error terms. Because of propositions~\ref{thm:decay-curvature}
and~\ref{thm:improved-curvature-est} we have the estimate
\begin{equation*}
  \int_\Sigma H^{-1}\Big(
  |\la \Acirc,G^T\ra|
  + |\omega| |\Acirc|
  + |\omega| |\nabla \log H|\Big)
  \dmu
  \leq
  C \rmin^{-2}\big(\tau^2 + \eta\rmin^{-2}\big).
\end{equation*}
Note that proposition~\ref{thm:decay-curvature} implies that
$\|(G^T)^\circ\|_{L^2(\Sigma)}^2 \leq C \rmin^{-4}\big(\tau^2+\eta\rmin^{-2}\big)$.
Assuming that $|\ScalM| \leq \eta r^{-5}$ we find that
\begin{equation*}
  \left|\int_\Sigma \ScalM \dmu\right|
  \leq
  C\eta\rmin^{-3}.
\end{equation*}
Lemma~\ref{thm:pohozaev-decay} implies that the first term on the
right hand side of~\eqref{eq:63} is also estimated by $C\eta\rmin^{-3}$, so
that the only term which yields a contribution of order 
$\rmin^{-2}$ is
\begin{equation*}
  \int_\Sigma H^{-1} \omega(e_i) g(\nabla_{e_i} b, \nu) \dmu.
\end{equation*}
We will explicitly evaluate this term. To this end note that
\begin{equation*}
  \begin{split}
    &
    \left|
      \int_\Sigma H^{-1} \omega(e_i) g(\nabla_{e_i} b, \nu) \dmu
      -
      \int_\Sigma (\bar H^S)^{-1} \Ric^S(e_i^S,\nu^S) g^S(\nabla^S_{e_i^S} b, \nu^S)
      \dmu^S
    \right|
    \\
    & \quad
    \leq
    C\rmin^{-3}(\tau+ \sqrt{\eta})
  \end{split}
\end{equation*}
where $\bar H^S$ is the quantity from corollary~\ref{thm:pointwiseH}
and $e_i^S$ constitute a tangential ON-frame with respect to the
metric induced by $g^S$. This estimate follows since the integrand
scales like $r^{-4}$ and the transition errors to Schwarzschild decay
at least one order faster and have factor $\eta$. Furthermore, the
replacement of $H$ by $\bar H^S$ introduces an extra error term of the
form $C\rmin^{-3}(\tau+\sqrt{\eta}\rmin^{-1})$. We calculate, using
that $D^eb \equiv 0$ and the transformation properties of the
Christoffel symbols under a conformal change of the metric (see for example \cite{schoen-yau:1994}),
\begin{equation*}
  \nabla^S_{e_i^S} b
  =
  2\phi^{-1}\big( e^S_i(\phi) b  + b(\phi) e^S_i - D^e\phi g^e(b,e_i^S)\big),
\end{equation*}
which implies that
\begin{equation*}
  g^S(\nabla^S_{e_i^S} b, \nu^S)
  =
  \phi^{-1} \frac{m}{r^2}
  \big(
  g^e(\rho,\nu^e)g^e(b,e^e_i)
  - g^e(\rho,e^e_i)g^e(b,\nu^e)
  \big).
\end{equation*}
Here $e_i^e = \phi^2 e_i^S$ is a tangential ON-frame with respect to
the metric induced by $g^e$. Furthermore, the formula from
lemma~\ref{thm:geometry-in-schwarzschild} yields that
\begin{equation*}
  \Ric^S(\nu^S,e^S_i)
  =
  -3 \frac{m}{r^3}\phi^{-6} g^e(\rho,\nu^e)g^e(\rho,e^e_i).
\end{equation*}
Multiplying these terms gives (note that we sum over $i=1,2$)
\begin{equation*}  
  \begin{split}
    &\Ric^S(\nu^S,e^S_i)g^S(\nabla^s_{e_i^S} b, \nu^S)
    \\
    &\quad=
    3 \frac{m^2}{r^5}\phi^{-7}
    \big(
    |\rho^T|^2 g^e(\rho,\nu^e) g^e(b,\nu^e)
    - |\rho^\perp|^2 g^e(\rho^T,b^T)
    \big)
    \\
    &\quad=
    3 \frac{m^2}{r^5}\phi^{-7}g^e(\rho,\nu^e)
    \big(
    g^e(b,\nu^e)
    - g^e(\rho,\nu^e) g^e(b,\rho)
    \big).
  \end{split}
\end{equation*}
As in the proof of proposition~\ref{thm:compute_lambda}, we replace
the integral over $\Sigma$ by an integral over $S=S_{R_e}(a_e)$ while
introducing error terms of one order lower. This implies that
\begin{equation*}
  \begin{split}
    &
    \bigg|
    \frac{3 m^2}{\bar\phi^7 \bar H^S}
    \int_{S}  \frac{1}{r^5}
    g^e(\rho,N)
    \big(
    g^e(b,N)
    - g^e(\rho,N) g^e(b,\rho)
    \big)
    \dmu^e
    \\
    &\qquad\qquad
    -\int_\Sigma H^{-1} \omega(e_i) g(\nabla_{e_i} b, \nu) \dmu
    \bigg|
    \leq
    C\rmin^{-3}(\tau + \sqrt{\eta}),
  \end{split}
\end{equation*}
where $N$ is the Euclidean normal vector to $S$ and $\bar\phi =
1+\frac{m}{2R_e}$ the quantity introduced in
proposition~\ref{thm:compute_lambda}. The first integral can be
evaluated explicitly, where we again introduce coordinates
$\vartheta,\varphi$ in which $g^e(b,N) = \cos \varphi$.  As $\rho =
r^{-1} (R_eN + a_e)$ we can express this integral by
\begin{equation*}
  \begin{split}
    Q(|a|,R)
    &:=
    \frac{3 m^2}{\bar\phi^7 \bar H^S}
    \int_{S} \frac{1}{r^5}
    g^e(\rho,N) \big( g^e(b,N) - g^e(\rho,N) g^e(b,\rho) \big) \dmu^e
    \\
    &=
    \frac{3 m^2}{\bar\phi^7 \bar H^S}
    \int_{S} \Big(
    R_e\tfrac{\cos\varphi}{r^6}
    + |a_e| \tfrac{\cos^2\varphi}{r^6}
    - |a_e|R_e^2 \tfrac{1}{r^8}
    - (R_e^3 + 2 |a_e|^2R_e)\tfrac{\cos\varphi}{r^8}
    \\
    &\phantom{\frac{3 m^2}{\bar\phi^6 \bar H^S} \int_{S} \Big(}\qquad
    - (|a_e|^3 + 2|a_e|R_e^2) \tfrac{\cos^2\varphi}{r^8}
    - |a_e|^2 R_e \tfrac{\cos^3\varphi}{r^8} \Big) \dmu^e.
  \end{split}
\end{equation*}
Explicitly evaluating these terms (see appendix $A.3$), we obtain the following
expression for $Q$. We already substituted $\tau:= |a_e|/R_e$:
\begin{equation*}
  Q(\tau,R_e)
  =
  \frac{m^2 \pi}{ 4 \bar\phi^7 \bar H^S R_e^3}
  \frac{3 (\tau^6 - 3\tau^4 + 3\tau^2 -1) \ln\frac{1-\tau}{1+\tau}
    + 6\tau^5 - 16\tau^3 -6\tau}
  {\tau^2 (1+\tau)^3(1-\tau)^3}. 
\end{equation*}
To analyze this expression we set
\begin{equation}
  \label{eq:58}
  f(\tau)
  =
  \frac{3 (\tau^6 - 3\tau^4 + 3\tau^2 -1) \ln\frac{1-\tau}{1+\tau}
    + 6\tau^5 - 16\tau^3 -6\tau}
  {\tau^2 (1+\tau)^3(1-\tau)^3}.   
\end{equation}
Recall the Taylor expansion of the function
$\ln\frac{1-\tau}{1+\tau}$:
\begin{equation*}
  \ln\frac{1-\tau}{1+\tau}
  =
  -2\tau - \frac{2}{3}\tau^3 + O(\tau^4),
\end{equation*}
for small $\tau$. Thus we find that the numerator in
equation~\eqref{eq:58} is
\begin{equation*}
  3(3\tau^2 -1) (-2\tau -\frac{2}{3}\tau^3) - 16\tau^3 -6\tau +
  O(\tau^4)
  =
  -32\tau^3 + O(\tau^4) .
\end{equation*}
Hence we get that
\begin{equation*}
  Q(\tau,R)
  =
  - \frac{8\pi m^2\tau}{\bar\phi^7 \bar H^S R_e^3} + \frac{O(\tau^2)}{R_e^2}
\end{equation*}
for small $\tau$.
In summary, the above computation implies the following estimate
\begin{equation}
  \label{eq:64}
  \left| \delta_f \CV(\Sigma) - \frac{16 \pi m^2\tau}{\bar\phi^7 \bar
      H^S R_e^3} \right|
  \leq
  C\rmin^{-2}\big(\tau^2 + \tau\rmin^{-1} + \sqrt{\eta}\rmin^{-1}\big).
\end{equation}
\subsection{Position estimates}
\label{sec:sub-position-estimates}
Theorem~\ref{thm:position-estimate} is a consequence from an iterative application of the
following proposition.
\begin{proposition}
  \label{thm:position-estimate-pre}
  If $(M,g)$ and $\Sigma$ are as in theorem~\ref{thm:position-estimate}, then
  \begin{equation*}
    \tau
    \leq
    C\big(\tau^2 + \tau\rmin^{-1} + \sqrt{\eta}\rmin^{-1}\big),
  \end{equation*}
\end{proposition}
\begin{proof}
  We computed in section~\ref{sec:right-hand-side} that (cf.~\eqref{eq:57}),
  \begin{equation*}
    \left| \lambda \int_\Sigma Hf \dmu   + \frac{8\pi m^2 \tau
      }{\bar\phi^2 R_S^2} \right|
    \leq
    C\rmin^{-3}\big(\tau + \sqrt{\eta}\big),     
  \end{equation*}
  in section~\ref{sec:variation-U} that (cf.~\eqref{eq:46}),
  \begin{equation*}
      |\delta_f \CU(\Sigma) |
      \leq
      C\rmin^{-2}\big(\tau^2 + \tau\rmin^{-1} + \sqrt{\eta}\rmin^{-1} \big),      
  \end{equation*}
  and in section~\ref{sec:variation-V} that (cf.\ \eqref{eq:64})
  \begin{equation*}
    \left| \delta_f \CV(\Sigma) - \frac{16\pi m^2
        \tau}{\bar H^S R_S^3} \right|
    \leq
    C\rmin^{-2}\big(\tau^2 + \tau\rmin^{-1} + \sqrt{\eta}\rmin^{-1}\big).    
  \end{equation*}
  Inserting these equations into equation~\eqref{eq:48} we find,
  after absorbing the lower order terms on the left into the error
  terms, that
  \begin{equation*}
    24\pi m^2 \tau
    \leq
    C\big(\tau^2 + \tau\rmin^{-1} + \sqrt{\eta}\rmin^{-1}\big),
  \end{equation*}
  which is the claimed estimate.  
\end{proof}
\subsection{Final version of the curvature estimates}
In this subsection we state our final version of the previous
curvature estimates.
\begin{theorem}
  \label{thm:final-curv}
  For all $m>0$, $\eta_0$ and $\sigma$ there exist $r_0<\infty$,
  $\tau_0>0$, $\eps>0$, and $C$ depending only on $m,\sigma$ and
  $\eta_0$ with the following properties.

  Assume that $(M,g)$ is $(m,\eta,\sigma)$-asymptotically
  Schwarzschild with $\eta\leq\eta_0$ and
  \begin{equation*}
    |\ScalM| \leq \eta r^{-5}.
  \end{equation*}
  Then if $\Sigma$ is a surface satisfying equation~\eqref{eq:1} with
  $H>0$, $\lambda>0$, $\rmin>r_0$ and
  \begin{equation*}
    \tau \leq \tau_0\qquad\text{and}\qquad R_e \leq \eps\rmin^{2},
  \end{equation*}
  where $R_e$ and $\tau$ are as in section~\ref{sec:impr-curv-estim}, we have the following estimates
  \begin{equation}
    \label{eq:79}
    \|H - \bar H^S \|_{L^{\infty}} 
    +
    \|\Acirc\|_{L^\infty}
    +
    \rmin
    \|\nabla H \|_{L^\infty}
    \leq
    C\sqrt{\eta}\rmin^{-3}.
  \end{equation}
  Here $\bar H^S = \frac{2}{R_S} - \bar \phi\frac{2 m}{R_S^2}$ with $R_S =
  \bar \phi^2 R_e$ and $\bar\phi = 1 + \frac{m}{2R_e}$.  Furthermore, we
  have that
  \begin{equation}
    \label{eq:80}
    \|\nu - \phi^{-2}\rho\|_{L^\infty}
    \leq
    C\sqrt{\eta}\rmin^{-1}.
  \end{equation}
  This implies, 
  \begin{equation}
    \label{eq:81}
    \begin{split}
      &\|\lambda + \Ric(\nu,\nu)\|_{L^\infty}
      +
      \|\Ric(\nu,\nu) + 2mR_S^{-3}\|_{L^\infty}    
      \leq
      C\sqrt{\eta} \rmin^{-4}.
      \\[.2ex]
      &
      \|\omega\|_{L^\infty}
      +
      \rmin
      \|\nabla\omega\|_{L^\infty}
      \leq
      C\sqrt{\eta} \rmin^{-4}.
    \end{split}
  \end{equation}
\end{theorem}
\begin{proof}
  The estimates in~\eqref{eq:79} are straight-forward consequences of
  the estimates in section~\ref{sec:impr-curv-estim} and the position
  estimate~\ref{thm:position-estimate}. The estimate for the gradient
  of the traceless second fundamental form is proven similarly as in
  lemma~\ref{thm:linfinityestimates}.
To prove~\eqref{eq:80} note
  that we can calculate the gradient of $\nu-\phi^{-2}\rho$ as
  follows. We let $e_i$ be a vector tangent to $\Sigma$ and calculate
  \begin{equation*}
    \nabla_{e_i} \nu
    =
    \half H e_i + \Acirc(e_i,\cdot).
  \end{equation*}
  Since $\phi^{-2}\rho$ is the normal to $S_r(0)$ in the Schwarzschild
  metric, and $S_r(0)$ is umbilical in this metric, we find that
  \begin{equation*}
    \nabla^S_{e_i} (\phi^{-2}\rho)
    =
    \half H_S(r) \big( e_i - g^S(e_i,\rho)\rho\big)
  \end{equation*}
  for $H_S(r) = \phi^{-2}\frac{2}{r} - 2\phi^{-3}\frac{m}{r^2}$.
  We calculate further and find
  \begin{equation*}
    g^S(e_i,\rho) = (g^S-g)(e_i,\rho)  + g(e_i, \rho - \phi^2\nu) + g(e_i,\phi^2\nu).
  \end{equation*}
  Note that the last term vanishes. In view of the estimates
  in~\eqref{eq:79} and definition~\ref{def:asymptotically-flat} we
  thus have 
  \begin{equation}
    \label{eq:76}
    \begin{split}
      \big|\nabla(\nu-\phi^{-2}\rho)\big|
      &\leq
      C\big(
      \rmin^{-1}|g-g^S|
      + |\nabla-\nabla^S|
      + |\Acirc|
      + |H-\bar H_S|
      \\
      &
      \qquad\qquad
      + |\bar H_S - H_S(r)|
      + \rmin^{-1} |\nu -\phi^{-2}\rho|
      \big)
      \\
      &\leq
      C\sqrt{\eta}\rmin^{-3} + C\rmin^{-1}|\nu-\phi^{-2}\rho|.
    \end{split}
  \end{equation}
  Proposition~\ref{thm:decay-curvature} then yields that
  \begin{equation*}
    \|\nabla(\nu-\phi^{-2}\rho)\|_{L^2}
    \leq
    C\sqrt{\eta}\rmin^{-1}.
  \end{equation*}
  We can now use the Michael-Simon-Sobolev inequality,
  proposition~\ref{thm:sobolev}, to get $L^4$-estimates
  \begin{equation*}
    \|\nu-\phi^{-2}\rho\|_{L^4}
    \leq
    C\sqrt{\eta} \rmin^{-1/2}.
  \end{equation*}
  Together with equation~\eqref{eq:76}, this implies $L^4$-bounds for
  the derivative of $\nu-\phi^{-2}\rho$. Thus an obvious modification
  of theorem 5.6 in \cite{Kuwert-Schatzle:2002} then yields the
  desired $L^\infty$-estimate:
  \begin{equation*}
    \|\nu-\phi^{-2}\rho\|_{L^\infty}
    \leq
    C\sqrt{\eta}\rmin^{-1}.
  \end{equation*}
  The estimates in~\eqref{eq:81} easily follow from~\eqref{eq:80}.
\end{proof}


%% file: linearization.tex
\section{Estimates for the linearized operator}
\label{sec:est-lin-op}
In this section we show that the linearized operator $W_\lambda=W-\lambda L$ is invertible.
\subsection{Eigenvalues of the Jacobi operator}
\label{sec:ev-jac-op}
To fix the notation let $\nu_i$ be the
$i$-th eigenvalue of the negative of the Laplace operator on $\mathbb{S}^2$, where we
count the eigenvalues with multiplicitites, i.e. $\nu_0=0,\,
\nu_1=\nu_2=\nu_3=2$, $ \nu_4= \ldots = \nu_8 = 4$ and $\nu_i>4$ for
$i\geq 9$. We denote by $\gamma^e_i$ the eigenvalues of the negative of the Laplace
operator on $\Sigma$, with respect to the Euclidean metric. We will need the following estimate from \cite[Corollary
1]{DeLellis-Muller:2006}.
\begin{theorem}
  \label{thm:ev-umbilical}
  There exist constants $C_i$ such that for every surface $\Sigma$
  as in theorem \ref{thm:final-curv}
  there holds
\begin{equation*}
|\gamma^e_i - R^{-2}_e\nu_i| \leq C_i \sqrt{\eta}r_\text{min}^{-4}\ .
\end{equation*} 
\end{theorem}
\begin{proof}
  Note that by theorem \ref{thm:final-curv} and lemma
  \ref{thm:general-to-schwarzschild} we have that
\begin{equation*}
\|\Acirc^e\|^2_{L^2(\Sigma,g^e)} \leq C \eta r_\text{min}^{-4}\ .
\end{equation*}
Scaling the estimate in  \cite[Corollary
1]{DeLellis-Muller:2006} gives the result.
\end{proof}
It can be checked from \cite{DeLellis-Muller:2006} that 
\begin{equation}\label{eq:82}
 C_i \leq C \nu_i\ ,
\end{equation}
where $C$ does not depend on $i$.

In the following we let $\bar{g}^S:=
\bar{\phi}^4 g^e$ be a uniform Schwarzschild-reference
metric on $\Sigma$. Thus $\bar{\Delta}^S:=\Delta^{\bar{g}^S}=
\bar{\phi}^{-4}\Delta^{g^e}$ and we denote the eigenvalues of
$-\bar{\Delta}^S$ by $\bar{\gamma}^S_i$.

\begin{corollary}
  \label{thm:cor-ev-umbilical}
  For any surface $\Sigma$
  as in theorem \ref{thm:final-curv}
  we have the estimate
\begin{equation*}
|\bar{\gamma}_i^S - R^{-2}_S\nu_i| \leq C_i \sqrt{\eta}r_\text{min}^{-4}.
\end{equation*}
\end{corollary}

To compute the eigenvalues of the Jacobi operator on $\Sigma$ we aim to compare it
with the operator 
\begin{equation}
  \label{eq:83}
  \bar{L}\alpha:= -\bar{\Delta}^S\alpha -\big(
  \tfrac{1}{2}\big(\bar{H}^S\big)^2-\lambda\big)\alpha.
\end{equation}
Let the eigenvalues and eigenfunctions of $L$ and $\bar{L}$ be denoted
by $\mu_i, \varphi_i$
and $\bar{\mu}_i, \bar{\varphi}_i$, respectively. Note that
\begin{equation}
  \label{eq:84}
\bar{\mu}_i=\bar{\gamma}_i^S-\tfrac{1}{2}\big(\bar{H}^S\big)^2+\lambda.
\end{equation}

\begin{lemma}
  \label{thm:ev-jac-op}
For any surface $\Sigma$ as in  theorem \ref{thm:final-curv}
  we have the estimate
\begin{equation*}
|\mu_i - \bar{\mu}_i| \leq C (|\bar{\mu}_i|+r_\text{min}^{-2}) \sqrt{\eta}r_\text{min}^{-2}\ .
\end{equation*}
\end{lemma}
\begin{proof}
We use the following characterization of the $i$-th eigenvalue
\begin{equation*}
\mu_i = \inf_{\substack{V\subset W^{1,2}(\Sigma)\\
              \text{dim}(V)=i+1}}\ \sup_{\psi \in V}
          \frac{\int_{\Sigma}\psi L\psi \, d\mu}{\int_\Sigma
            \psi^2\, d\mu},
\end{equation*}
where $V$ is any linear subspace of $W^{1,2}(\Sigma)$. Let
$\bar{\varphi}\in W^{1,2}(\Sigma)$ with $\int \bar{\varphi}^2\,
d\bar{\mu}^S =1$. We estimate,
using \eqref{eq:79} and \eqref{eq:81}
\begin{equation}
  \label{eq:85}\begin{split}
\int \bar\varphi L\bar{\varphi} \, d\mu &= \int |\nabla
\bar{\varphi}|^2-\bar{\varphi}^2\big(|\Acirc|^2+\frac{1}{2}H^2+\RicM(\nu,\nu)\big)\,
d\mu \\
&\leq \int |\nabla
\bar{\varphi}|^2-\bar{\varphi}^2\big(\frac{1}{2}(\bar{H}^S)^2-\lambda\big)\,
d\mu + C\sqrt{\eta}\rmin^{-4}.
\end{split}
\end{equation} In the following we repeatedly use the estimates from
definition \ref{def:asymptotically-flat} and lemma
\ref{thm:general-to-schwarzschild}. We can estimate the first term on
the right hand side by
\begin{equation}
\label{eq:86}\begin{split}
\int |\nabla \bar{\varphi}|^2\, d\mu &\leq \int |\nabla
\bar{\varphi}|^2\, d\mu^S + C \eta \rmin^{-2} \int |\nabla
\bar{\varphi}|_{g^S}^2\, d\mu^S\\
&\leq \int |\nabla\bar{\varphi}|_{g^S}^2\, d\mu^S + C \eta \rmin^{-2} \int |\nabla
\bar{\varphi}|_{g^S}^2\, d\mu^S\\
&\leq \int |\nabla\bar{\varphi}|_{\bar{g}^S}^2\, d\bar{\mu}^S + C \eta \rmin^{-2} \int |\nabla
\bar{\varphi}|_{\bar{g}^S}^2\, d\bar{\mu}^S,
\end{split}
\end{equation}
where we used the conformal invariance of the Dirichlet energy from
the second to the third line. The second term on the right hand side is estimated
similarly by
\begin{equation}
  \label{eq:87}
  \begin{split}
    -\int \bar{\varphi}^2\big(\tfrac{1}{2}(\bar{H}^S)^2-\lambda\big)\, d\mu
    &\leq
    -\int \bar{\varphi}^2\big(\tfrac{1}{2}(\bar{H}^S)^2-\lambda\big)\,
    d\mu^S 
    \\
    &\qquad
    + C \eta \rmin^{-2}\int
    \bar{\varphi}^2\big|\tfrac{1}{2}(\bar{H}^S)^2-\lambda\big|\, d\mu^S
    \\
    &\leq
    -\int \bar{\varphi}^2\big(\tfrac{1}{2}(\bar{H}^S)^2-\lambda)\,
    d\bar{\mu}^S + C \sqrt{\eta} \rmin^{-4},
  \end{split}
\end{equation}
where we used that $d\mu^S = (\phi/\bar{\phi})^4 d\bar{\mu}^S$ and
\begin{equation*}
  \left|\frac{\phi}{\bar{\phi}}-1\right|
  \leq
  C m \left|\frac{1}{r}-\frac{1}{R_e}\right|
  \leq
  \frac{C}{\rmin}\left( \tau + \frac{\sqrt{\eta}}{\rmin}\right)
  \leq
  C \sqrt{\eta}\rmin^{-2},
\end{equation*}
by proposition \ref{thm:sphere-approx-1} and  theorem
\ref{thm:position-estimate}. Now
\begin{equation}
  \label{eq:87a}
  \begin{split}
    \int |\nabla\bar{\varphi}|_{\bar{g}^S}^2\, d\bar{\mu}^S
    &=
    \int \bar{\varphi} \bar{L}\bar{\varphi} \, d\bar{\mu}^S
    + \int \bar{\varphi}^2 \big(\tfrac{1}{2}(\bar{H}^S)^2-\lambda)\,
    d\bar{\mu}^S
    \\
    &\leq
    \int \bar{\varphi} \bar{L}\bar{\varphi}\, d\bar{\mu}^S
    + C \rmin^{-2}.
  \end{split}
\end{equation}
Combining \eqref{eq:86}, \eqref{eq:87} and \eqref{eq:87a} we see that
\begin{equation}
  \label{eq:88}
  \int \bar{\varphi} L\bar{\varphi} \, d\mu
  \leq
  (1+C \eta \rmin^{-2})
  \int \bar\varphi \bar{L}\bar{\varphi}\, d\bar{\mu}^S
  + C \sqrt{\eta}\rmin^{-4}.
\end{equation}
Moreover, by arguing as above, we have the estimate
\begin{equation*}
  \left|\int \bar{\varphi}^2 \, d\mu -1\right|
  =
  \left|\int \bar{\varphi}^2\, d\mu -\int\bar{\varphi}^2\,
    d\bar{\mu}^S \right| 
  \leq C\sqrt{\eta}\rmin^{-2}.
\end{equation*}
Combining this with \eqref{eq:88} and the variational characterization
of the eigenvalues, we see that
\begin{equation*}
\mu_i\leq \bar{\mu}_i+ C\sqrt{\eta}\rmin^{-2}|\bar{\mu}_i|+
C\sqrt{\eta} \rmin^{-4}.
\end{equation*}
The reverse inequality follows from a similar calculation, interchanging
$\bar{L}$ and~$L$.
\end{proof}
From theorem \ref{thm:final-curv}, \eqref{eq:84} and lemma \ref{thm:ev-jac-op} we get the following
\begin{corollary}
  \label{thm:cor-ev-jac-op}
  For any surface $\Sigma$ as in  theorem \ref{thm:final-curv}
  we have the estimate
  \begin{equation*}
    \left|\mu_i - \left(\frac{\nu_i-2}{R_S^2}+3\lambda\right)\right|
    \leq C(1+\nu_i)\sqrt{\eta}\rmin^{-4} + C \rmin^{-5} + C \nu_i\sqrt{\eta} \rmin^{-6}.
  \end{equation*}
\end{corollary}

\subsection{The linearized Willmore equation}
\label{sec:lin-op}
In the following we aim at proving a positive lower bound for the
first eigenvalue of the linearization of the Willmore equation with prescribed area. We start by recalling the expression (see \eqref{eq:45})
\begin{equation*}
  \begin{split}
    \int_\Sigma \alpha W_\lambda \alpha \dmu &=\int_\Sigma \alpha W\alpha - \lambda \alpha L \alpha\dmu
    \\
    &\quad
    =
    \int_\Sigma
    (L\alpha)^2-  \lambda \alpha L \alpha
    + \half H^2 |\nabla\alpha|^2
    - 2 H\Acirc(\nabla\alpha,\nabla\alpha)
    \\
    &\qquad\quad
    + \alpha^2\big(
     |\nabla H|^2
    + 2 \omega(\nabla H)
    + H \Delta H + 2\la \nabla^2 H, \Acirc\ra\\ 
    &\qquad\quad\phantom{+ \alpha^2\big(}+ 2 H^2 |\Acirc|^2
    + 2 H\la \Acirc, T\ra
    - H\nabla_\nu\RicM(\nu,\nu)\\
    &\qquad\quad\phantom{+ \alpha^2\big(}- \half H^2|A|^2
    - \half H^2\RicM(\nu,\nu) 
    \big)\dmu .
  \end{split}
\end{equation*}
Integration by parts of the third term on the right yields
\begin{equation*}
\frac{1}{2}\int_\Sigma  H^2 |\nabla\alpha|^2\dmu = \frac{1}{2}\int_\Sigma
\alpha^2\big(|\nabla H|^2 + H \Delta H\big) - H^2 \alpha\Delta
\alpha\dmu .
\end{equation*}   
Together with $L\alpha = -\Delta \alpha -\alpha\big(|A|^2 +
\RicM(\nu,\nu)\big)$ and \eqref{eq:1} this yields 
\begin{equation*}
  \begin{split}
    &\int_\Sigma \alpha W_\lambda \alpha \dmu
    \\
    &=
    \int_\Sigma
    (L\alpha)^2 +\half H^2 \alpha L \alpha -  \lambda \alpha L \alpha
    - 2 H\Acirc(\nabla\alpha,\nabla\alpha)
    \\
    &\quad
    + \alpha^2\big(
     \tfrac{3}{2}|\nabla H|^2
    - \tfrac{3}{2}(H^2|\Acirc|^2 + H^2 \RicM(\nu,\nu) + \lambda H^2)+ 2 \omega(\nabla H)\\
    & \quad \phantom{+\alpha^2\big(} 
    + 2\la \nabla^2 H, \Acirc\ra 
    + 2 H^2 |\Acirc|^2
    + 2 H\la \Acirc, T\ra
    - H\nabla_\nu\RicM(\nu,\nu) 
    \big)\dmu .
  \end{split}
\end{equation*}
To understand the last term on the RHS above we recall that the
Einstein tensor is divergence free and \eqref{eq:62a}, which implies
\begin{equation}\label{nablaRic}\begin{split}
\nabla_\nu\RicM(\nu,\nu) &= -\nabla_{e_i} G(\nu,e_i) -\half \nabla_\nu\ScalM\\
 &= - \divSig \omega  + \la \Acirc, G^T \ra +\half H \ScalM 
  -\tfrac{3}{2} H \RicM(\nu,\nu) - \half \nabla_\nu\ScalM .
\end{split}
\end{equation}
Note that 
\begin{equation*}
\begin{split}
\frac{1}{2}\int_\Sigma  H^2 \alpha L\alpha \dmu &= \frac{1}{2} \int_\Sigma
(H\alpha)L(H\alpha) + \alpha^2 H \Delta H + H\langle\nabla H,
\nabla(\alpha^2)\rangle\, \dmu\\
&= \frac{1}{2} \int_\Sigma (H\alpha)L(H\alpha) - \alpha^2 |\nabla H|^2\,
\dmu .
\end{split}\end{equation*}
Putting all together we arrive at
\begin{equation}
\label{eq:89}
  \begin{split}
    &\int_\Sigma \alpha W_\lambda \alpha \dmu
    \\
    &=\int_\Sigma L\alpha(L\alpha-3\lambda \alpha) +\half \big((H\alpha) L(H\alpha)
    -3\lambda (H\alpha)^2\big) + 2 \lambda \alpha L \alpha\\
    &\qquad- 2 H\Acirc(\nabla\alpha,\nabla\alpha)
    + \alpha^2\big(
      |\nabla H|^2
    + \tfrac{1}{2}H^2|\Acirc|^2
    + 2\la \nabla^2 H, \Acirc\ra\\ 
    &\qquad 
    + H\la \Acirc, T\ra 
    -\half H^2 \ScalM
    + \half H \nabla_\nu\ScalM
    + H \divSig \omega
    + 2 \omega(\nabla H)
    \big)\dmu .
  \end{split}
\end{equation}

We decompose $W^{2,2}(\Sigma)$ using the eigenspaces of $L$, more
precisely consider the $L^2(\Sigma)$-orthonormal decomposition $W^{2,2}(\Sigma)=
V_0\oplus V_1 \oplus V_2$ where $V_0=\text{span}\{\varphi_0\},\
V_1=\text{span}\{\varphi_1,\varphi_2,\varphi_3\},\
V_2=\text{span}\{\varphi_4,\varphi_5,\ldots\}$. For any $\alpha \in
W^{2,2}(\Sigma)$ let $\alpha_0,\alpha_1,\alpha_2$ be the respective
orthogonal projections on these subspaces.
Our aim is to show that $\int\alpha W_\lambda \alpha$ is
positive on $V_0^\perp$. 
\begin{lemma}
\label{thm:est-lin-willmore-1}
For any surface $\Sigma$ as in theorem \ref{thm:final-curv}
  we have the estimate
\begin{equation*}
\int_\Sigma L\alpha(L\alpha -3\lambda \alpha) + 2\lambda \alpha L
\alpha \dmu \geq \big(24m^2 R_S^{-6} - C \sqrt{\eta}\rmin^{-7} - C
\rmin^{-8}\big) \int_\Sigma \alpha^2 \dmu
\end{equation*}
for all $\alpha \in V_0^\perp$.
\end{lemma}
\begin{proof}
This follows from the estimates on the eigenvalues of $L$ in corollary
\ref{thm:cor-ev-jac-op} and theorem \ref{thm:final-curv}.
\end{proof}
\begin{lemma}
  \label{thm:est-lin-willmore-2}
  For any surface $\Sigma$ as in theorem \ref{thm:final-curv}
  we have the estimate
  \begin{equation*}
    \begin{split}
      & \int_\Sigma (H\alpha)L(H\alpha) -3\lambda (H\alpha)^2 \dmu
      \\
      &\geq
      -C \sqrt{\eta} \rmin^{-6} \int_\Sigma \alpha_1^2 \dmu
      + \frac{1}{4}\rmin^{-2} \int_\Sigma |\nabla \alpha_2|^2\dmu
      + \frac{1}{4}\rmin^{-4}\int_\Sigma \alpha_2^2\dmu
    \end{split}
  \end{equation*}
  for all $\alpha \in V_0^\perp$.
\end{lemma}
\begin{proof}
We can write
\begin{equation*}
\begin{split}
\int_\Sigma (H\alpha)L(H\alpha) -3\lambda (H\alpha)^2 \dmu
=&\int_\Sigma (H\alpha_1)L(H\alpha_1) -3\lambda (H\alpha_1)^2 \dmu\\
&+2 \int_\Sigma (H\alpha_1)L(H\alpha_2) -3\lambda (H\alpha_1)(H\alpha_2) \dmu\\
&+\int_\Sigma (H\alpha_2)L(H\alpha_2) -3\lambda (H\alpha_2)^2 \dmu ,
\end{split}
\end{equation*}
and we denote the terms on the RHS by $(i), (ii)$ and
$(iii)$. Note that we can always estimate
\begin{equation}
  \label{eq:90}
  \begin{split}
    \left|\langle H\alpha_i,\varphi_j\rangle_{L^2(\Sigma)}\right|
    &=
    \left|\int_\Sigma H \alpha_i \varphi_j \dmu \right|
    \\
    &\leq
    \int_\Sigma  |H-\bar{H}^S||\alpha_i|\,|\varphi_j|\dmu
    \leq
    C \sqrt{\eta} \rmin^{-3}\left(\int_\Sigma \alpha_i^2\dmu\right)^{1/2}
\end{split}
\end{equation}
for $i\neq j$. So we see
\begin{equation}
  \label{eq:91}
  \begin{split}
    (i)
    &\geq
    -|\mu_0-3\lambda| \int_\Sigma |(H\alpha_1)_0|^2\dmu
    - \max_{j=1,2,3}|\mu_j-3\lambda|\int_\Sigma
    \big|(H\alpha_1)_{V_0^\perp}\big|^2\dmu
    \\
    &\geq
    -C \eta \rmin^{-8} \int_\Sigma \alpha_1^2 \dmu
    - C\sqrt{\eta}\rmin^{-4}
    \int_\Sigma \big|(H\alpha_1)_{V_0^\perp}\big|^2\dmu
    \\
    &\geq
    -C \eta \rmin^{-8} \int_\Sigma \alpha_1^2 \dmu
    - C\sqrt{\eta}\rmin^{-4} \left(\int_\Sigma (H\alpha_1)^2\dmu
      + C \eta \rmin^{-6}\int_\Sigma \alpha_1^2\dmu\right)
    \\
    &\geq
    -C \eta \rmin^{-8} \int_\Sigma \alpha_1^2 \dmu
    - C\sqrt{\eta}\rmin^{-6} \int_\Sigma \alpha_1^2\dmu .
\end{split}
\end{equation}
To estimate $(ii)$ we write
\begin{equation}
  \label{eq:92}
  \begin{split}
    (ii)
    &\geq
    -2 |\mu_0-3\lambda|\int_\Sigma
    \big|(H\alpha_1)_0\big|\big|(H\alpha_2)_0\big|\dmu 
    \\
    &\quad
    - 2\max_{j=1,2,3}|\mu_j-3\lambda|\int_\Sigma \big|(H\alpha_1)\big|
    \big|(H\alpha_2)_1\big|\dmu
    \\
    &\quad
    + 2\int_\Sigma(H\alpha_1)_2(L-3\lambda)(H\alpha_2)\dmu
    \\
    &\geq
    -C\eta \rmin^{-8}\left(\int_\Sigma \alpha_1^2\dmu\right)^{1/2}
    \left(\int_\Sigma \alpha_2^2\dmu\right)^{1/2}
    \\
    &\quad    
    + 2\int_\Sigma(H\alpha_1)_2(L-3\lambda)(H\alpha_2)\dmu .
  \end{split}
\end{equation}
For the last term in \eqref{eq:92} we write
\begin{equation*}
  (H\alpha_1)_2 = H\alpha_1-\sum_{j=0}^3 \langle H\alpha_1,\varphi_j
  \rangle \varphi_j = \sum_{j=0}^3 \beta_j\varphi_j ,
\end{equation*}
where $\beta_j=H\langle \alpha_1,\varphi_j\rangle - \langle H\alpha_1,\varphi_j
\rangle$. Note that 
\begin{equation*}
  |\nabla \beta_j|
  \leq
  C \sqrt{\eta}\rmin^{-4}\left(\int_\Sigma \alpha_1^2\dmu \right)^{1/2}
\end{equation*}
and
\begin{equation*}
  |\beta_j| \leq 2||H-\bar{H}^S||_{L^\infty}\sum_{j=0}^3|\langle \alpha_1,
  \varphi_j\rangle| \leq C \sqrt{\eta} \rmin^{-3}\left(\int_\Sigma
  \alpha_1^2\dmu\right)^{1/2}.
\end{equation*}
Then
\begin{equation}
  \label{eq:93}
  \begin{split}
    &2\int_\Sigma(H\alpha_1)_2(L-3\lambda)(H\alpha_2)\dmu
    \\
    &=
    2\int_\Sigma \left\langle \nabla \sum_{j=0}^3 \beta_j \varphi_j, \nabla (H\alpha_2)\right\rangle 
    -
    (H\alpha_1)_2(H\alpha_2)\big(|A|^2+\RicM(\nu,\nu)+3\lambda\big)\dmu
    \\
    &\geq
    -C \left(
      \int_\Sigma \sum_{j=0}^3|\nabla\beta_j \varphi_j + \beta_j\nabla
      \varphi_j|^2\dmu\right)^{1/2}
    \left(\int_\Sigma |\nabla(H \alpha_2)|^2\dmu \right)^{1/2}
    \\
    &\quad
    - C\sqrt{\eta}\rmin^{-6}
    \left(\int_\Sigma \alpha_1^2\dmu\right)^{1/2}
    \left(\int_\Sigma \alpha_2^2\dmu\right)^{1/2}
    \\
    &\geq
    -C \sqrt{\eta} \rmin^{-4}
    \left(\int_\Sigma \alpha_1^2\dmu\right)^{1/2}
    \left(\int_\Sigma |\nabla(H\alpha_2)|^2\dmu \right)^{1/2}
    \\
    & \quad
    - C\sqrt{\eta}\rmin^{-6} \left(\int_\Sigma \alpha_1^2\dmu\right)^{1/2}
    \left(\int_\Sigma \alpha_2^2\dmu\right)^{1/2} ,
\end{split}
\end{equation}
where we used in the last step that $\int_\Sigma |\nabla \varphi_j|^2\dmu
\leq C\rmin^{-2}$ for $0\le j \le 3$. This follows from
\begin{equation}
  \label{eq:93a}
  \begin{split}
    \int_\Sigma |\nabla \varphi_j|^2\dmu
    &=
    \int_\Sigma \varphi_jL\varphi_j + \varphi_j^2\big(|A|^2+\RicM(\nu,\nu)\big)\dmu
    \\
    &\leq
    (\max_{j=0,1,2,3}|\mu_j| + C\rmin^{-2})\int_\Sigma \varphi_j^2\dmu
    \leq  
    C\rmin^{-2}.
\end{split}
\end{equation}
Putting \eqref{eq:93} and \eqref{eq:92} together we
see
\begin{equation}
\label{eq:94}
\begin{split}
(ii)&\geq -C \eta \varepsilon^{-1} \rmin^{-8} \int_\Sigma \alpha_1^2\dmu
-\varepsilon \rmin^{-4} \int_\Sigma \alpha_2^2\dmu
- \varepsilon
\int_\Sigma |\nabla(H\alpha_2)|^2\dmu\\
&\geq -C \eta \varepsilon^{-1} \rmin^{-8} \int_\Sigma \alpha_1^2\dmu
-\varepsilon \rmin^{-4} \int_\Sigma \alpha_2^2\dmu
- 2\varepsilon
\int_\Sigma H^2|\nabla\alpha_2|^2\dmu\\
&\ \ \ \ -C \varepsilon \eta \rmin^{-8}\int_\Sigma \alpha_2^2\dmu
\end{split}
\end{equation}
for an arbitrary $\varepsilon>0$.

For the term $(iii)$, we see
\begin{equation*}
  \begin{split}
    (iii)
    &\geq
    -|\mu_0-3\lambda|\int_\Sigma \big|(H\alpha_2)_0\big|^2\dmu
    -\max_{j=1,2,3}|\mu_j-3\lambda| \int_\Sigma \big|(H\alpha_2)_1\big|^2\dmu
    \\
    &\quad
    + \int_\Sigma (H\alpha_2)_2(L-3\lambda)(H\alpha_2)_2\dmu
    \\
    &\geq
    -C\eta\rmin^{-8}\int_\Sigma \alpha_2^2\dmu
    + \int_\Sigma (H\alpha_2)_2(L-3\lambda)(H\alpha_2)_2\dmu .
  \end{split}
\end{equation*}
If $\beta \in V_2$ and $\delta>0$ are arbitrary, we have the estimate
\begin{equation*}
  \begin{split}
    &\int_\Sigma \beta(L-3\lambda)\beta\dmu
    =
    \int_\Sigma \delta |\nabla \beta|^2
    + \beta (L+\delta \Delta -3\lambda)\beta\dmu
    \\
    &=
    \int_\Sigma \delta |\nabla \beta|^2
    + (1-\delta)\beta (L-3 \lambda)\beta
    - \delta \beta^2\big(|A|^2+\RicM(\nu,\nu)+ 3\lambda\big)\dmu
    \\
    &\geq
    \delta \int_\Sigma |\nabla \beta|^2\dmu
    + (1-\delta)\big((\nu_4-2)R_S^{-2} - C\sqrt{\eta}\rmin^{-4}\big)
    \int_\Sigma \beta^2\dmu
    \\
    &\quad
    - 3\delta \rmin^{-2}\int_\Sigma \beta^2\dmu
    \\
    &\geq
    \delta \int_\Sigma |\nabla \beta|^2\dmu
    + (1-4\delta)\rmin^{-2}\int_\Sigma \beta^2\dmu .
  \end{split}
\end{equation*}
With $\beta=(H\alpha_2)_2$ and $\delta =1/5$ this yields
\begin{equation}
  \label{eq:95}
  \begin{split}
    (iii)
    &\geq
    -C\eta\rmin^{-8}\int_\Sigma \alpha_2^2\dmu
    +\frac{1}{5} \int_\Sigma |\nabla(H\alpha_2)_2|^2\dmu
    + \frac{1}{5} \rmin^{-2}\int_\Sigma \big|(H\alpha_2)_2\big|^2\dmu
    \\
    &\geq
    \frac{2}{5}\rmin^{-2}\int_\Sigma |\nabla \alpha_2|^2\dmu
    + \frac{2}{5}\rmin^{-4}\int_\Sigma \alpha_2^2\dmu ,
  \end{split}
\end{equation}
where we used that
\begin{equation*}
  \begin{split}
    \int_\Sigma \big|(H\alpha_2)_2\big|^2\dmu
    &=
    \int_\Sigma H^2\alpha_2^2\dmu
    -\int_\Sigma \big|(H\alpha_2)_0\big|^2 + \big|(H\alpha_2)_1\big|^2\dmu
    \\
    &\geq
    \int_\Sigma H^2\alpha_2^2\dmu
    - C\eta\rmin^{-6}\int_\Sigma \alpha_2^2\dmu
    \geq
    3\rmin^{-2}\int_\Sigma \alpha_2^2\dmu ,
  \end{split}
\end{equation*}
and 
\begin{equation*}
  \begin{split}
    &\int_\Sigma |\nabla(H\alpha_2)_2|^2\dmu
    =
    \int_\Sigma |\nabla(H\alpha_2 - \sum_{j=0}^3\langle
    H\alpha_2,\varphi_j\rangle \varphi_j)|^2\dmu
    \\
    &\geq
    \frac{3}{4} \int_\Sigma |\nabla (H\alpha_2)|^2 \dmu
    - C \sum_{j=0}^3|\langle H\alpha_2,\varphi_j\rangle|^2
    \int_\Sigma |\nabla \varphi_j|^2\dmu
    \\
    &\geq
    \frac{2}{3}\int_\Sigma H^2|\nabla\alpha_2|^2\dmu
    - C\int_\Sigma |\nabla H|^2\alpha_2^2\dmu
    - C\eta\rmin^{-8}\int_\Sigma \alpha_2^2\dmu
    \\
    &\geq
    2\rmin^{-2}\int_\Sigma |\nabla \alpha_2|^2\dmu
    -C \eta \rmin^{-8}\int_\Sigma \alpha_2^2\dmu .
  \end{split}
\end{equation*}
Combining the estimates for $(i),(ii)$ and $(iii)$, and choosing
$\varepsilon = 1/100$ and $r_0$ big enough we arrive at the claimed
statement.
\end{proof}

\begin{theorem}
\label{est-lin-willmore-3}
In addition to the hypotheses of theorem \ref{thm:final-curv}, there
exists $\eta_0$ and $r_0$, depending only on $m,\sigma$ and
$\varepsilon$ such that on such a surface $\Sigma$ it holds
\begin{equation*}
  \int_\Sigma \alpha W_\lambda \alpha  \dmu \geq 12 m^2
  R_S^{-6} \int \alpha^2 \dmu
\end{equation*}
for all $\alpha \in V_0^\perp$.
\end{theorem}
\begin{proof}
By lemma \ref{thm:est-lin-willmore-1} and lemma \ref{thm:est-lin-willmore-2} we
only have to check that the remaining terms in \eqref{eq:89} have
the right decay. First we note that by arguing as in the estimate \eqref{eq:93a} we get
\begin{equation*}
  \begin{split}
    \int_\Sigma |\nabla \alpha_1|^2\dmu  \leq  
    C\rmin^{-2}\int_\Sigma \alpha_1^2\dmu .
  \end{split}
\end{equation*}
Thus we have
\begin{equation*}
  \begin{split}
    \left|\int_\Sigma 2H \Acirc (\nabla \alpha,\nabla \alpha)\dmu\right|
    &\leq
    C \sqrt{\eta}\rmin^{-4} \int_\Sigma |\nabla \alpha_1|^2
    + |\nabla \alpha_2|^2\dmu
    \\
    &\leq
    C \sqrt{\eta}\rmin^{-4} \left(
      \rmin^{-2}\int_\Sigma \alpha_1^2\dmu +\int_\Sigma |\nabla \alpha_2|^2\dmu
    \right) .
  \end{split}
\end{equation*}
We rewrite
\begin{equation*}
  \begin{split}
    \int_\Sigma 2 \alpha^2\langle \nabla^2H,\Acirc\rangle \dmu
    &=
    -\int_\Sigma 4\alpha\nabla_i\alpha \nabla_jH \Acirc_{ij}
    + 2\alpha^2 \langle \nabla H, \div \Acirc \rangle \dmu.
  \end{split}
\end{equation*}
Furthermore
\begin{equation*}
  \left|\int_\Sigma 2\alpha^2 \langle \nabla H, \div \Acirc \rangle \dmu
  \right|
  =
  \left|\int_\Sigma 2\alpha^2 \langle \nabla H, \half\nabla H
    + \omega \rangle \dmu \right|
  \leq
  C \eta \rmin^{-8} \int_\Sigma \alpha^2\dmu,
\end{equation*}
and 
\begin{equation*}
  \begin{split}
    &
    \left|\int_\Sigma 4 \alpha \nabla_i\alpha \nabla_j H \Acirc_{ij}\dmu \right|
    \leq
    C\eta \rmin^{-7}\left(\int_\Sigma \alpha^2\dmu\right)^{1/2}
    \left(\int_\Sigma |\nabla \alpha|^2\dmu\right)^{1/2}
    \\
    &\leq
    C \eta \rmin^{-7}\left(\int_\Sigma \alpha^2\dmu \right)^{1/2}
    \left(\rmin^{-2}\int_\Sigma \alpha_1^2\dmu
      +\int_\Sigma |\nabla\alpha_2|^2\dmu\right)^{1/2}
    \\
    &\leq
    C \eta \rmin^{-8} \int_\Sigma \alpha^2 \dmu
    + C\eta \rmin^{-6}\int_\Sigma |\nabla\alpha_2|^2\dmu .
  \end{split}
\end{equation*}
In view of the estimates of theorem \ref{thm:final-curv} we find
\begin{equation*}
  \begin{split}
    &\Bigg|\int_\Sigma
    \alpha^2\big( |\nabla H|^2  + \tfrac{1}{2}H^2|\Acirc|^2 + H\la \Acirc, T\ra 
    -\half H^2 \ScalM  + \half H \nabla_\nu\ScalM
    \\ 
    &
    + H \divSig \omega
    + 2 \omega(\nabla H)
    \big)\dmu \Bigg|
    \leq
    C \sqrt{\eta} \rmin^{-6}\int_\Sigma \alpha^2 \dmu .
  \end{split}
\end{equation*}
Altogether this finishes the proof of the theorem.
\end{proof}


\subsection{Invertibility of the linearized operator}
\label{sec:inv-lin-op}
In this subsection we show that the linearized operator
$W_\lambda$ is invertible. In order to do this, we need
good estimates for the projection of a function onto $V_0$. We start
with a different calculation for the first eigenvalue $\mu_0$ of $L$.
\begin{lemma}
  \label{mu0}
  For any surface $\Sigma$ as in theorem \ref{thm:final-curv}
  we have the estimate
  \begin{equation}
    \label{mu0a}
    \big|\mu_0+|A|^2+\RicM(\nu,\nu)\big|
    \le
    C\sqrt{\eta} \rmin^{-4}.
  \end{equation}
\end{lemma}
\begin{proof}
  From theorem \ref{thm:final-curv} we know that
  \begin{align*}
    \big|\half \bar H_S^2-|A|^2\big|
    \le
    C\sqrt{\eta} \rmin^{-4}
  \end{align*}
  and
  \begin{align*}
    \left|\half \bar H_S^2-\frac{2}{R_S^2}+\frac{4m}{R_S^3}\right|
    \le
    C\rmin^{-5}.
  \end{align*}
  Combining these two estimates with theorem \ref{thm:final-curv} and
  corollary \ref{thm:cor-ev-jac-op} we get
  \begin{equation*}
    \begin{split}
      \big|\mu_0+|A|^2+\RicM(\nu,\nu)\big|
      &
      \le
      \left|3\lambda -\frac{2}{R_S^2}+\frac{2}{R_S^2}-\frac{6m}{R_S^3}\right|
      +C\sqrt{\eta}\rmin^{-4}
      \le
      \frac{C\sqrt{\eta}}{\rmin^{4}}.
    \end{split}
  \end{equation*}
\end{proof}
Next we prove a $W^{2,2}$-estimate for the eigenfunction of $L$
corresponding to the eigenvalue~$\mu_0$.
\begin{lemma}
  \label{eigen}
  Let $\Sigma$ be a surface as in theorem \ref{thm:final-curv} and let
  $u\in C^\infty(\Sigma)$ be a solution of $L u=\mu_0 u$. Then we have
  \begin{equation}
    \int_\Sigma |u-\bar u|^2 \dmu
    +\rmin^2 \int_\Sigma |\nabla u|^2 \dmu
    +\rmin^6 \int_\Sigma |\nabla^2 u|^2 \dmu
    \le
    C\sqrt{\eta} \rmin^{-2}\|u\|^2_{L^2(\Sigma)},\label{eigena}
  \end{equation}
  where $\bar u=|\Sigma|^{-1} \int_\Sigma u \dmu$. Moreover we have the pointwise estimate
  \begin{align}
    \|u-\bar u\|_{L^\infty(\Sigma)}
    \le
    C \eta^{1/4}\rmin^{-2}\|u\|_{L^2(\Sigma)}. \label{eigenb}
  \end{align}
\end{lemma}
\begin{proof}
  By a scaling argument we see that we can assume without loss of generality that $\|u\|_{L^2(\Sigma)}=1$. Using the definition of $L$ and lemma \ref{mu0} we get
  \begin{align*}
    \int_\Sigma |\nabla u|^2 \dmu
    &=
    \int_\Sigma uLu +u^2\big(|A|^2+\RicM(\nu,\nu)\big)\dmu
    \\
    &=
    \int_\Sigma u^2\big(\mu_0+|A|^2+\RicM(\nu,\nu)\big)\dmu
    \\
    &\le
    C\sqrt{\eta} \rmin^{-4}.
  \end{align*}
  In view of theorem \ref{thm:ev-umbilical} there is a Poincar\'e
  inequality on $\Sigma$ with constant close to the one on
  $S^{2}_{R}$. This yields
  \begin{equation*}
    \int_\Sigma |u-\bar u|^2 \dmu
    \le
    cR_S^2\|\nabla u\|^2_{L^2(\Sigma)}
    \le C\sqrt{\eta} \rmin^{-2}.
  \end{equation*}
 Similarly as above we calculate
  \begin{align*}
    \int_\Sigma |\Delta u|^2 \dmu
    &=
    \int_\Sigma (Lu)^2 +2uLu\big(|A|^2+\RicM(\nu,\nu)\big)
    \\
    &\quad
    +u^2\big(|A|^2+\RicM(\nu,\nu)\big)^2\dmu
    \\
    &=
    \int_\Sigma u^2\big(\mu_0+|A|^2+\RicM(\nu,\nu)\big)^2\dmu.
  \end{align*}
  Hence, again by lemma \ref{mu0}, we get the estimate
  \begin{align*}
    \int_\Sigma |\Delta u|^2 \dmu\le C\eta \rmin^{-8}.
  \end{align*}
  Integrating by parts and interchanging derivatives as in
  \eqref{eq:15} (note that by doing this we get an additional Gauss
  curvature term from which we now know that it is positive) we
  conclude
  \begin{align*}
    \int_\Sigma |\nabla^2 u|^2 \dmu
    \le
    \int_\Sigma |\Delta u|^2  \dmu
    \le
    C\sqrt{\eta}\rmin^{-8}.
  \end{align*}
  Lemma \ref{thm:interpolation} and the previous estimates now give
  \begin{align*}
    \|u-\bar u\|_{L^\infty(\Sigma)}^4
    \le
    C \int_\Sigma |u-\bar u|^2 \dmu \int_\Sigma |\nabla^2 u|^2 +H^4
    |u-\bar u|^2\dmu
    \le
    C\eta \rmin^{-8}.
\end{align*}
This finishes the proof of the lemma.
\end{proof}
In the following lemma we show an $L^2$-estimate for solutions of
$W_\lambda u=f$.
\begin{lemma}
  \label{const}
  Let $\delta>0$, let $\Sigma$ be a surface as in theorem \ref{thm:final-curv} and let
  $u\in C^\infty(\Sigma)$ be a solution of $W_\lambda u=f$ with
  $\int_\Sigma (f-f_0)^2 \dmu\le
  \delta R_S^{-12}\|u\|_{L^2(\Sigma)}^{2}$, where $f_0$ and $u_0$ are
  the projections of $f$ respectively $u$ onto $V_0$. Then we have
  \begin{align}
    \|u-u_0\|_{L^2(\Sigma)}
    \le
    C(\sqrt{\delta}+\sqrt{\eta}+R_S^{-1})\|u\|_{L^2(\Sigma)}.
    \label{consta}
  \end{align}
\end{lemma} 
\begin{proof}
  By a scaling argument we see that we can assume without loss of
  generality that $\|u\|_{L^2(\Sigma)}=1$. Next we combine our
  assumption with equation \eqref{eq:44} and the fact that $Lu_0=\mu_0
  u_0$ to get
  \begin{align}
   W_\lambda(u-u_0)=&
      f- \mu_0u_0(\mu_0
      + \half H^2 -\lambda)
      + 2H\la \Acirc, \nabla^2 u_0\ra
      + 2H\omega(\nabla u_0)
      \nonumber \\
      &+ 2\Acirc(\nabla u_0,\nabla H)
      +u_0\big(
      |\nabla H|^2
      + 2 \omega(\nabla H)
      + H \Delta H \nonumber \\
      &+ \la \nabla^2 H, \Acirc\ra +
      2 H^2 |\Acirc|^2
      + 2 H\la \Acirc, T\ra
      - H\nabla_\nu\RicM(\nu,\nu)
      \big) .
    \label{eq:66}
  \end{align}
  With the help of theorem \ref{est-lin-willmore-3} we conclude
  \begin{align*}
    \int_\Sigma (u-u_0)W_\lambda(u-u_0)\dmu
    \ge
    12m^2 R_S^{-6} \int_\Sigma (u-u_0)^2 \dmu.
  \end{align*}
  To get an upper bound for this intergral we multiply
  equation~\eqref{eq:66} by $(u-u_0)$ and estimate term by term. We
  start with the term involving $f$
  \begin{align*}
    \left|\int_\Sigma f(u-u_0)\dmu\right|
    &=
    \left|\int_\Sigma (f-f_0)(u-u_0)\dmu\right|
    \\
    &\le
    m^{-2}R_S^{6} \int_\Sigma (f-f_0)^2 \dmu
    + m^2 R_S^{-6} \int_\Sigma (u-u_0)^2 \dmu
    \\
    &\le
    C\delta R_S^{-6}+m^2 R_S^{-6} \int_\Sigma (u-u_0)^2 \dmu.
  \end{align*}
  Next, using a variant of lemma \ref{mu0}, we estimate
  \begin{align*}
    &\left|\int_\Sigma \mu_0u_0(\mu_0 + \half H^2
      -\lambda)(u-u_0)\dmu\right|
    \\
    &\leq
    m^2 R_S^{-6} \int_\Sigma (u-u_0)^2 \dmu
    +m^{-2} R_S^{2}\int_\Sigma u_0^2(\mu_0 + \half H^2 -\lambda)^2
    \dmu
    \\
    &\le
    m^2 R_S^{-6} \int_\Sigma (u-u_0)^2 \dmu+C\eta R_S^{-6}.
  \end{align*} 
  Now we estimate all terms containing derivatives of $u_0$. By
  arguing as before we see that we only have to bound the term
  \begin{align*}
   C m^2 R_S^{6}\int_\Sigma
    H^2 |\Acirc|^2 |\nabla^2 u_0|^2
    +|\nabla u_0|^2(H^2 |\omega|^2
    +|\Acirc|^2 |\nabla H|^2)\dmu
    \le
    CR_S^{-8},
  \end{align*}
  where we used theorem \ref{thm:final-curv} and lemma
  \ref{eigen}. Finally we estimate the terms involving $u_0$. We start
  with
  \begin{align*}
    &R_S^{6}\int_\Sigma u_0^2 \big(
    |\nabla H|^4
    +|\omega|^2 |\nabla H|^2
    +H^2|\Delta H|^2
    +|\Acirc|^2|\nabla^2 H|^2
    +H^4|\Acirc|^4\big) \dmu 
    \\
    &\le
    C R_S^{-8}+cR_S^{4}\int_\Sigma u_0^2 |\Delta H|^2\dmu
    +C\eta \int_\Sigma u_0^2|\nabla^2 H|^2\dmu
    \\
    &\le
    C R_S^{-8}
    + C R_S^{2}\int_\Sigma
    u_0^2\big(|\Acirc|^4+\lambda+\RicM(\nu,\nu)\big)^2\dmu
    \\
    &\le
    C R_S^{-8}+C\eta R_S^{-6},
  \end{align*}
  where we used lemma \ref{thm:l2H}, theorem \ref{thm:final-curv} and
  lemma \ref{eigen}. In the third term in the second line we can use
  lemma \ref{eigen} to replace $u_0^2$ by $\bar{u}_0^2$.
  Finally, we use \eqref{nablaRic} and Theorem
  \ref{thm:final-curv} to get
  \begin{align*}
    &\left|\int_\Sigma (u-u_0)u_0 H\nabla_\nu
      \RicM(\nu,\nu)\dmu\right|
    \\
    &\le
    \frac{3}{2}\left|\int_\Sigma (u-u_0)u_0 H^2\RicM(\nu,\nu)\dmu\right|
    +m^2 R_S^{-6}\int_\Sigma (u-u_0)^2\dmu
    +\frac{C\eta}{R_S^{6}}.
  \end{align*}
  Now we use the $L^2$-orthogonality of $u_0$ and $u-u_0$ to estimate
  \begin{align*}
    &\frac{3}{2}\left|\int_\Sigma (u-u_0)u_0  H^2\RicM(\nu,\nu)\dmu\right|
    \\
    &\le
    \frac{3}{2}\left|\int_\Sigma (u-u_0)u_0  H^2\big(
        \RicM(\nu,\nu)+\tfrac{2m}{R_S^{3}}
      \big)\dmu\right|
    \\
    &\qquad
    +3mR_S^{-3}\left|\int_\Sigma (u-u_0)u_0 (H^2-4R^{-2}_S)\dmu\right|
    \\
    &\le
    m^2 R_S^{-6}\int_\Sigma (u-u_0)^2\dmu+C\eta R_S^{-6}.
  \end{align*}
  Combining all these estimates we get
  \begin{align*}
    3m^2 R_S^{-6}\int_\Sigma (u-u_0)^2\dmu\le C R_S^{-6}(\delta+\eta+R_S^{-2})
  \end{align*}
  which finishes the proof of the lemma.
\end{proof}
From the proof of the lemma we directly obtain the following 
\begin{corollary}
  \label{coro}
  Let $\delta>0$, let $\Sigma$ be a surface as in theorem \ref{thm:final-curv} and let
  $u\in C^\infty(\Sigma)$. Then we have
  \begin{equation} 
    \left|\int_\Sigma (u-u_0)W_\lambda u_0 \dmu\right|
    \le
    \frac{4m^2}{R_S^{6}}\|u-u_0\|_{L^2}^2 +\frac{c}{R_S^{6}}(\delta+\eta+R_S^{-2})\|u\|^2_{L^2(\Sigma)}.\label{trick}
  \end{equation}
  Moreover, if $u$ is a solution of $W_\lambda u=f$ with
  \begin{equation*}
    \int_\Sigma (u-u_0)f \dmu
    \le
    \delta R_S^{-6}\|u\|_{L^2(\Sigma)}
    \|u-u_0\|_{L^2(\Sigma)},
  \end{equation*}
  then we have
  \begin{align}
    \|u- u_0\|_{L^2(\Sigma)}
    \le
    C (\sqrt{\delta}+\sqrt{\eta}+R_S^{-1})\|u\|_{L^2(\Sigma)}.
    \label{coroa}
  \end{align}
\end{corollary} 
In the following lemma we prove $L^2$-estimates for the operator $W_\lambda$.
\begin{lemma}\label{thm:W-estimates}
 Let $\Sigma$ be as in theorem \ref{thm:final-curv}. Then we have
  \begin{equation*}
    \| \nabla^2 u \|^2_{L^2(\Sigma)} + R_S^{-2} \| \nabla u\|^2_{L^2(\Sigma)}
    \leq
    CR_S^{-4} \|u\|^2_{L^2(\Sigma)} + CR_S \left|\int_\Sigma u
      W_\lambda u \dmu\right| .
  \end{equation*}
\end{lemma}
\begin{proof}
  From~\eqref{eq:45}, we get the following expression, after
  integration by parts of the term $u\Delta u (|A|^2 +
  \RicM(\nu,\nu))$ in $(Lu)^2$:
  \begin{equation}
    \label{eq:100}
    \begin{split}
      \int_\Sigma u W_\lambda u \dmu &= \int_\Sigma (\Delta u)^2 +
      \big( \half H^2 - \lambda -2|A|^2 -2 \RicM(\nu,\nu) \big)|\nabla
      u|^2
      \\
      &\quad + u^2 \big(-\half H^2|A|^2 - \half H^2 \RicM(\nu,\nu) -
      H\nabla_\nu\RicM(\nu,\nu)
      \\
      &\qquad\phantom{+u^2\big(}
      + \lambda |A|^2
      + |A|^4 + 2|A|^2\RicM(\nu,\nu) \big)
      \\
      &\qquad
      + a(u,\nabla u) + b u^2
      + u \nabla_ku A^{ij}\nabla^k A_{ij}\dmu .
    \end{split}
  \end{equation}
  Here $|a(u,\nabla u) + bu^2| \leq CR_S^{-4} |\nabla u|^2 + CR_S^{-6}
  u^2$, where we integrated by parts and used
  lemma~\ref{thm:geometry-in-schwarzschild}, 
  definition~\ref{def:asymptotically-flat} and
  theorem~\ref{thm:final-curv}. In particular we can estimate
  \begin{equation}
    \label{ablricci}
    \begin{split}
      \big|\nabla_{e_i}\big(\RicM(\nu,\nu)\big)\big| &\leq
      \big|\big(\nabla_{e_i}\RicM\big)(\nu,\nu)+ 2 h_{i}^k\omega_k\big|\\
      &\leq  
      \big|\big(\nabla^S_{e_i}\RicM^S\big)(\nu,\nu)\big|+ C
      \sqrt{\eta}\rmin^{-5}\\&
      \leq \big|\big(\nabla^S_{P^\perp_{\rho}(e_i)}\RicM^S\big)(\rho,\rho)\big|+ C
      \sqrt{\eta}\rmin^{-5}\\
      &\leq C \sqrt{\eta}\rmin^{-5},
    \end{split}
  \end{equation}
  where we used the above mentioned theorems, and where
  $P^\perp_{\rho}$ is the projection onto the $g^S$-orthogonal
  subspace to $\rho$. In view of the Gauss equation, the Bochner
  formula \cite[Chapter IV, Proposition
  4.15]{Gallot-Hulin-Lafontaine:1993} implies that
  \begin{equation*}
    \int_\Sigma (\Delta u)^2 \dmu
    =
    \int_\Sigma 2|(\nabla^2 u)^\circ|^2 + \big(\ScalM - 2\RicM(\nu,\nu)
    + \half H^2 - |\Acirc|^2\big) |\nabla u|^2 \dmu.
  \end{equation*}
  Together with~\eqref{eq:100} this yields
  \begin{equation*}
    \begin{split}
      \int_\Sigma u W_\lambda u \dmu
      &=
      \int_\Sigma 2|(\nabla^2 u)^\circ|^2
      + |\nabla u|^2 \big( -4\RicM(\nu,\nu) -\lambda \big)
      + u \nabla_ku A^{ij}\nabla^k A_{ij}
      \\
      &\qquad
      + u^2 \big( - \half H^2 \RicM(\nu,\nu) -
      H\nabla_\nu\RicM(\nu,\nu)+ \lambda |A|^2\\
      &\qquad+
      2|A|^2\RicM(\nu,\nu) \big)+a(u,\nabla u) + bu^2 \dmu .
    \end{split}
  \end{equation*}
  In combination with the estimate $|\RicM(\nu,\nu) + \lambda| \leq
  CR_S^{-4}$ and the fact that
  \begin{equation*}
    \begin{split}
      &- \half H^2 \RicM(\nu,\nu) -
      H\nabla_\nu\RicM(\nu,\nu)
      + \lambda |A|^2+ 
      2|A|^2\RicM(\nu,\nu)
      \\
      \qquad\qquad
      &=
      -\tfrac{3}{2} H^2 \lambda + O(R_S^{-6}) 
    \end{split}
  \end{equation*}
  we obtain the estimate
  \begin{equation}
    \label{eq:101}
    \begin{split}
      &2\| (\nabla^2 u)^\circ\|^2_{L^2} + 2\lambda \| \nabla u\|^2_{L^2}
      \\
      &\quad\leq
      CR_S^{-2}\lambda \|u\|^2_{L^2}
      +
      C\left|\int_\Sigma u W_\lambda u\dmu \right|
      + C\int_\Sigma |u||\nabla u| |A| |\nabla A|\dmu
      .
    \end{split}
  \end{equation}
  To treat the last term, observe that
  \begin{equation*}
    \begin{split}
      \int_\Sigma |u||\nabla u| |A| |\nabla A|\dmu
      &\leq
      \int_\Sigma \lambda |\nabla u|^2 + \tfrac{1}{4\lambda} |u|^2
      |A|^2 |\nabla A|^2 \dmu
      \\
      &\leq 
      \lambda \|\nabla u\|^2_{L^2} + CR_S^{-5} \|u\|_{L^{\infty}}^2      
    \end{split}
  \end{equation*}
  using theorem~\ref{thm:improved-curvature-est},
  theorem~\ref{thm:position-estimate} and $\lambda = 2m/R_S^3 +
  O(R_S^{-4})$. In particular
  \begin{equation*}
    \| \nabla u\|^2_{L^2}
    \leq
    CR_S^{-2} \|u\|^2_{L^2} + CR_S^3 \left|\int_\Sigma u W_\lambda u\dmu \right| +CR_S^{-2}\|u\|_{L^\infty}^2.
  \end{equation*}
  Note that in view of this estimate~\eqref{eq:100} implies that
  \begin{equation*}
    \|\Delta u\|_{L^2}^2
    \leq
    CR_S^{-2} \|\nabla u\|_{L^2}^2 + CR_S^{-4} \| u\|^2_{L^2} +
     C\left|\int_\Sigma u W_\lambda u\dmu \right|+ CR_S^{-5}\|u\|_{L^\infty}^2.
  \end{equation*}
  Together with~\eqref{eq:101}, we obtain that
  \begin{equation}
    \label{eq:34}
    \|\nabla^2 u\|_{L^2}^2
    +
    R_S^{-2} \|\nabla u\|^2_{L^2}
    \leq
    CR_S^{-4} \|u\|^2_{L^2} + CR_S \left|\int_\Sigma u W_\lambda u\dmu \right| + CR_S^{-4}\|u\|_{L^\infty}^2.
  \end{equation}
  From lemma~\ref{thm:interpolation} we conclude that in view of
  theorem~\ref{thm:final-curv}  
  \begin{equation*}
    \| u\|_{L^\infty}^2
    \leq
    CR_S^{-2} \|u\|_{L^2}^2 + C\|u\|_{L^2}\|\nabla^2 u\|_{L^2}.    
  \end{equation*}
  Inserting this into equation~\eqref{eq:34}, we get
  \begin{equation}
    \label{eq:47}
    \begin{split}
      &\|\nabla^2 u\|_{L^2}^2
      +
      R_S^{-2} \|\nabla u\|^2_{L^2}
      \\
      &\quad
      \leq
      CR_S^{-4} \|u\|^2_{L^2} +  CR_S \left|\int_\Sigma u W_\lambda u\dmu \right| + CR_S^{-4}\|u\|_{L^2}\|\nabla^2u\|_{L^2}. 
    \end{split}
  \end{equation}
  For large enough $R_S$, we can therefore apply the Cauchy-Schwarz inequality and
  absorb the term containing second derivatives to the left.  This
  yields the claimed estimate.
\end{proof}
With the help of the last two results we are able to show
that certain solutions of $W_\lambda u=f$ are almost constant.
\begin{lemma}
  \label{mean}
  There exists $\delta_0>0$ such that for all $0<\delta\le \delta_0$,
  all surfaces $\Sigma$ as in theorem \ref{thm:final-curv} and all
  solutions $u\in C^\infty(\Sigma)$ of $W_\lambda u=f$ with
  \begin{equation*}
    \int_\Sigma (u-u_0)f \dmu
    \le
    \delta R_S^{-6}\|u\|_{L^2(\Sigma)}
    \|u-u_0\|_{L^2}
  \end{equation*}
   we have 
  \begin{align}
    \|u-\bar u_0\|_{L^\infty(\Sigma)}
    \le
    C(\sqrt{\delta}+\eta^{1/4}+R_S^{-1})|\bar u_0|. \label{meana}
  \end{align}
\end{lemma} 
\begin{proof}
  We assume that $\|u\|_{L^2(\Sigma)}=1$ and apply corollary \ref{coro} to get
  \begin{align*}
    \|u-u_0\|_{L^2(\Sigma)}
    \le
    C(\sqrt{\delta}+\sqrt{\eta}+R_S^{-1}).
  \end{align*}
  Moreover, by lemma \ref{eigen}, we have that
  \begin{align*}
    \|u_0-\bar u_0\|_{L^\infty(\Sigma)}
    \le
    C\eta^{1/4}R_S^{-2}.
  \end{align*}
  Combining these two facts we get
  \begin{align}
    \|u-\bar u_0\|_{L^2(\Sigma)}
    &\le
    \|u-u_0\|_{L^2(\Sigma)}
    + C R_S\|u_0-\bar u_0\|_{L^\infty(\Sigma)}\nonumber
    \\
    &\le  C(\sqrt{\delta}+\eta^{1/4}+R_S^{-1}).\label{pre}
  \end{align}
  Using lemma \ref{thm:W-estimates} (with $u$ replaced by $u-u_0$)
  we get
  \begin{equation*}
    \begin{split}
      \|\nabla^2(u-u_0)\|_{L^2(\Sigma)}
      &\le
      CR_S^{-4}\|u-u_0\|_{L^2(\Sigma)}
      +cR_S\left|\int_\Sigma (u-u_0)W_\lambda (u-u_0)\dmu\right|
      \\
      &\le
      CR_S^{-4}(\delta+\sqrt{\eta}+R_S^{-2}),
    \end{split}
  \end{equation*}
  where we used corollary \ref{coro} and the assumption of the lemma. Combining this with lemma \ref{eigen} we have
  \begin{align*}
    \|\nabla^2(u-\bar u_0)\|_{L^2(\Sigma)}\le CR_S^{-4}(\delta+\sqrt{\eta}+R_S^{-2})
  \end{align*}
  and therefore, with the help of lemma \ref{thm:interpolation} and \eqref{pre}, we conclude
  \begin{align}
    \|u-\bar u_0\|_{L^\infty(\Sigma)}\le CR_S^{-1}(\sqrt{\delta}+\eta^{1/4}+R_S^{-1}).\label{pre1}
  \end{align}
  Next we note that by orthogonality 
  \begin{align*}
    0\le 1-\|u_0\|_{L^2}^2 = \|u-u_0\|_{L^2}^2
  \end{align*}
  and from theorem \ref{est-lin-willmore-3}, \eqref{trick} and the assumption of the lemma we get
  \begin{equation*}
    \begin{split}
      \|u-u_0\|_{L^2}^2
      &\le
      \frac{R_S^6}{12m^2}\int_\Sigma (u-u_0)W_\lambda(u-u_0)\dmu
      \\
      &
      \le
      C \sqrt{\delta}\|u-u_0\|_{L^2}^2+\frac{1}{3}\|u-u_0\|_{L^2}^2
      +C (\delta+\eta+R_S^{-2})
      \\
      &
      \le
      \frac{1}{2}\|u-u_0\|_{L^2}^2+C (\delta+\eta+R_S^{-2}).
    \end{split}
  \end{equation*}
  Hence for $\delta$, $\eta$ small enough and $R_S$ large enough we have
  \begin{align*}
    \|u_0\|_{L^2}^2\ge \tfrac14
  \end{align*}
  and moreover, by lemma \ref{eigen}, this implies that there exists a constant $c_1>0$ such that 
  \begin{align*}
  c_1^{-1} R_S^{-1} \le |\bar u_0|\le c_1R_S^{-1}.
  \end{align*}
  Inserting this estimate into \eqref{pre1} we get
  \begin{align*}
  \|u-\bar u_0\|_{L^\infty(\Sigma)} \le C(\sqrt{\delta}+\sqrt{\eta}+R_S^{-1})|\bar u_0|.
  \end{align*}
\end{proof}
Next we show that the above estimates yield the invertibility of the
operator $W_\lambda :C^{4,\alpha}(\Sigma)\rightarrow
C^{0,\alpha}(\Sigma)$.
\begin{theorem}\label{inve}
  There exists $\delta_0>0$ such that for every surface $\Sigma$ as in theorem \ref{thm:final-curv} the operator $W_\lambda :C^{4,\alpha}(\Sigma)\rightarrow
  C^{0,\alpha}(\Sigma)$ is invertible for every $0<\alpha<1$. Its
  inverse $W_\lambda^{-1}:C^{0,\alpha}(\Sigma)\rightarrow
  C^{4,\alpha}(\Sigma)$ exists and is continuous. Moreover it
  satisfies the estimates
  \begin{align}
    \|W_\lambda^{-1} f\|_{L^2(\Sigma)}
    &\le
    \frac{R_S^6}{\delta_0} \|f\|_{L^2(\Sigma)}
    \quad\text{for every}\quad f\in L^2(\Sigma)\ \text{and}
    \label{inve1}
    \\
    \|W_\lambda^{-1} f\|_{C^{0,\alpha}(\Sigma)}
    &\le
    \frac{cR_S^6}{\delta_0} \|f\|_{C^{4,\alpha}(\Sigma)}
    \quad\text{for every}\quad f\in C^{4,\alpha}(\Sigma).
    \label{inve2}
  \end{align}
\end{theorem}
\begin{proof}
  We argue by contradiction as in \cite{Metzger:2007ce}. Namely we
  assume that there exists a smooth function $u$ with $\|u\|_{L^2(\Sigma)}=1$ and
  \begin{align}
    \sup_{\|v\|_{L^2(\Sigma)}=1}
    \left|\int_\Sigma v W_\lambda u\dmu\right|
    \le
    \delta_0 R_S^{-6}.\label{inve3}
  \end{align}
  Choosing $v=u-u_0$, we conclude from lemma \ref{mean}
  that $\bar u_0 \not= 0$ and therefore we can assume without loss of
  generality that $\bar u_0>0$. Again from lemma \ref{mean} we then
  conclude that for $\delta_0$, $\eta$ small and $R_S$ large enough we have for
  every $x\in \Sigma$ that $\frac{\bar u_0}{2}\le u(x) \le 2\bar
  u_0$. Arguing as in the proof of lemma \ref{mean} we get $\frac12 \le \|u_0\|_{L^2(\Sigma)}\le1$ and, with the help of lemma \ref{eigen}, this implies
  \begin{align*}
    \tfrac{1}{2}|\Sigma|^{-1/2}
    \le
    |\Sigma|^{-1/2} \|u_0\|_{L^2(\Sigma)}
    \le
    \bar u_0
    \le
    |\Sigma|^{-1/2} \|u_0\|_{L^2(\Sigma)}
    \le
    |\Sigma|^{-1/2}.
  \end{align*}
  Moreover, by choosing $v=1$ in \eqref{inve3}, we get
  \begin{align}
    \left|\int_\Sigma W_\lambda u\dmu\right|
    \le
    \delta_0R_S^{-6}|\Sigma|^{1/2}
    \le
    C\delta_0 R_S^{-5}.
    \label{inve4}
  \end{align}
  On the other hand, by using \eqref{eq:45} and the corresponding
  equation for the $\lambda L$ term, we get
  \begin{equation*}
    \begin{split}      
      \int_\Sigma W_\lambda u\dmu
      &
      =
      \int_\Sigma u\Big(
      |A|^4
      +2|A|^2\RicM(\nu,\nu)
      +(\RicM(\nu,\nu))^2
      \\
      &\qquad\quad
      +\Delta(|A|^2+\RicM(\nu,\nu))
      + \lambda \big(|A|^2+\RicM(\nu,\nu)\big)
      + |\nabla H|^2
      \\
      &\qquad\quad
      + 2 \omega(\nabla H)
      + H \Delta H + 2\la \nabla^2 H, \Acirc\ra 
      + 2 H^2 |\Acirc|^2
      + 2 H\la \Acirc, T\ra
      \\
      &\qquad\quad
      - H\nabla_\nu\RicM(\nu,\nu)
      - \half H^2|A|^2
      - \half H^2\RicM(\nu,\nu)\Big)\dmu.
    \end{split}
  \end{equation*}
Now we calculate
\begin{equation*}
  \begin{split}
    &|A|^2\big(|A|^2 +2\RicM(\nu,\nu)+\lambda\big)
    - H\nabla_\nu\RicM(\nu,\nu)
    -\half H^2\big(|A|^2+\RicM(\nu,\nu)\big)
    \\  
    &\quad
    =
    \tfrac{3}{2} H^2\RicM(\nu,\nu)+O(R_S^{-6}).
  \end{split}
\end{equation*}
Moreover we estimate $\|\Delta\Acirc\|_{L^2(\Sigma)}$ as in the
proof of lemma \ref{thm:linfinityestimates}, and  using lemma \ref{thm:l2H} 
\begin{equation*}
  \begin{split}
    \left|\int_\Sigma u\Delta |A|^2 \dmu\right|
    &=
    \left|\int_\Sigma u\Delta (|\Acirc|^2+\half H^2) \dmu\right|
    \\
    &\le
    \left|\int_\Sigma u H \Delta H\dmu\right|
    + C R^{-6}_S
    \le
     C R_S^{-5}.
  \end{split}
\end{equation*}  
Now we integrate by parts and use proposition
\ref{thm:nabla-omega-small} and lemma \ref{thm:W-estimates} to conclude
\begin{equation*}
  \left|\int_\Sigma u\Delta \RicM(\nu,\nu) \dmu\right|
  \le
  C R_S^{-4}||\nabla u||_{L^2(\Sigma)}
  \le  C R_S^{-5},
\end{equation*}
where we used in the last step that $|\int_\Sigma u W_\lambda u\dmu |\le \delta_0 R_S^{-6}$, which follows from \eqref{inve3}.
We combine these estimates with the ones done previously in this section, \eqref{inve4} and theorem \ref{thm:final-curv} to conclude
  \begin{align*}
    -\int_\Sigma uH^2\RicM(\nu,\nu)\dmu
    \le
    C \left|\int_\Sigma W_\lambda u\dmu\right|
    +
     C R_S^{-5}\le CR_S^{-5}.
  \end{align*} 
  The estimates $\bar u_0 \leq 2 u$ and $\frac{1}{\bar u_0}\le 2 R_S$ imply
  \begin{align*}
    2mR_S^{-3}\int_\Sigma H^2 \dmu
    \le&
    -\int_\Sigma H^2 \RicM(\nu,\nu) \dmu
    +C R_S^{-4}\\
    \le& -
    \frac{1}{2 \bar u_0} \int_\Sigma u H^2 \RicM(\nu,\nu)\dmu
    +C R_S^{-4}\\
    \le&\ 
    C R_S^{-4}.
  \end{align*}
  This contradicts the estimate for $\int_\Sigma H^2 \dmu$ in lemma \ref{thm:initial-roundness}. Hence the operator $W_\lambda$ is
  injective. By the Fredholm alternative $W_\lambda$ is also surjective. The rest of the statements in the theorem are then a
  consequence of standard elliptic theory.
\end{proof}


%% file: foliation.tex
\section{Existence and Uniqueness of the Foliation}
\label{sec:exist-uniq-foli}
In this last section we use the implicit function theorem to prove
theorem \ref{thm:main_existence} and
theorem~\ref{thm:main_uniqueness}.
\subsection{Uniqueness in Schwarzschild}
In this subsection we show that in Schwarzschild the only surfaces
satisfying the assumptions of theorem \ref{thm:position-estimate} are
the round spheres with center at the origin.
\begin{theorem}
  For all $m>0$ there exist $r_0<\infty$, $\tau_0>0$ and $\eps>0$
  with the following properties.  \newline Assume that
  $(M,g)=(\IR^3,g^S_m)$ and let $\Sigma$ be a surface
  satisfying \eqref{eq:1} with $H>0$, $\lambda>0$, $\rmin>r_0$ and
  \begin{equation*}
    \tau \leq \tau_0\qquad\text{and}\qquad R_e \leq \eps\rmin^{2},
  \end{equation*}
  where $R_e$ and $\tau$ are as in section~\ref{sec:impr-curv-estim}.
  Then $\Sigma=S_{R_e}(0)$.
\end{theorem}
\begin{proof}
  Since $(M,g)=(\IR^3,g^S_m)$ we can apply proposition
  \ref{thm:sphere-approx-1}, theorem \ref{thm:position-estimate} and
  theorem \ref{thm:final-curv} with $\eta=0$ to get $\tau = 0$,
  $\Acirc^S=0$, and $\lambda = \frac{2m}{R_S^3}$. Since $\Acirc^S=0$,
  we get that $\Sigma$ is umbilical with respect to the Euclidean
  background metric, as $\Acirc^S = \phi^{-2}\Acirc^e$. Hence $\Sigma$
  is a sphere. Since $\tau=0$ in fact $\Sigma= S_{R_e}(0)$ where $R_e
  = \phi^{-2}R_S$, or otherwise the expression for $\lambda$ could not
  be true.
\end{proof}

\subsection{Existence and uniqueness for the general case}
The main goal in this subsection is to show that for any manifold
which is $(m,\eta,\sigma)$-asymptotically Schwarzschild and all small
enough Lagrange multipliers $\lambda$ there exists a unique surface
$\Sigma_\lambda$ which solves the equation \eqref{eq:1}. More
precisely we have the following theorem.
\begin{theorem}\label{existence}
  For all $m>0$ and $\sigma$ there exist $\eta_0>0$, $\lambda_0>0$ and
  $C$ depending only on $m$ and $\sigma$ with the following
  properties.

  If $(M,g)$ is $(m,\eta,\sigma)$-asymptotically Schwarzschild and
  satisfies
  \begin{itemize}
  \item[(1)] $|\ScalM|\le \eta r^{-5}$ and
  \item[(2)] $\eta\le \eta_0$
  \end{itemize}
  then for all $0<\lambda<\lambda_0$ there exists a surface
  $\Sigma_\lambda$ which solves \eqref{eq:1} for the given $\lambda$.
  Moreover the surface is well approximated in the $C^3$-norm by a
  coordinate sphere $S_{r_\lambda}(a_\lambda)$ with $|a_\lambda|\le
  C$.
\end{theorem} 
\begin{proof}
  We define $g_\tau=(1-\tau)g^S+\tau g$ and we note that $(M,g_\tau)$
  is $(m,\eta,\sigma)$-asymptotically Schwarzschild. For
  $(M,g^S)$ a standard calculation shows that all spheres $S_{r}(0)$
  centered at the origin solve equation \eqref{eq:1} with
  \begin{align*}
    \lambda(r)=\frac{2m}{r^{3}}\Big(1+\frac{m}{2r}\Big)^{-6}.
  \end{align*}
  This function is invertible for $r$ large enough. Moreover this
  shows that we can solve equation \eqref{eq:1} in $(M,g^S)$ for
  any $\lambda$ small enough. More precisely, for any small $\lambda$
  there exists a radius $r(\lambda)$ such that $S_{r(\lambda)}(0)$
  solves \eqref{eq:1} with the given $\lambda$. Next we want to use
  the implicit function theorem to get the existence of a family of
  such solutions for all $0\le \tau \le 1$.

  In order to do this we consider the following conditions on our surfaces
  \begin{itemize}
  \item[(A1)] $H>0$,
  \item[(A2)] $\tau\le \tau_0$ and
  \item[(A3)] $R_e\le \eps\rmin^{2}$,
  \end{itemize}
  where $\tau_0$ and $\varepsilon$ are chosen such that we can apply
  the results from section~\ref{sec:position-estimates}. From these
  results we then get that the above conditions hold with better
  constants on surfaces $\Sigma$ with $\rmin>r_0$
  \begin{itemize}
  \item[(B1)] $|H-2R_S^{-1}+(1+\frac{m}{2R_e})2mR_S^{-2}|\le C\sqrt{\eta}\rmin^{-3}$,
  \item[(B2)] $\tau\le C\sqrt{\eta}\rmin^{-1}$ and
  \item[(B3)] $C^{-1}\rmin\le R_e\le C \rmin$.
  \end{itemize}  
  Without loss of generality we can furthermore assume that the
  conditions (B1)-(B3) imply that the linearized operator $W_\lambda$
  is invertible. From \eqref{eq:80} we also get that $\Sigma$ is
  globally a graph over $S^2$.
  
  Now we define the sets
  \begin{align*}
    S_1(\tau)=&\,\{ \Sigma| \ \ \rmin>r_0 \ \ \text{and}\ \ (A1)-(A3)\ \ \text{hold w.r.t.} \ \ g_\tau\}\\
    S_2(\tau)=&\,\{ \Sigma| \ \ \rmin>2r_0 \ \ \text{and}\ \ (B1)-(B3)\ \ \text{hold w.r.t.} \ \ g_\tau\}.
  \end{align*}
  We choose $\lambda_2$ so small that the centered spheres $S_r(0)$
  which solve \eqref{eq:1} with $0<\lambda<\lambda_2$ are in
  $S_2(\tau)$. Finally (for $\lambda_1$ small) we let
  \begin{align*}
    \kappa:[0,1]\rightarrow&\, (0,\lambda_1)\times [0,1]\\
    \kappa(t)=&\,(\lambda(t),\tau(t))
  \end{align*}
  be a continuous, piecewise smooth curve with $\tau(0)=0$ and we define
  \begin{align*}
    I_\kappa=\{t\in[0,1]| \exists \ \ \Sigma(t)\in S_2(\tau(t)) \ \
    \text{satisfying}\ \ \eqref{eq:1}\ \ \text{with}\ \
    \lambda=\lambda(t)\}.
  \end{align*}
  As in \cite{Metzger:2007ce} we can show that $I_\kappa$ is open and
  closed and since moreover $0\in I_\kappa$ by our assumption we get
  $I_\kappa=[0,1]$ and this finishes the proof of the Theorem.
\end{proof} 
By reversing the process used in the above theorem as in the proof of
theorem $6.5$ in \cite{Metzger:2007ce} we furthermore get a uniqueness
result for solutions of \eqref{eq:1}.
\begin{theorem}\label{unique}
  Let $m>0$ and $\sigma$ be given. Then there exist $\eta_0>0$,
  $\tau_0$, $r_0<\infty$, and $\eps>0$ depending only on $m$ and $\sigma$
  such that the following holds.

  Assume that $(M,g)$ is $(m,\sigma,\eta)$-asymptotically
  Schwarzschild with 
  \begin{itemize}
  \item[(1)] $|\ScalM|\leq \eta r^{-5}$, and 
  \item[(2)] $\eta < \eta_0$.
  \end{itemize}
  Furthermore, let $\Sigma$ be a surface with approximating sphere
  $S_{r_\lambda}(a_\lambda)$ as in section~\ref{sec:impr-curv-estim},
  such that
  \begin{itemize}
  \item[(3)] $\Sigma$ satisfies equation~\eqref{eq:1},
  \item[(4)] $H>0$,
  \item[(5)] $\rmin > r_0$, and $r_\lambda < \eps\rmin^{2}$, 
  \item[(6)] $\tau_\lambda = r_\lambda/a_\lambda < \tau_0$,
  \end{itemize}
  then $\Sigma = \Sigma_\lambda$, where $\Sigma_\lambda$ is the
  surface from theorem~\ref{existence}.
\end{theorem}
\qed

\subsection{Foliation}
Next we show that the surfaces obtained in theorem \ref{existence} form a foliation.
\begin{theorem}\label{foliation}
  For all $m>0$ and $\sigma$ there exists $\eta_0>0$ depending only on
  $m$ and $\sigma$ with the following properties.

  If $(M,g)$ is $(m,\eta,\sigma)$-asymptotically Schwarzschild and
  satisfies
  \begin{itemize}
  \item[(1)] $|\ScalM|\le \eta r^{-5}$ and
  \item[(2)] $\eta\le \eta_0$
  \end{itemize}
  then for all $0<\lambda<\lambda_0$ the surfaces $\Sigma_\lambda$
  constructed in theorem \ref{existence} form a foliation. In
  addition, there is a differentiable map
  \[
  F:S^2\times (0,\lambda_0)\times [0,1] \rightarrow M
  \]
  such that the surfaces $F(S^2,\lambda,\tau)$ satisfy \eqref{eq:1}
  with respect to the metric $g_\tau=(1-\tau)g^S+\tau g$ for the given
  $\lambda$. This foliation can therefore be obtained by deforming a
  piece of the foliation of $(\IR^3,g^S)$ by centered spheres.
\end{theorem}
\begin{proof}
  The proof follows along the same lines as the one given in
  \cite[Theorem 6.4]{Metzger:2007ce}. Therefore we only sketch the main
  ideas of the argument.

  For $0 < \lambda < \lambda_0$ we consider the curve
  $\kappa_\lambda(t)=(\lambda,t)$ and by using theorem \ref{existence}
  we obtain a family of surfaces $\Sigma_{\lambda,t}$ which solve
  \eqref{eq:1} for the given $\lambda$.

  The map $F$ can now be defined by
  $F(S^2,\lambda,t)=\Sigma_{\lambda,t}$ where we can choose the
  parametrization of $\Sigma_{\lambda,t}$ such that $\dd{F}{\lambda}
  \perp \Sigma_{\lambda,t}$. The differentiability of $F$ with respect
  to $p\in S^2$ and $\tau$ follows from the construction of
  $\Sigma_{\lambda,t}$.

  It remains to prove that the surfaces form a foliation. In order to
  show this we fix $\lambda_1 \in (0,\lambda_0)$ and we get from the
  above construction a surface $\Sigma_{\lambda_1,1}$. For
  $\lambda_2<\lambda_1$ we define the curve
  $h_{\lambda_2}(t)=((1-t)\lambda_1+t\lambda_2,1)$. By combining the
  curves $\kappa_{\lambda_1}$ and $h_{\lambda_2}$ we get a family of
  surfaces $\Sigma_{\lambda(t),1}'$ which solve \eqref{eq:1} with
  $\lambda(t)=(1-t)\lambda_1+t\lambda_2$ for $t\in [0,1]$. Moreover we
  get a differentiable map $G:S^2\times
  [\lambda_2,\lambda_1]\rightarrow M$ such that
  $G(S^2,\lambda(t))=\Sigma_{\lambda(t),1}'$. From the local
  uniqueness statement in the implicit function theorem we get that
  $\Sigma_{\lambda(t),1}'=\Sigma_{\lambda(t),1}=:\Sigma_{\lambda(t)}$.

  Now we let $\nu_{\lambda(t)}$ be the normal to $\Sigma_{\lambda(t)}$
  in $M$ and we let
  $\alpha_{\lambda(t)}=g(\nu_{\lambda(t)},\frac{\partial G}{\partial \lambda})$. We
  calculate
  \begin{equation*}
    \begin{split}
      H(\lambda_1-\lambda_2)
      &=
      \frac{d}{dt}\big(-\Delta
      H-H|\Acirc|^2-H\RicM(\nu,\nu)\big)-\lambda(t)
      \frac{d}{dt}H
      \\
      &=
      W_{\lambda(t)}\alpha_{\lambda(t)}(\lambda_1-\lambda_2).
    \end{split}
  \end{equation*}
  Next we claim that 
  \begin{align}
    \int_\Sigma (\alpha_{\lambda(t)}-(\alpha_{\lambda(t)})_0)H\dmu
    \le
    \frac{C\eta^{1/4}}{R_S^{6}} \|\alpha_{\lambda(t)}\|_{L^2(\Sigma)}\|\alpha_{\lambda(t)}-(\alpha_0)_{\lambda(t)}\|_{L^2(\Sigma)}.\label{claim}
  \end{align}
  If we assume that this claim is true we see that for $\eta$ small
  enough we can apply lemma \ref{mean} and get that
  $\alpha_{\lambda(t)}$ does not change sign. Therefore the family
  $\Sigma_{\lambda(t)}$ is a foliation.
  
  In order to prove \eqref{claim} we let $W_{\lambda(t)}=W_\lambda$,
  $\alpha=\alpha_{\lambda(t)}$ and we note that we can argue as in the
  proof of theorem \ref{inve} to get
  \begin{align*}
    \left|\int_\Sigma W_\lambda \alpha \dmu\right|
    &\le CR_S^{-6}\int_\Sigma |\alpha|\dmu
    +\frac32\left|\int_\Sigma\alpha H^2 \RicM(\nu,\nu) \dmu\right|
    \\
    & \quad
    +C\left| \int_\Sigma\alpha \Delta(|A|^2+\RicM(\nu,\nu))\dmu\right|.
  \end{align*}
  Using theorem \ref{thm:final-curv} we get
  \begin{align*}
    \frac32\left|\int_\Sigma\alpha H^2 \RicM(\nu,\nu) \dmu\right|
    &\le
    12mR_S^{-5}\int_\Sigma |\alpha|\dmu+CR_S^{-6}\int_\Sigma |\alpha|\dmu.
  \end{align*}
  Moreover, using integration by parts, theorem \ref{thm:final-curv},
  lemma \ref{thm:W-estimates} and \eqref{ablricci} we estimate
  \begin{align*}
    \left|\int_\Sigma\alpha \Delta(|A|^2+\RicM(\nu,\nu))\dmu\right|
    &\le
    C\sqrt{\eta}R_S^{-4}\|\nabla \alpha\|_{L^2(\Sigma)}
    \\
    &\le
    C\sqrt{\eta}(R_S^{-5}\|\alpha\|_{L^2(\Sigma)}+R_S^{-5/2}\|\alpha\|^{1/2}_{L^2(\Sigma)}).
  \end{align*}
  Putting these estimates together we conclude
  \begin{align}
    \left|\int_\Sigma W_\lambda \alpha \dmu\right|
    &\le
    C_1 R_S^{-4}\|\alpha\|_{L^2(\Sigma)}\nonumber 
    \\
    & \quad
    +C(R_S^{-5}\|\alpha\|_{L^2(\Sigma)}+R_S^{-5/2}\|\alpha\|^{1/2}_{L^2(\Sigma)}).\label{est1}
  \end{align}
  On the other hand we have, using again theorem \ref{thm:final-curv},
  \begin{align}
    \left|\int_\Sigma W_\lambda \alpha \dmu\right|
    &=
    \left|\int_\Sigma H\dmu\right|\nonumber
    \\
    &\ge
    C_2 R_S-CR_S^{-1}. \label{est2}
  \end{align}
  Combining the two estimates we get
  \begin{align*}
    C_2 R_S^5-CR_S^{3}
    &\le
    C_1\|\alpha\|_{L^2(\Sigma)}+C(R_S^{-1}\|\alpha\|_{L^2(\Sigma)}+R_S^{3/2}\|\alpha\|^{1/2}_{L^2(\Sigma)}).
  \end{align*}
  From this estimate we easily see that there exists a constant $C_3>0$ such that for $R_S$ large enough we have
  \begin{align}
    \|\alpha\|_{L^2(\Sigma)}\ge C_3 R_S^5. \label{alp}
  \end{align}

  Using H\"older's inequality we get
  \begin{align*}
    \int_\Sigma (\alpha-\alpha_0)H\dmu\le \|H-H_0\|_{L^2(\Sigma)}\|\alpha-\alpha_0\|_{L^2(\Sigma)}
  \end{align*}
  and hence, combing this with \eqref{alp}, we see that \eqref{claim} will be a consequence of the estimate
  \begin{align}
    \|H-H_0\|_{L^2(\Sigma)}\le C \eta^{1/4}R_S^{-1}.\label{claima}
  \end{align}
 
  We note that
  \begin{align*}
    LH=(\lambda-\frac12 H^2)H=\mu_0 H+(\lambda-\frac12 H^2-\mu_0)H
  \end{align*}
  and therefore we can estimate
  \begin{align*}
    \mu_0 &\le \frac{\int_\Sigma HLH\dmu}{\int_\Sigma H^2\dmu}\le \mu_0+C\sqrt{\eta}\rmin^{-4},
  \end{align*}
  where the first inequality follows from the Rayleigh quotient
  characterization of the eigenvalues of $L$ and the second inequality
  follows from the above estimate and lemma \ref{mu0}. Next we
  decompose $H=\sum_i \langle H,\varphi_i \rangle \varphi_i$ and we
  calculate
  \begin{align*}
    \frac{\int_\Sigma HLH\dmu}{\int_\Sigma H^2\dmu}
    &=
    \frac{\sum_i \mu_i\int_\Sigma H_i^2\dmu}{\int_\Sigma H^2\dmu}
    \\
    &=
    \mu_0 +\frac{\sum_i (\mu_i-\mu_0)\int_\Sigma H_i^2\dmu}{\int_\Sigma H^2\dmu}.
  \end{align*}
  Hence we get
  \begin{align*}
    0\le \sum_{i=1}^\infty (\mu_i-\mu_0)\int_\Sigma H_i^2\dmu\le C\sqrt{\eta}\rmin^{-4}.
  \end{align*}
  For every $i\in \mathbb{N}$ we have $(\mu_i-\mu_0)\ge 2R_S^{-2}$ (see corollary \ref{thm:cor-ev-jac-op}) and therefore
  \begin{align*}
    \|H-H_0\|^2_{L^2(\Sigma)}\le CR_S^2 \sum_{i=1}^\infty (\mu_i-\mu_0)\int_\Sigma H_i^2\dmu\le C\sqrt{\eta}\rmin^{-2},
  \end{align*}
  which finishes the proof of \eqref{claima} and therewith also the proof of the theorem.
\end{proof}


%% file: appendix.tex
\section{Maple scripts for the calculations}
For the explicit calculations in the proof of Proposition
\ref{thm:compute_lambda}, in section~\ref{sec:right-hand-side} and in
section~\ref{sec:variation-V} we used Maple~\cite{Maple} to evaluate
certain integrals. Here we present the scripts we used.
\subsection{Proposition~\ref{thm:compute_lambda}}
Here it is necessary to evaluate the integral
\begin{equation}
  \label{eq:50}
  \begin{split}
    E_1:=\int_S \left(
      \frac{1}{r^3}
      - 3R_e^2 \frac{1}{r^5}
      - 6 R_e |a_e| \frac{\cos\varphi}{r^5}
      - 3 |a_e|^2 \frac{\cos^2\varphi}{r^5}
    \right)
    \dmu^e
  \end{split}
\end{equation}
where $S = S_{R_e}(a_e)$ is a fixed sphere with center $a$ and radius
$R_e$. The calculation is based on the formula
\begin{equation*}
  C_k^l
  :=
  \int_S \frac{\cos^l\varphi}{r^k} \dmu^e
  =
  \frac{2\pi R_e}{|a_e|} (2R_e|a_e|)^{-l}
  \int_{|R_e-|a_e||}^{R_e+|a_e|}
  r^{1-k} (r^2 -R_e^2 - |a_e|^2)^l dr.
\end{equation*}
which was derived in the proof of
proposition~\ref{thm:compute_lambda}. Hence equation~\eqref{eq:50} can be written as
\begin{equation*}
  E_1 = C_3^0 - 3R_e^2 C_5^0 - 6R_e|a_e|C_5^1 -3|a_e|^2 C_5^2.
\end{equation*}
This is evaluated using the following Maple script.
{\small
\begin{verbatim}
assume (R>0, a>0, R>a);
c0r3 := 2*PI*R/a *(2*R*a)^(0)
        * int(r^(-2)*(r^2 - R^2 - a^2 )^(0),r=R-a..R+a);
c0r5 := 2*PI*R/a *(2*R*a)^(0)
        * int(r^(-4)*(r^2 - R^2 - a^2)^(0),r=R-a..R+a);
c1r5 := 2*PI*R/a *(2*R*a)^(-1)
        * int(r^(-4)*(r^2 - R^2 - a^2 )^(1), r=R-a..R+a);
c2r5 := 2*PI*R/a *(2*R*a)^(-2)
        * int(r^(-4)*(r^2 - R^2 - a^2 )^(2), r=R-a..R+a);
E1   := c0r3 -3*R^2*c0r5 - 6*R*a*c1r5 - 3*a^2*c2r5;
simplify(%);
\end{verbatim}
}
where we used \texttt{R} to denote $R_e$, \texttt{a} to denote
$|a_e|$ and \texttt{c}$l$\texttt{r}$k$ to denote $C_k^l$.

\subsection{Section~\ref{sec:right-hand-side}}
In section~\ref{sec:right-hand-side} the integral to evaluate was
\begin{equation*}
  E_2 := \int_S \frac{\cos\phi}{r^3} = C_3^1
\end{equation*}
This is evaluated by the script
{\small
\begin{verbatim}
assume (R>0, a>0, R>a);
c1r3 := 2*PI*R/a *(2*R*a)^(-1)
        * int(r^(-2)*(r^2 - R^2 - a^2 )^(1),r=R-a..R+a);
E2   := c1r3;
simplify(%);
\end{verbatim}
}

\subsection{Section~\ref{sec:variation-V}}
The longest calculation is for the term 
\begin{equation*}
  \begin{split}
    \bar Q
    &:=
    \int_{S} \Big(
    R_e\tfrac{\cos\varphi}{r^6}
    + |a_e| \tfrac{\cos^2\varphi}{r^6}
    - |a_e|R_e^2 \tfrac{1}{r^8}
    - (R_e^3 + 2 |a_e|^2R_e)\tfrac{\cos\varphi}{r^8}
    \\
    &\phantom{\int_{S} \Big(}\qquad
    - (|a_e|^3 + 2|a_e|R_e^2) \tfrac{\cos^2\varphi}{r^8}
    - |a_e|^2 R_e \tfrac{\cos^3\varphi}{r^8} \Big) \dmu^e   
  \end{split}
\end{equation*}
from section~\ref{sec:variation-V}, where we omit certain fixed
factors here. The following script evaluates this expression.
{\small
\begin{verbatim}
assume (R>0, a>0, R>a);
c1r6:=2*PI*R/a *(2*R*a)^(-1)
      * int(r^(-5)*(r^2 - R^2 - a^2)^(1),r=R-a..R+a);
c2r6:=2*PI*R/a *(2*R*a)^(-2)
      * int(r^(-5)*(r^2 - R^2 - a^2 )^(2),r=R-a..R+a);
c0r8:=2*PI*R/a *(2*R*a)^(0)
      * int(r^(-7)*(r^2 - R^2 - a^2 )^(0),r=R-a..R+a);
c1r8:=2*PI*R/a *(2*R*a)^(-1)
      * int(r^(-7)*(r^2 - R^2 - a^2 )^(1),r=R-a..R+a);
c2r8:=2*PI*R/a *(2*R*a)^(-2)
      * int(r^(-7)*(r^2 - R^2 - a^2 )^(2),r=R-a..R+a);
c3r8:=2*PI*R/a *(2*R*a)^(-3)
      * int(r^(-7)*(r^2 - R^2 - a^2 )^(3),r=R-a..R+a);
Q   := R * c1r6 + a * c2r6 - a*R^2*c0r8 - (R^3 + 2*a^2*R)*c1r8
       - (2*a*R^2 + a^3) *c2r8 - a^2 * R * c3r8);
subs(a = tau * R, Q);
simplify(%);
\end{verbatim}
